\documentclass[table]{amsart}

\makeatletter
\newcommand{\mylabel}[2]{#2\def\@currentlabel{#2}\label{#1}}
\makeatother
\usepackage{amssymb,stmaryrd,mathrsfs}
\usepackage{comment}
\usepackage{float}
\usepackage{longtable} % for 'longtable' environment
\usepackage{pdflscape} % for 'landscape' environment
\usepackage{orcidlink}
\usepackage{booktabs}
\usepackage{hyperref}
\usepackage{mathbbol}
\usepackage{mathtools}
\definecolor{vegasgold}{rgb}{0.77, 0.7, 0.35}
\definecolor{darkgoldenrod}{rgb}{0.72, 0.53, 0.04}
\definecolor{gold(metallic)}{rgb}{0.83, 0.69, 0.22}
\hypersetup{
 colorlinks=true,
 linkcolor=darkgoldenrod,
 filecolor=brown,      
 urlcolor=gold(metallic),
 citecolor=darkgoldenrod,
 pdftitle={Diophantine approximation and the subspace theorem},
 }

\usepackage[all,cmtip]{xy}

\usepackage[margin=1.2 5in]{geometry}

\usepackage{tikz}
\usetikzlibrary{shapes.geometric}
\tikzset{every loop/.style={min distance=10mm,looseness=10}}

\DeclareFontFamily{U}{wncy}{}
\DeclareFontShape{U}{wncy}{m}{n}{<->wncyr10}{}
\DeclareSymbolFont{mcy}{U}{wncy}{m}{n}
\DeclareMathSymbol{\Sh}{\mathord}{mcy}{"58}
\usepackage[T2A,T1]{fontenc}
\usepackage[OT2,T1]{fontenc}

\newtheorem{theorem}{Theorem}[section]
\newtheorem{lemma}[theorem]{Lemma}

\newtheorem{proposition}[theorem]{Proposition}
\newtheorem{corollary}[theorem]{Corollary}
\newtheorem{definition}[theorem]{Definition}
\numberwithin{equation}{section}

\theoremstyle{remark}
\newtheorem{remark}[theorem]{Remark}

\newcommand{\op}[1]{\operatorname{#1}}

\newcommand{\ind}{\operatorname{Ind}}
\newcommand{\bbX}{\mathbf{X}}

\newcommand{\ord}{\mathrm{ord}}

\newcommand{\bx}{\mathbf{x}}
\newcommand{\Z}{\mathbb{Z}}

\newcommand{\Q}{\mathbb{Q}}

\newcommand{\cO}{\mathcal{O}}

\begin{document}
\title[Diophantine approximation and the subspace theorem]{Diophantine approximation and the subspace theorem}

\author[S.~Goel]{Shivani Goel\, \orcidlink{0009-0000-7841-2011}}
\address[Goel]{Chennai Mathematical Institute, H1, SIPCOT IT Park, Kelambakkam, Siruseri, Tamil Nadu 603103, India}
\email{shivanig@cmi.ac.in}

\author[R.~Lunia]{Rashi Lunia\, \orcidlink{0000-0002-2074-7668}}
\address[Lunia]{Max-Planck-Institut F\"ur Mathematik, Vivatsgasse 7, D-53111 Bonn, Germany.}
\email{lunia@mpim-bonn.mpg.de}

\author[A.~Ray]{Anwesh Ray\, \orcidlink{0000-0001-6946-1559}}
\address[Ray]{Chennai Mathematical Institute, H1, SIPCOT IT Park, Kelambakkam, Siruseri, Tamil Nadu 603103, India}
\email{anwesh@cmi.ac.in}

\begin{abstract}
Diophantine approximation explores how well irrational numbers can be approximated by rationals, with foundational results by Dirichlet, Hurwitz, and Liouville culminating in Roth’s theorem. Schmidt’s subspace theorem extends Roth’s result to higher dimensions, with profound implications for Diophantine equations and transcendence theory. This article provides a self-contained and accessible exposition of Roth’s theorem and Schlickewei’s refinement of Schmidt’s subspace theorem, with an emphasis on proofs. The arguments presented are classical and approachable for readers with a background in algebraic number theory, serving as a streamlined, yet condensed reference for these fundamental results.
\end{abstract}


\maketitle

\section{Introduction}
\label{section:intro}

\par Diophantine approximation originates from the study of how well irrational numbers can be \emph{approximated} by rational numbers. The first result in this direction was due to Dirichlet, who showed that if $\alpha\in \mathbb{R}$ is an irrational number, then there are infinitely many rational numbers $\frac{p}{q}$ (with $p$, $q$ coprime) such that $|\alpha-\frac{p}{q}|<\frac{1}{q^2}$. Note that this result is clearly false when $\alpha\in \Q$. Dirichlet's result motivates further questions about how well an irrational number may be approximated by rationals. Indeed, one would like to know if better approximations than those given by Dirichlet's theorem do exist. Hurwitz improved upon Dirichlet's result, by showing that the bound $\frac{1}{q^2}$ can be improved to $\frac{1}{\sqrt{5}q^2}$. The constant $\frac{1}{\sqrt{5}}$ is optimal in the sense that replacing $\frac{1}{\sqrt{5}}$ with a smaller constant makes the result false for certain quadratic irrationals $\alpha$. Liouville \cite{Liouville} was the first to provide a significant result in this direction by proving that algebraic numbers of degree greater than 1 cannot be approximated "too well" by rationals. Specifically, he showed that if \(\alpha\) is an algebraic number of degree \(d \geq 2\), then there exists a constant \(C > 0\) such that \(|\alpha - \frac{p}{q}| > \frac{C}{q^d}\) for all rational numbers \(\frac{p}{q}\). However, Liouville's bound is not sharp for algebraic numbers of degree greater than 2, and this led to the eventual development of sharper results due to Thue \cite{Thue}, Siegel \cite{Siegel}, Dyson \cite{Dyson} and Gelfond \cite{Gelfond}, finally culminating in Roth's theorem.  

\par Roth's theorem states that for any \(\varepsilon > 0\), there are only finitely many rational approximations \(\frac{p}{q}\) such that \(|\alpha - \frac{p}{q}| < \frac{1}{q^{2+\varepsilon}}\). Unlike Liouville's theorem, this does not depend on the degree of $\alpha$. Lang \cite{Lang} and Leveque \cite{LeVeque} generalized Roth’s theorem to include approximations involving valuations at places in any number field $K$. In earlier work Ridout \cite{Ridout} had obtained such results over $\Q$. Schmidt \cite{Schmidtnote, Schmidtbook, Schmidtsubspacemainref} further extended the ideas and introduced the \emph{subspace theorem}, which generalizes Roth’s result to higher-dimensional settings. The subspace theorem states that if a set of linear forms satisfies certain non-degeneracy conditions, then the solutions to an inequality involving these forms lie in a finite union of proper subspaces. Schmidt's subspace theorem has had profound implications in Diophantine equations, transcendence theory, and the distribution of rational and integral points on varieties. For further details, we refer to \cite{CorvajaZanniercompositio,BomGubHeightsinDiophantineG, CorvajaZannierbook}. An extension of Schmidt’s theorem allowing for completions over a finite set of places, entirely analogous to Ridout’s and Lang’s generalizations of Roth’s theorem, was later obtained by Schlickewei \cite{Schlickewei}. Faltings \cite{Faltingsmainpaper} subsequently introduced powerful geometric techniques that complemented and extended Schmidt's work, leading to a significant generalization now known as the product theorem. For some of the details pertaining to the proof, the reader may also refer to also the article by van der Put \cite{Vanderput}.  

Building on these developments, Faltings and W\"ustholz \cite{FaltingsWustholz} established a deep and far-reaching extension of the subspace theorem which applies to inequalities governed by forms of arbitrary degree. The central finiteness result in this generalized context asserts that the set of solutions to such inequalities is not merely sparse but is, in fact, confined to a proper algebraic subvariety of the ambient space. This represents a significant strengthening of earlier results, as it shifts the conclusion from a mere quantitative bound on the number of solutions to a strong structural constraint on their distribution. 
\par Quantitative versions of the subspace theorem were obtained by Evertse \cite{Evertse}, and Evertse and Schlickewei \cite{EvertseSchlickeweiQuant}. In these works one is able to control the number of hyperplanes containing the solutions to the subspace equation. Further refinements of these results have been obtained by Evertse and Ferreti \cite{EvertseFerretti}.

This article provides a self-contained yet concise reference for the proofs of Roth’s theorem and Schmidt’s subspace theorem, following the treatments of Bombieri and Gubler \cite[Chapter 6 and 7]{BomGubHeightsinDiophantineG}, Hindry and Silverman \cite[Part D]{HindrySilverman}  and Schmidt \cite{Schmidtbook}. Section \ref{s 2} develops the necessary preliminaries, beginning with absolute values on number fields, Ostrowski’s theorem, places of a number field, and the product formula. We then introduce height functions, establishing key properties of the Weil height and proving fundamental results such as Northcott’s and Kronecker’s theorems. The discussion extends to polynomial heights, where we prove Gauss’s lemma and derive essential height inequalities. We further examine Mahler’s measure, various formulations of Siegel’s lemma, and the properties of the index associated with polynomial functions in multiple variables. Finally, we introduce the generalized Wronskian and present a criterion for establishing the linear independence of polynomials in several variables.

In Section \ref{s 3}, we present a proof of Roth's theorem. To motivate the result, we first introduce and prove Liouville's theorem, which states that for any real algebraic number \(\alpha\) of degree \(n > 1\), there exists a constant \(c(\alpha) > 0\) such that for every rational number \(p/q \in \mathbb{Q}\), we have \(\left|\alpha - \frac{p}{q}\right| > \frac{c(\alpha)}{q^n}\). We then formulate two versions of Roth's theorem for a number field \(K\) with a finite set of places \(S\) and demonstrate their equivalence. Next, we prove \emph{Roth's lemma}, which shows that the index of a polynomial vanishing at a distinguished point $\bar{\beta}=(\beta_1, \dots, \beta_m)$ defined over $K$ is bounded. This point is constructed so that each coordinate $\beta_i$ is close to $\alpha$. Roth's lemma is proved by induction on the number of variables, beginning with the base case for polynomials in one variable and extending to higher dimensions by considering the decomposition of the polynomial into sums of products of lower-dimensional polynomials. On the other hand, we apply Siegel's lemma to construct an auxiliary polynomial with controlled height in a large number of variables that vanishes to a high order at $\bar{\alpha}$. The construction requires that $\bar{\beta}$ is sufficiently close to $\bar{\alpha}$, and thus one obtains a lower bound on the index at $\bar{\beta}$.
%the point \((\alpha, \alpha, \dots, \alpha)\). 
The contradiction is then achieved, completing the proof of Roth's theorem.
\par Finally, Section \ref{s 4} is devoted to proving Schlickewei's refinement of Schmidt's subspace theorem over a number field \( K \). Denote by $K_{\mathbb{A}}$ the ring of adeles over $K$. We begin by formulating a version of Minkowski's second theorem adapted to adelic domains contained in $K_{\mathbb{A}}^N$ for a suitable large number $N$. These adelic domains are associated to solutions of the subspace equation and are known as \emph{approximation domains}. Assuming for contradiction that the subspace theorem is false, we construct an infinite family of solutions to the subspace inequality and associate to them approximation domains with adelic volumes that decrease asymptotically. The rank of an approximation domain is defined as the dimension \( k \) of the Euclidean space over \( K \) that embeds into it. The proof proceeds by induction on \( k \), establishing that, outside of finitely many hyperplanes, the adelic approximation domains corresponding to solutions of the subspace inequality have dimension at most \( k \). The induction step relies on the construction of an auxiliary polynomial with bounded height and prescribed vanishing properties, and a generalized form of Roth’s lemma is employed to obtain crucial nonvanishing results.

\subsection*{Acknowledgments}
This article originated from a series of lectures delivered during the two-week workshop \emph{"The subspace theorem and its applications"}, held at the Chennai Mathematical Institute from December 16 to December 28, 2024. An exposition of Roth’s theorem and the subspace theorem were presented by Sinnou David and the third author, respectively in this workshop. The authors express their sincere gratitude to the speakers—Sinnou David, Anup Dixit, Ram Murty, Kumar Murty, Siddhi Pathak and Purusottam Rath—as well as to the organizers, Anup Dixit and Purusottam Rath, for orchestrating this enriching and stimulating experience. We also extend our thanks to the Chennai Mathematical Institute for fostering an environment conducive to intellectual engagement. The second author would like to thank the Max-Planck-Institut f\"ur Mathematik for providing a friendly atmosphere for research. The third author would like to thank Sinnou David for inspiring conversations and Purusottam Rath for introducing him to the broader subject of diophantine approximation. We thank the anonymous referee for the thorough and timely report.

\section{Preliminary notions}\label{s 2}
We begin this section with a brief review of absolute values on a number field \( K \).

\begin{definition}
An \emph{absolute value} (or \emph{norm}) on a field \( K \) is a real-valued function \( ||\cdot|| : K \to \mathbb{R} \) satisfying the following properties for all \( x, y \in K \):  
\begin{itemize}
    \item \( ||x|| \geq 0 \), with equality if and only if \( x = 0 \).  
    \item \( ||x y|| = ||x|| \cdot ||y|| \).  
    \item \( ||x + y|| \leq ||x|| + ||y|| \).  
\end{itemize}  
\end{definition}
An absolute value is called \emph{non-archimedean} if it satisfies the stronger inequality  
\[
||x + y|| \leq \max\{ ||x||, ||y|| \} \quad \text{for all } x, y \in K.
\]  
Otherwise, it is called \emph{archimedean}.  
Every absolute value \( ||\cdot|| \) defines a metric on \( K \) via the distance function \( d(x, y) = ||x - y|| \), thereby endowing \( K \) with a corresponding metric topology. Two absolute values \( ||\cdot||_1 \) and \( ||\cdot||_2 \) on \( K \) are said to be equivalent if they induce the same topology on \( K \).  

Let \( \mathcal{O}_K \) denote the ring of integers of \( K \). The following classification theorem, due to Ostrowski, describes all inequivalent absolute values on \( K \).

\begin{theorem}[Ostrowski's Theorem]
Any non-trivial absolute value on \( K \) is equivalent to either  
\begin{itemize}
    \item a \( \mathfrak{p} \)-adic absolute value for a unique prime \( \mathfrak{p} \) in \( \mathcal{O}_K \), or  
    \item an absolute value arising from a real or complex embedding of \( K \).  
\end{itemize}
\end{theorem}

 \begin{proof}
     See \cite[Theorem 3.3]{Narkiewicz}.
 \end{proof}

For \( K = \mathbb{Q} \), the non-trivial absolute values, up to equivalence, consist of the usual (archimedean) absolute value and the \( p \)-adic absolute values for distinct primes \( p \).

We now define the normalized valuations of a number field \( K \). Let \( \mathcal{O}_K \) denote the ring of integers of \( K \), and for any prime ideal \( \mathfrak{p} \subset \mathcal{O}_K \), let \( N_{K/\mathbb{Q}}(\mathfrak{p}) \) denote the absolute norm of \( \mathfrak{p} \). For \( \alpha \in K \), the normalized absolute value at \( \mathfrak{p} \) is given by  
\[
|\alpha|_{\mathfrak{p}} = N_{K/\mathbb{Q}}(\mathfrak{p})^{-\ord_{\mathfrak{p}}(\alpha)},
\]
%|\alpha|_{\mathfrak{p}} = N_{K/\mathbb{Q}}(\mathfrak{p})^{-\ord_{\mathfrak{p}}(\alpha)/[K:\Q]},\] 
where \( \ord_{\mathfrak{p}}(\alpha) \) denotes the order of \( \alpha \) at \( \mathfrak{p} \). For an archimedean absolute value corresponding to an embedding \( \sigma: K \hookrightarrow \mathbb{C} \), we define\[
|\alpha|_{\sigma} = |\sigma(\alpha)|,
\]
%Liouville's name added
%\[|\alpha|_{\sigma} = |\sigma(\alpha)|^{1/[K:\Q]},\]  
where \( |\cdot| \) denotes the usual absolute value on \( \mathbb{R} \) or \( \mathbb{C} \) depending on whether \( \sigma \) is real or complex, respectively.

In this manner, we have selected a representative absolute value for each inequivalent class of non-trivial absolute values of \( K \). An equivalence class of non-trivial absolute values is referred to as a \emph{place} \( v \) of \( K \).
Let $$
|\alpha|_{v}=|\alpha|_{\mathfrak{p}}^{1/[K:\Q]}
$$
when $v$ corresponds to the equivalence class of a non-zero prime $\mathfrak{p} \subset \mathcal{O}_K$ and 
$$
|\alpha|_{v}= |\alpha|_{\sigma}^{d_{v}/[K:\Q]}
$$
where $d_{v}=2$ when $v$ is an equivalence class of a complex embedding $\sigma$ of $K$ and $d_{v}=1$ otherwise.

Let $M_K$ be the set of all such places on $K$. Then $\alpha \in M_K \backslash \{0\}$ satisfies the product formula
\begin{equation}\label{product_formula}
\prod_{v \in M_K}|\alpha|_{v}=1.
\end{equation}
 We denote by $M_K^{\infty}$ (resp. $M_K^{f}$) the set of archimedean (resp. non-archimedean) places $v \in M_K$.

\subsection{Heights}
The \emph{natural height} $H$ is defined on the set of non-zero rationals as
$$
H(p/q)= \max\{|p|, |q|\}.
$$
We say that a set $\mathcal{A}$ has the \emph{Northcott property} with respect to a height $H$ if 
for any real $X > 0$, there exist only finitely many elements $a \in \mathcal{A}$ with $H(a) \le X$.
The set of rational numbers has the Northcott property with respect to the natural height. 

The notion of heights can be extended to non-zero algebraic numbers as follows.
\begin{definition}[Weil height]
    Let $\alpha \in \mathbb{\bar{Q}}\backslash \{0\}$. Further, let $K$ be a number field containing $\alpha$. The Weil height of $\alpha$ is defined as 
$$
H(\alpha)=\prod_{v \in M_K}\max\{1, |\alpha|_{v}\}.
$$
%where 
%\begin{equation}\label{norm}
%|\alpha|_{v} = \begin{cases}
    %|\sigma(\alpha)|^{d_{v}/[K: \mathbb{Q}]} & \mathrm{for~} v \in M_K^{\infty} \\
     %(N_{K/\mathbb{Q}}(\mathfrak{p}))^{-\ord_{\mathfrak{p}}(\alpha)/[K:\mathbb{Q}]} & \mathrm{for~} v \in M_K^f\text{ arising from a prime ideal }\mathfrak{p},
    % \end{cases}
%end{equation} 
%and $d_{v}=2 $ when $v$ is an equivalence class of complex embeddings of $K$ and $d_{v}=1$ otherwise.
\end{definition} 

\begin{lemma} For $\alpha \in \mathbb{\bar{Q}} \backslash \{0\}$,
\begin{itemize}
    \item $H(\alpha)$ is well-defined.
    \item $H(\alpha^{-1})= H(\alpha)$.
    \item If $\alpha$ and $\beta$ are conjugates, then $H(\alpha)=H(\beta)$.
    \item For any $m \in \mathbb{Z}$, $H(\alpha^m)= H(\alpha)^{|m|}$.
\end{itemize}
\end{lemma}
\begin{proof}
    The proof is an easy consequence of the definition of height, we refer to \cite[chapter 3 sec 3.2]{Waldschmidtbook} for further details.
\end{proof}

\begin{remark}
    The set $\mathbb{\bar{Q}}$ does not have Northcott property with respect to the Weil height. For example, for any $n \in \mathbb{N}$, $H(2^{1/n})=2^{1/n}$ and thus the set of all $\alpha \in \mathbb{\bar{Q}}$ such that $H(\alpha) <2$ is an infinite set.
\end{remark}

\begin{theorem}[Northcott]  
There are finitely many \( \alpha \in \mathbb{\bar{Q}} \) with bounded height and bounded degree. In other words, for any \( M > 0 \), the set \( \{ \alpha \in \mathbb{\bar{Q}} \mid [\mathbb{Q}(\alpha) : \mathbb{Q}] < M \} \) has the Northcott property with respect to the Weil height.  
\end{theorem}

\begin{proof}  
Let \( X > 0 \) and \( \alpha \in \mathbb{\bar{Q}} \setminus \{0\} \) such that \( H(\alpha) \leq X \). Let \( f(x) = a_n x^n + \ldots + a_0 \in \mathbb{Z}[x] \) be the minimal polynomial of \( \alpha \), and let \( \alpha = \alpha_1, \ldots, \alpha_n \) be its roots. Writing 
\[f(x)=a_n\prod_{i=1}^n(x-\alpha_i)\] we see that 
\[|a_i|=|a_n|\left|\sum_{1 \le j_1 < \dots < j_i \le n} \alpha_{j_1}\dots \alpha_{j_i}\right|\leq |a_n|\binom{n}{i} \op{max}\{|\alpha_1|, \dots, |\alpha_n|\}^i.\]
Noting that for each $i$,
\[|\alpha_i|\leq \prod_{v\in M_K} \op{max}\{1, |\alpha_i|_v\}\leq H(\alpha_i)=H(\alpha),\] we deduce that
\[  
|a_i| \leq |a_n| \binom{n}{i}H(\alpha)^i  \leq X 2^n H(\alpha)^n \leq 2^n X^{n+1}.  
\] 
%\max_j |\alpha_j|^i
Since there are finitely many polynomials of degree \( n \) with coefficients bounded by \( 2^n X^{n+1} \), there are finitely many such algebraic numbers.  
\end{proof}
The Kronecker theorem provides a remarkable characterization of algebraic numbers with minimal possible height. Specifically, it asserts that the only algebraic numbers with Weil height exactly equal to \( 1 \) are the roots of unity.

\begin{theorem}[Kronecker]  
For \( \alpha \in \mathbb{\bar{Q}} \), \( H(\alpha) = 1 \) if and only if \( \alpha \) is a root of unity.  
\end{theorem}

\begin{proof}  
First, suppose \( \alpha \) is a root of unity. Then, there exists a positive integer \( n \) such that \( \alpha^n = 1 \). Since the Weil height satisfies the multiplicative property \( H(\alpha^n) = H(\alpha)^n \), we have  
\[  
H(1) = 1 = H(\alpha^n) = H(\alpha)^n.  
\]  
Taking the \( n \)-th root of both sides yields \( H(\alpha) = 1 \).
\par Conversely, suppose \( H(\alpha) = 1 \). Then $|\alpha|_v\le1$ for every $v\in M_K$. %Consider the set \( S = \{ \alpha^n \mid n \geq 1 \} \) of all positive integer powers of \( \alpha \). Since \( H(\alpha^n) = H(\alpha)^n = 1 \) for all \( n \), the height of every element in \( S \) is bounded by 1. 
Let $d$ be the degree of \( \alpha \) and let $\alpha_1, \dots , \alpha_d$ be an ordering on the Galois conjugates of $\alpha$ over $\Q$ and set $\boldsymbol{\alpha}:=(\alpha_1,\ldots,\alpha_d)$. For $1 \le i \le d$, let \[s_i(x_1, \dots, x_d):=\sum_{1\leq j_1<j_2<\dots<j_i\leq d} x_{j_1}x_{j_2}\dots x_{j_i}\] and for every positive integer $m$, set \[s_i(\boldsymbol{\alpha}^m):=s_i(\alpha_1^m, \dots, \alpha_d^m).\] Since $|\alpha|_v\le1$ for all non-archimedean  
places $v$, $\alpha$ is an algebraic integer. We deduce that $s_i(\boldsymbol{\alpha}^m)\in \Z$ for every $m>0$. Since $|\alpha|_v\leq 1$, for all $i$,
\[\sum_{i=1}^d|s_i(\boldsymbol{\alpha}^m)|\leq \sum_{i=1}^d \binom{d}{i}=2^d.\]In other words, the vector $(s_1(\boldsymbol{\alpha}^m), \dots, s_d(\boldsymbol{\alpha}^m))$ lies in the box $[-2^d,2^d]^d$. By the Pigeon-hole principle, there exists $m\ne n$ such that \[(s_1(\boldsymbol{\alpha}^m), \dots, s_d(\boldsymbol{\alpha}^m))=(s_1(\boldsymbol{\alpha}^n), \dots ,s_d(\boldsymbol{\alpha}^n)).\]This implies that 
\[s_j(\boldsymbol{\alpha}^m)=(\boldsymbol{\alpha}^n) \ \text{for all }\ j=1,\dots, d.\]  
Thus, there is a permutation $\pi\in \op{Aut}\left(\{1, \dots, d\}\right)$ such that
\[\boldsymbol{\alpha}^m=\pi(\boldsymbol{\alpha}^n).\]
Iterating this relation $k$ times, one finds that 
\[\boldsymbol{\alpha}^{m^k}=\pi^k(\boldsymbol{\alpha}^{n^k}).\]
Choose $k\geq 1$ such that $\pi^k=\text{id}$, to deduce that $\boldsymbol{\alpha}^{m^k}=\boldsymbol{\alpha}^{n^k}$. In particular,
$\alpha^{m^k} = \alpha^{n^k}$. 
 Rearranging this equation, we find that \( \alpha^{m^k-n^k} = 1 \), hence \( \alpha \) is a root of unity. This completes the proof. 
\end{proof}

To facilitate certain estimates and inequalities, it is often convenient to work with the logarithm of the Weil height.

\begin{definition}  
The \emph{logarithmic Weil height} of a non-zero algebraic number \( \alpha \) is defined as  
\[  
h(\alpha) = \log H(\alpha).  
\]  
\end{definition}
 
The logarithmic Weil height satisfies several useful inequalities, which make it particularly handy in many applications.
\begin{proposition}  
Let \( \alpha_1, \ldots, \alpha_n \in K \setminus \{0\} \), where \( K \) is a number field. Then,  
\begin{itemize}  
    \item \( h(\alpha_1 \cdots \alpha_n) \leq \sum_{i=1}^n h(\alpha_i) \).  
    \item \( h(\alpha_1 + \cdots + \alpha_n) \leq \sum_{i=1}^n h(\alpha_i) + \log n \).  
\end{itemize}  
\end{proposition}

\begin{proof}The first inequality follows from the multiplicative property of the Weil height, as \( H(\alpha_1 \cdots \alpha_n) \leq H(\alpha_1) \cdots H(\alpha_n) \), which translates into an additive inequality when taking logarithms:  
\[  
h(\alpha_1 \cdots \alpha_n) = \log H(\alpha_1 \cdots \alpha_n) \leq \log \left( H(\alpha_1) \cdots H(\alpha_n) \right) = \sum_{i=1}^n \log H(\alpha_i) = \sum_{i=1}^n h(\alpha_i).  
\]
The second inequality follows by using triangle inequality for archimedean places, strong triangle inequality for non-archimedean places and product formula \eqref{product_formula}.
For further details, see \cite[Chapter 1]{BomGubHeightsinDiophantineG} or \cite[Section 3.2]{Waldschmidtbook}.

\end{proof}

The height of points in affine space plays an important role in Diophantine geometry, where it provides a quantitative way to measure the complexity of algebraic points. Let \( \mathbb{A}_K^n \) denote the affine \( n \)-space over \( K \). For a point \( \mathbf{x} = (x_1, \ldots, x_n) \in \mathbb{A}_K^n \setminus \{\mathbf{0}\} \), its height is defined as  
\[
H(\mathbf{x}) := \prod_{v \in M_K} \max_j \{ 1, |x_j|_v \},  
\]  
where the maximum is taken over the coordinates of \( \mathbf{x} \), and \( |x_j|_v \) denotes the \( v \)-adic absolute value of the coordinate \( x_j \) at the place \( v \). The logarithmic height, often more convenient due to its additive properties, is defined as follows:
\[
h(\mathbf{x}) := \sum_{v \in M_K} \max_j \log^+ |x_j|_v,  
\]  
where \( \log^+ t = \max \{ \log t, 0 \} \) for \( t > 0 \) and $\log^+0=0$.

Let \( f(x) = a_n x^n + \ldots + a_0 \) be a polynomial with coefficients in a number field \( K \). The height of the polynomial \( f \) is a measure of the arithmetic complexity of its coefficients and is defined as  
\[H(f):=\prod_{v\in M_K} \op{max}\{1, |f|_v\}\]
and the logarithmic height of $f$ is defined as 
\[
h(f):= \log H(f)= \sum_{v \in M_K} \log^+|f|_v.
\]  
The local absolute value of \( f \) at a place \( v \) is given by  
\[
|f|_v = \max \{ |a_0|_v, \ldots, |a_n|_v \}.
\] 
\noindent In other words, we view a polynomial as a tuple of its coefficients in a suitable affine space and its height coincides with the affine height of that point.

Gauss’s lemma asserts that for non-archimedean places, the height behaves multiplicatively under the product of two polynomials.
\begin{lemma}[Gauss' lemma] \label{GL} 
For polynomials \( f, g \in K[x] \) and a place \( v \in M_K \setminus M_K^{\infty} \) (i.e., a non-archimedean place), we have  
\[
|fg|_v = |f|_v |g|_v.  
\]  
\end{lemma}

\begin{proof}  
Let \( f(x) = \sum_{i=0}^{m} a_i x^i \) and \( g(x) = \sum_{i=0}^{n} b_i x^i \), and denote their product by \( fg(x) = \sum_{k=0}^{m+n} c_k x^k \) where  $c_k=\sum_{i+j=k}a_ib_j$.
Since \( v \) is non-archimedean, we know that the non-archimedean triangle inequality gives  
\[
|fg|_v \leq |f|_v |g|_v.  
\]  
Without loss of generality, assume \( |f|_v = 1 \) and \( |g|_v = 1 \). Suppose, for the sake of contradiction, that \( |fg|_v < 1 \). This would imply \( |c_j|_v < 1 \) for all \( j \). Let \( j_0 \) be the smallest index for which \( |a_{j_0}|_v = 1 \). Since \( |f|_v = 1 \), such a \( j_0 \) must exist. Now, consider the terms in the product \( fg \) involving \( a_{j_0} \). Since  \( |c_{j_0}|_v < 1 \) and \( |a_{k}|_v<1 \)  for all \( k<j_0 \), hence $|b_0|_v<1$. If we repeat the process for $c_{j_o+l}$, we get $|b_{l}|_v<1$. Applying a similar argument recursively to the coefficients of \( g \), we deduce that \( |b_j|_v < 1 \) for all \( j \), contradicting the assumption that \( |g|_v = 1 \). Hence, we conclude \( |fg|_v = 1 \), proving the lemma.  
\end{proof}
The notion of height also extends to multivariable polynomials. Let 
\[f(x_1, \ldots, x_m) = \sum_{\mathbf{i} = (i_1, \ldots, i_m)} 
a_{\mathbf{i}}\, x_1^{i_1}\cdots x_m^{i_m}\]
be a polynomial with coefficients in a number field \( K \). 
The \emph{height} of \( f \) is defined as follows:  
\[
H(f) := \prod_{v \in M_K} \op{max}\{1, |f|_v\},
\]
and the \emph{logarithmic height} of \( f \) is defined as  
\[
h(f) := \log H(f) = \sum_{v \in M_K} \log^+ |f|_v,
\]
where
\[
|f|_v := \max_{\mathbf{i}} \{|a_{\mathbf{i}}|_v\}.
\]

\begin{lemma}\label{ht of product poly}
    For polynomials $f_1, \ldots ,f_r \in K[x_1, \ldots, x_m]$ with $\deg_{x_j}f_i \le d_j$ for all $1 \le i \le r$ and $1 \le j \le m$, we have 
    $$
  h(f_1\ldots f_r) \le \sum_{i=1}^rh(f_i)+ (d_1 + \ldots + d_m )\log 2 + m(r-1) \log 2.
    $$ 
\end{lemma}
\begin{proof}
 Let $\boldsymbol{\mu}=(\mu_1,\ldots,\mu_m)\in \Z_{\geq 0}^m$ be such that $\mu_j\leq d_j$ for $j=1, \dots, m$. For $\mathbf{x}=(x_1, \dots, x_m)$, we write $\mathbf{x}^{\boldsymbol{\mu}}:=x_1^{\mu_1}\cdots x_m^{\mu_m}$ and $f_i(\mathbf{x})=\sum_{\boldsymbol{\mu}}a_{i,\boldsymbol{\mu}}\mathbf{x}^{\boldsymbol{\mu}}$. Setting \[F(\mathbf{x}):=f_1(\mathbf{x})\cdots f_r(\mathbf{x}),\] we have that
\[F(\mathbf{x})=\sum_{\boldsymbol{\mu}}\left(\sum_{\boldsymbol{\gamma}_1+\ldots+\boldsymbol{\gamma}_r=\boldsymbol{\mu}}a_{1,\boldsymbol{\gamma}_1}\cdots a_{r,\boldsymbol{\gamma}_r}\right)\mathbf{x}^{\boldsymbol{\mu}}.\]
  For $v\in M_K$, we have that
  \[|F|_v=\max_{\boldsymbol{\mu}}\left|\sum_{\boldsymbol{\gamma}_1+\ldots+\boldsymbol{\gamma}_r=\boldsymbol{\mu}}a_{1,\boldsymbol{\gamma}_1}\cdots a_{r,\boldsymbol{\gamma}_r}\right|_v.\]
  Let $N_{\boldsymbol{\mu}}$ be the number of ways in which $\boldsymbol{\mu}$ can be written as a sum $\boldsymbol{\gamma}_1+\dots+\boldsymbol{\gamma}_r=\boldsymbol{\mu}$, and set $N:=\op{max}_{\boldsymbol{\mu}} N_{\boldsymbol{\mu}}$. We find that  \[N_{\boldsymbol{\mu}}=\prod_{j=1}^m\binom{\mu_j+r-1}{\mu_j}\]and therefore,\begin{equation}\label{number of terms}
      N\le \max_{\boldsymbol{\mu}}\prod_{j=1}^m\binom{\mu_j+r-1}{\mu_j}\le \max_{\boldsymbol{\mu}}\prod_{j=1}^m2^{\mu_j+r-1}\le 2^{d_1+\cdots+d_m+m(r-1)}.\end{equation}
  Therefore, by triangle inequality,
  \begin{align*}
      |F|_v&\le N\max_{\boldsymbol{\mu}}\left(\max_{\boldsymbol{\gamma}_1+\ldots+\boldsymbol{\gamma}_r=\boldsymbol{\mu}}|a_{1,\boldsymbol{\gamma}_1}\cdots a_{r,\boldsymbol{\gamma}_r}|_v\right)\\
      & \le N\prod_{i=1}^r\max_{\boldsymbol{\gamma}_i}\{1,|a_{i,\boldsymbol{\gamma}_i}|_v\}\\& \le N\prod_{i=1}^r\max\{1,|f_i|_v\}.
  \end{align*}
 Therefore, we deduce that
  \begin{align*}
      H(f_1\cdots f_r
      )&=\prod_{v\in M_K}\max\{1,|f_1\cdots f_r|_v\}
      \\& \le \prod_{v\in M_K}\left\{N\prod_{i=1}^r\max\{1,|f_i|_v\}\right\}\\&
      \le N \prod_{i=1}^r H(f_i).
  \end{align*}
  Taking logarithms of both sides, one finds that:
  $$
  h(f_1\ldots f_r) \le \sum_{i=1}^rh(f_i)+ \log N \le \sum_{i=1}^rh(f_i)+ (d_1 + \ldots + d_m )\log 2 + m(r-1) \log 2,
    $$ 
    \noindent thus proving the result.
\end{proof}

\begin{remark}\label{height of poly}
    If $f(x_1, \ldots, x_m)$ and $g(y_1, \ldots, y_n)$ are polynomials in different set of variables, then 
    $$
    h(fg)= h(f)+h(g).
    $$
\end{remark}

\begin{proposition}
    For $\mathbf{x}_1, \ldots , \mathbf{x}_m \in \mathbb{A}_K^n\backslash \{\mathbf{0}\}$, we have
    $$
    h(\mathbf{x}_1+ \ldots +\mathbf{x}_m)\le \sum_{i=1}^mh(\mathbf{x}_i)+\log m.
    $$
\end{proposition}

\begin{proof}
    Let $\mathbf{x}_i = (x_{i1}, \ldots, x_{in})$ for $ 1\le i \le m$. Then we have
    $$
    h(\mathbf{x}_1+ \ldots +\mathbf{x}_m) = \sum_{v \in M_K} \max_j \log^+ |x_{1j} + \ldots + x_{mj}|_{v}.
    $$
    If $v$ is non-archimedean, then 
    $$
    |x_{1j} + \ldots + x_{mj}|_{v} \le \max_{i}|x_{ij}|_v.
    $$
    If $v$ is archimedean, then using triangle inequality, we have 
    $$
    |x_{1j} + \ldots + x_{mj}|_{v} \le |m|_v\max_{i}|x_{ij}|_v.
    $$
    Using the fact that 
    $$
    \prod_{v \in M_K^{\infty}}|m|_{v}=m,
    $$ 
    we get 
 \begin{align*}  
  h(\mathbf{x}_1+ \ldots+ \mathbf{x}_m) 
& \le \log m + \sum_{v \in M_K} \max_{i,j} \log^+ |x_{ij}|_v \\
& \le \log m + \sum_{v \in M_K} \sum_{i=1}^n \max_{j} \log^+ |x_{ij}|_v \\
& = \log m + \sum_{i=1}^mh(\mathbf{x}_i)
  \end{align*}
    which proves the result.
\end{proof}

\begin{proposition}\label{propn sum of f1+...+fr}
    Given non-zero polynomials $f_1, \ldots ,f_r \in \Z[x_1, \ldots, x_m]$, we have that
    $$
  h(f_1+\cdots +f_r) \le \max_{1\le i\le r}h(f_i)+ \log r.
    $$
\end{proposition}

\begin{proof}
 Since the coefficients of $f_i$ are integers, for any non-archimedean $v$, we have 
   \[\max\{1,|f_1+\ldots+f_r|_v\}=\max\{1,|f_1|_v\}=\cdots=\max\{1,|f_r|_v\}=1.\]
   this implies that
   \begin{align*}
       H(f_1+\cdots+f_r)&=\prod_{v\in M_\Q}\max\{1,|f_1+\ldots+f_r|_v\}\\&
       \le r\prod_{v\in M_\Q}\left(\max_{1\le i\le r}\{1,|f_i|_v\}\right)\\&\le r\max_{1\le i\le r}\{|f_i|_\infty\} \\ & \le r\max_{1\le i\le r}H(f_i).
   \end{align*}
  Hence, we have  the required result.
\end{proof}
\begin{proposition}[Fundamental inequality]\label{Fundamental inequality}
    Let $S$ be a finite subset of $M_K$. For $\alpha \in K \backslash \{0\}$,
    $$
    -h(\alpha) \le \sum_{v \in S}\log|\alpha|_{v} \le h(\alpha).
    $$
\end{proposition}
\begin{proof}
    The right inequality follows by definition while the left inequality additionally uses the fact that $h(\alpha^{-1})=h(\alpha)$.
\end{proof}
Let $\mathbb{P}^n_{K}$ denote the $n$-dimensional projective space with coordinates in $K$.
The height of a point $\mathbf{x} =(x_0: \ldots : x_n) \in \mathbb{P}_K^n$ is defined as
$$
h(\mathbf{x})=\sum_{v \in M_K}\max_j \log|x_j|_{v}.
$$
\begin{remark}
    The definition of height of $\mathbf{x} \in \mathbb{P}_K^n$ is independent of the choice of coordinates. This can be shown by using the product formula \eqref{product_formula}. Moreover, under the usual embedding
  $\mathbb{A}_K^n \hookrightarrow  \mathbb{P}^{n+1}_{K}$ as
       $$
       \mathbf{x}=(x_1, \ldots, x_n) \mapsto (1: x_1: \ldots :x_n)
       $$
we note that $h(\mathbf{x})$ is equal to the height of image of $\mathbf{x}$ under the above embedding. 
   \end{remark}
\subsection{Mahler measure}
Let $\alpha \in \mathbb{\bar{Q}} \backslash \{0\}$ and $f(x) \in \mathbb{Z}[x]$ be its minimal polynomial. Let
$$
f(x)= a_nx^n+a_{n-1}x^{n-1}+ \ldots +a_0 = a_n(x-\alpha_1)\cdots(x-\alpha_n)
$$
where $\alpha_i$'s are the roots of $f$. Then the {\emph {Mahler measure}} of $\alpha$ is defined as
$$
M(\alpha) := |a_n|\prod_{j=1}^n\max\{1, |\alpha_j|\}.
$$
\begin{proposition} 
Let $\alpha \in \mathbb{\bar{Q}}\backslash \{0\}$ and $f(x) \in \mathbb{Z}[x]$ be its minimal polynomial of degree $n$. Then we have 
    $$
    H(\alpha)=M(\alpha)^{1/n}.
    $$
\end{proposition}

\begin{proof}
Let $f(x)=a_nx^n+\cdots +a_0$ be the minimal polynomial of $\alpha$ over $\Z$. In other words, $f(x)$ is a non-zero irreducible polynomial with coefficients in $\Z$ such that $f(\alpha)=0$, and moreover, $\op{gcd}(a_n, a_{n-1}, \dots, a_0)=1$. Without loss of generality, we choose a number field $K$ such that $\alpha \in K$ and $K$ is a Galois extension over $\mathbb{Q}$ with Galois group $G$.
    We have 
    $$
    h(\alpha)= \log H(\alpha) = \sum_{v \in M_K} \log^{+}|\alpha|_{v}
    $$
    where $\log^{+}x=\max\{0, \log x\}$ for $x>0$ and $\log^+0=0$. Factor $f(x)$ into a product of linear polynomials 
\[f(x)=a_n \prod_{i=1}^n(x-\alpha_i).\] Given a nonarchimedian valuation $v\in M_K$, we have that $|f|_v=1$ since the coefficients $a_i$ of $f$ are rational integers and $\op{gcd}(\{a_i|\,i=1, \dots, n\})=1$. By Gauss' lemma \ref{GL},
\[1=|f|_v=|a_n|_v\prod_{i=1}^n |(x-\alpha_i)|_v=|a_n|_v\prod_{i=1}^n \op{max}\{1, |\alpha_i|_v\}.\] 
The conjugates of $\alpha$, say $\alpha_j$ $(j=1, \ldots, n)$ appear exactly $[K:\mathbb{Q}]/n$ many times in the list $(\sigma(\alpha))_{\sigma \in G}$. Therefore, we deduce that
  \begin{equation}\label{sigma in G prod = 1}\begin{split}
  |a_n|_v\prod_{\sigma\in G}\max(1,|\sigma(\alpha)|_v)^{n/[K:\mathbb{Q}]}=1\end{split}\end{equation}for any nonarchimedian valuation $v\in M_K$.
  By the product formula, we find that
 \begin{eqnarray*}
 h(\alpha)
 & = &  \sum_{v \in M_K^{\infty}} \log^{+}|\alpha|_{v} + \sum_{v \in M_K \backslash M_K^{\infty}} \log^{+}|\alpha|_{v} \\
& = & \frac{1}{[K:\mathbb{Q}]}\left(\frac{[K:\mathbb{Q}]}{n}\sum_{j=1}^n\log^{+}{|\alpha_j|} 
+ \frac{[K:\mathbb{Q}]}{n} \log{|a_n|} \right)\\ 
& = & \frac{1}{n}\log\left(|a_n|\prod_j\max\{1, |\alpha_j|\}\right) \\
& = & \frac{1}{n}\log M(\alpha),
 \end{eqnarray*}
 where we have invoked \eqref{sigma in G prod = 1} in the second equality. This proves the result.
\end{proof}
\subsection{Siegel's lemma}
Consider the system of homogeneous linear equations
\begin{eqnarray}\label{system}
&a_{11}x_1 + a_{12}x_2 + \ldots + a_{1n}x_n = 0 \nonumber\\
&\vdots  \\
&a_{m1}x_1 + a_{m2}x_2 + \ldots + a_{mn}x_n = 0 \nonumber
\end{eqnarray}
with $a_{ij} \in \mathbb{Z}$, not all zero. If $m < n$, then the system of linear equations \eqref{system} always has a non-trivial solution in $\mathbb{Z}$.

\begin{lemma}[Siegel's lemma]\label{Siegel1}
Consider a system of $M$ linear equations in $N$ variables \eqref{system} with $M<N$ and coefficients $a_{ij} \in \mathbb{Z}$, not all zero. 
%Suppose that the coefficients satisfy $|a_{ij}| \le A$ for a natural number $A$ and for all $i,j$. 
Then there is a non-trivial solution 
$(x_1, \ldots, x_N)$ of \eqref{system} in $\mathbb{Z}^N$ with 
$$
|x_i| < (NA)^{M/(N-M)},
$$
where $A=\max_{ij}|a_{ij}|$.
\end{lemma}

\begin{proof}
For any $H>0$, the number of points $\mathbf{x} = (x_1, \ldots, x_N)\in \mathbb{Z}^N$ satisfying $0 \le x_i \le H$ is at most $(H+1)^N$. Let $T: \mathbb{R}^N \to \mathbb{R}^M$ be a linear map with 
    $$
    T(x_1 \ldots, x_N)= (a_{11}x_1 + \ldots a_{1N}x_N, \ldots ,a_{M1}x_1 + a_{M2}x_2 + \ldots + a_{MN}x_N). 
    $$
For any real $a$, we set $a^+=\max\{a, 0\}$ and $a^-=\max\{-a, 0\}$. Then $
a=a^+-a^-$ and $|a|=a^++a^-$.
We also set $L_j^+ = \sum_{i=1}^Na_{ij}^+$,  $L_j^- = \sum_{i=1}^Na_{ij}^-$ and $L_j = \sum_{i=1}^N|a_{ij} |$.
For any integer $H>0$, writing $T(\mathbf{x})= \mathbf{y}= (y_1, \ldots, y_M)$ we see that
 whenever  $0 \le x_i \le H$ for all $i$, 
 $$
 -L_j^- H \le y_j \le L_j^+H
 $$ 
 for all $j$.
The number of integers in the box $\prod_{j=1}^M[-L_j^-H, L_j^+H]$ is at most
$$
\prod_{j=1}^M(L_j^+H+L_j^-H + 1)=\prod_{j=1}^M(L_jH + 1).
$$
If $(H+1)^N > \prod_{j=1}^M(L_jH+1)$
then by pigeonhole principle, there exists $\mathbf{x}'\neq  \mathbf{x}'' \in \mathbb{Z}^N$ with $0 \le x_i', x_i'' \le H$ for all $i$ such that $T(\mathbf{x}')=T(\mathbf{x}'')$.
Then $\mathbf{x}=\mathbf{x'}- \mathbf{x}'' \in \ker{T}$ with $|x_i| = |x_i'-x_i''| \le H$.
We choose $H$ to be the natural number
$$
H=\left[(NA)^{M/N-M}\right]
$$
where $A= \max_{ij}|a_{ij}|$ and $[t]$ denotes the greatest integer part of $t$.
Then using the fact that $L_j \le NA$ for all $j$, we see that
$$
(H+1)^N > (H+1)^M (NA)^M
\ge (NAH+1)^M  \ge \prod_{j=1}^M(L_jH + 1).
$$
Thus $\mathbf{x}= (x_1, \ldots, x_N)$ is a non-trivial solution to the system of linear equations \eqref{system} satisfying 
$$
|x_i| < (NA)^{M/(N-M)}.
$$
\end{proof}

\iffalse
\begin{proof}
    For any $H>0$, the number of points $\mathbf{x} = (x_1, \ldots, x_N)\in \mathbb{Z}^N$ satisfying $|x_i| \le H$ is at most $(2H+1)^N$. Let $T: \mathbb{R}^N \to \mathbb{R}^M$ be a linear map with 
    $$
    T(x_1 \ldots, x_N)= (a_{11}x_1 + \ldots a_{1N}x_N, \ldots ,a_{M1}x_1 + a_{M2}x_2 + \ldots + a_{MN}x_N). 
    $$
    Writing $T(\mathbf{x})= \mathbf{y}= (y_1, \ldots, y_M)$, since $|a_{ij}| \le A$ we see that $|y_i| \le NHA.$
    Now the number of points $\mathbf{y} \in \mathbb{Z}^M$ with $|y_i| \le NHA$ is at most $(2NHA + 1)^m$.
    If $(2H+1)^N > (2NHA + 1)^M$, then by pigeonhole principle, there exists $\mathbf{x}'\neq  \mathbf{x}'' \in \mathbb{Z}^N$ with $|x_i'|, |x_i''| \le H$ such that $T(\mathbf{x}')=T(\mathbf{x}'')$.
    Then $\mathbf{x}=\mathbf{x'}- \mathbf{x}'' \in \ker{T}$ with $|x_i| = |x_i'-x_i''| \le 2 H$.
    Choose $H$ to be a natural number satisfying 
    $$
    (NA)^{M/(N-M)} -1 < 2 H < (NA)^{M/(N-M)} + 1.
    $$
    Then 
    $$ (2 H + 1)^N > (2 H + 1)^M (NA)^M > (2 H N A + 1)^M.
    $$
    Thus $\mathbf{x}= (x_1, \ldots, x_N)$ is a non-trivial solution to the system of linear equations \eqref{system} satisfying $|x_i| < 1+(NA)^{M/(N-M)}$.
\end{proof}
\fi
Next we give a generalization of Siegel's lemma for number fields in the form that will be used later.
\begin{lemma}[Siegel's lemma for number fields]\label{Siegel lemma}
 Let $K$ be a number field of degree $d$ over $\mathbb{Q}$ and $M,N$ be positive integers such that $N>dM$. Consider a homogeneous system of $M$ linear equations in $N$ variables \eqref{system} with coefficients $a_{ij} \in K$ not all zero. Then, there exists a non-trivial solution 
$\mathbf{x} \in \mathbb{Z}^N$ of this system of linear equations such that 
$$
h(\mathbf{x}) \le \frac{dM}{N-dM}\left(h(B)+\log N+ c(K)\right), 
$$
where $B= (\ldots, a_{ij}, \ldots)$, the vector formed by all the coefficients and $c(K)$ is a positive constant depending only on $K$.
 \end{lemma}
 \begin{proof}
      Let $\omega_1, \ldots, \omega_d$ be a $\mathbb{Z}$-basis of $\mathcal{O}_K$. Then we can write 
    \begin{equation}\label{lin}
    a_{ij}= \sum_{k=1}^d a_{ij}^{(k)}\omega_k, \phantom{mm} a_{ij}^{(k)} \in \mathbb{Z}.
    \end{equation}
     Now for any $1 \le i \le M$,
    \begin{equation*}
        \sum_{j=1}^N a_{ij}x_j=0 \phantom{mm} {\rm ~if ~ and ~ only ~ if ~}
    \phantom{mm} 
    \sum_{j=1}^N\sum_{k=1}^d a_{ij}^{(k)}x_j\omega_k =0.
\end{equation*}
   Since $\omega_1, \ldots, \omega_d$ is a $\mathbb{Z}$-basis, we conclude that for any $1 \le i \le M$, 
    \begin{equation}\label{sys2}
    \sum_{j=1}^N a_{ij}x_j=0 \phantom{mm} {\rm ~if ~ and ~ only ~ if ~}
    \phantom{mm} 
    \sum_{j=1}^N a_{ij}^{(k)}x_j =0, \phantom{mm} {\rm ~for~ all ~}  1 \le k \le d.
    \end{equation}
Thus we have a system of $Md$ linear equations in $N$ variables with $N>Md$. 
For each embedding $ \sigma_r: K \to \mathbb{C}$, $1 \le r \le d$, using \eqref{lin} we have 
$$
\sigma_r(a_{ij}) = \sum_{k=1}^d a_{ij}^{(k)}\sigma_r(\omega_k).
$$
Using the fact that $(\sigma_r(\omega_k))_{1 \le k,r \le d}$ is an invertible matrix, we obtain 
\begin{equation}\label{bd3}
\max_k|a_{ij}^{(k)}| \le c_1\max_{r}|\sigma_r(a_{ij})|
\end{equation}
for a positive constant $c_1$ depending on $K$.
Let $A= \max_{r,i,j}|\sigma_r(a_{ij})|$. 
Using Lemma \ref{Siegel1}
we get a non-trivial solution $\mathbf{x} \in \mathbb{Z}^N$ to the system of linear equations \ref{sys2} satisfying 
$$
|x_i| < (Nc_1A)^{dM/(N-dM)}.
$$
Since $\log{A} \le h(B)$, we get 
$$
h(\mathbf{x}) \le \frac{dM}{N-dM}( h(B) + \log N + c(K)) 
$$
where $c(K):=\log c_1$ only depends on $K$.
 \end{proof}

\subsection{Index}
Let \( K \) be a number field, and let \( P \in K[x_1, \ldots, x_m] \) be a polynomial in \( m \) variables with coefficients in \( K \). For a given variable \( x_h \), the partial degree of \( P \) with respect to \( x_h \) is denoted by \( \deg_{x_h}(P) \), which represents the highest exponent of \( x_h \) appearing in \( P \). The notion of the index \( \text{Ind}_{\bar{\alpha}, \bar{r}}(P) \) provides a measure of the vanishing order of \( P \) at \( \bar{\alpha} \) relative to the weights \( r_1, \ldots, r_m \). We give the definition below.

\begin{definition}
The \emph{index} of a polynomial \( P \in K[x_1, \ldots, x_m] \) at the point \( \bar{\alpha} = (\alpha_1, \ldots, \alpha_m) \in K^m \) with respect to a given \( m \)-tuple of positive integers \(\bar{r} = (r_1, \ldots, r_m) \) is defined as follows: 
\[
\op{Ind}_{\bar{\alpha}, \bar{r}}(P) = \min_{(i_1, \ldots, i_m)} \left\{ \frac{i_1}{r_1} + \ldots + \frac{i_m}{r_m} : \partial_{I} P(\bar{\alpha}) \neq 0 \right\},
\]  
where \( \partial_{I}  \) denotes the differential operator
$$
 \partial_{I} =\frac{1}{i_1!\cdots i_m!}\left(\frac{\partial}{\partial x_1}\right)^{i_1}
 \ldots \left(\frac{\partial}{\partial x_m}\right)^{i_m}
 $$
 corresponding to the multi-index \(I = (i_1, \ldots, i_m) \in \Z_{\ge 0}^m\).

\end{definition}

\begin{remark}
With the above notation, the index \( \op{Ind}_{\bar{\alpha}, \bar{r}}(P) \) has the following properties:
\begin{itemize}
    \item \( \op{Ind}_{\bar{\alpha}, \bar{r}}(P) \geq 0 \). This is immediate from the definition, as all terms \( \frac{i_h}{r_h} \) are non-negative.
    \item \( \op{Ind}_{\bar{\alpha}, \bar{r}}(P) = 0 \) if and only if \( P(\bar{r}) \neq 0 \). Indeed, if \( P(\bar{r}) \neq 0 \), then the partial derivative corresponding to the zero multi-index \( I = (0, \ldots, 0) \) is nonzero, yielding an index of zero.
    \item For any positive integer $t$, \( \ind_{\bar{\alpha}, t\bar{r}}(P) = \ind_{\bar{\alpha}, \bar{r}}(P)/t \). This follows by definition.
    \item If \( \deg_{x_h}(P) \leq r_h \) for all \( 1 \leq h \leq m \), then \( \op{Ind}_{\bar{\alpha}, \bar{r}}(P) \leq m \), with equality when  
    \[
    P(x) = \prod_{h=1}^m (x_h - \alpha_h)^{r_h}.
    \]  
    In this case, \( P \) vanishes to order exactly \( r_h \) along each coordinate direction, leading to an index equal to the sum of the weights.
\end{itemize}
\end{remark}
The next lemma gives a relation between height of a polynomial and its derivative.
\begin{lemma}\label{ht of derivative}
    Let $P\in K[x_1, \ldots, x_m]$ be a polynomial with $\deg_{x_h}P=d_h$ and  $ \partial_{I}  $ be a differential operator with  multi-index $ I = (i_1, \ldots, i_m) $. Then we have
    $$
    h(\partial_{I} P) \le (d_1+ \ldots+ d_m)\log 2 + h(P).
    $$
\end{lemma}
\begin{proof}
The proof follows by observing that the coefficients of
$$
 \partial_{I}P(x_1, \ldots, x_m) =\frac{1}{i_1!\cdots i_m!}\left(\frac{\partial}{\partial x_1}\right)^{i_1}
 \ldots \left(\frac{\partial}{\partial x_m}\right)^{i_m}P(x_1, \ldots, x_m) 
 $$
consist of products of $\binom{d_1}{i_1} \ldots \binom{d_m}{i_m} $ and coefficients of $P$.
         
    \end{proof}
We now establish several fundamental properties of the index under addition, multiplication, and differentiation of polynomials.

\begin{lemma}\label{prop index}
Let \( P, Q \in K[x_1, \ldots, x_m] \), let \( \bar{r} = (r_1, \ldots, r_m) \) be an \( m \)-tuple of positive integers, and let \( I = (i_1, \ldots, i_m) \in \Z_{\geq 0}^m \) be a multi-index. For any point \( \bar{\alpha} = (\alpha_1, \ldots, \alpha_m) \in K^m \), the following properties hold:
\begin{enumerate}
    \item \emph{Multiplicativity:}  
    \[
    \op{Ind}_{\bar{\alpha}, \bar{r}}(PQ) = \op{Ind}_{\bar{\alpha}, \bar{r}}(P) + \op{Ind}_{\bar{\alpha}, \bar{r}}(Q).
    \]
    
    \item \emph{Subadditivity under addition:}  
    \[
    \op{Ind}_{\bar{\alpha}, \bar{r}}(P + Q) \geq \min \left( \op{Ind}_{\bar{\alpha}, \bar{r}}(P), \op{Ind}_{\bar{\alpha}, \bar{r}}(Q) \right).
    \]
    
    \item \emph{Inequality under differentiation:}  
    \[
    \op{Ind}_{\bar{\alpha}, \bar{r}}(\partial_{I} P) \geq \op{Ind}_{\bar{\alpha}, \bar{r}}(P) - \sum_{h=1}^m \frac{i_h}{r_h}.
    \]
\end{enumerate}
\end{lemma}
\begin{proof}
These properties are an easy consequence of the definition of the index and are left to the reader.    
\end{proof}

\subsection{Generalized Wronskian}
   Let \( K \) be a field of characteristic zero, and let \( x_1, \ldots, x_m \) be a set of algebraically independent variables over \( K \). Consider a family of polynomials \( \varphi_1, \ldots, \varphi_n \in K[x_1, \ldots, x_m] \). To study the linear independence of these polynomials, we define a generalized version of the classical Wronskian determinant, which incorporates partial derivatives with respect to multiple variables.

\iffalse
Let \( \mu \) denote a differential operator of the form  
\[
\mu = \frac{d^{|I|}}{dx_1^{i_1} \cdots dx_m^{i_m}},
\]  
where \( I = (i_1, \ldots, i_m) \) is a multi-index, and the total order of \( \mu \) is given by \( O(\mu) = |I| = i_1 + \ldots + i_m \). For each \( i \leq n \), we choose a differential operator \( \mu_i \) such that \( O(\mu_i) \leq i - 1 \).

\begin{definition}
    The \emph{generalized Wronskian determinant} associated with the polynomials \( \varphi_1, \ldots, \varphi_n \) and differential operators \( \mu_1, \ldots, \mu_n \) is defined as  
\[
W_{\mu_1, \ldots, \mu_n}(x_1, \ldots, x_m) := \det \left( \mu_i \varphi_j \right)_{\substack{i \leq n \\ j \leq n}}.
\] 
\end{definition}
\fi

\begin{definition}
   A \emph{generalized Wronskian determinant} associated with the polynomials \( \varphi_1, \ldots, \varphi_n \) is defined as  
\[
W_{\mathbf{\mu_1}, \ldots, \mathbf{\mu_n}}(x_1, \ldots, x_m) = \det(\partial_{\mathbf{\mu_i}}\varphi_j)_{\substack{i \leq n \\ j \leq n}}
\]
where \( \mathbf{\mu_i} = (\mu_{i1}, \ldots, \mu_{im}) \) is a multi-index for each $1 \le i \le n$, with
$|\mathbf{\mu_i}|=\mu_{i1}+ \ldots +\mu_{im} \le i-1$.
\end{definition}
We now state a criterion for linear independence based on the generalized Wronskian.

\iffalse
\begin{lemma}[Wronskian criterion for linear independence]\label{Wronskian for LI}  
Let \( \varphi_1, \ldots, \varphi_n \in K[x_1, \ldots, x_m] \) be polynomials. Then \( \varphi_1, \ldots, \varphi_n \) are linearly independent over \( K \) if and only if the generalized Wronskian \( W_{\mu_1, \ldots, \mu_n}(x_1, \ldots, x_m) \) is not identically zero for some differential operators with $|\mu_i|=\mu_{i1}\leq i-1$ for all $i$. 
\end{lemma}
\fi

\begin{lemma}[Wronskian criterion for linear independence]\label{Wronskian for LI}  
Let \( \varphi_1, \ldots, \varphi_n \in K[x_1, \ldots, x_m] \) be polynomials. Then \( \varphi_1, \ldots, \varphi_n \) are linearly independent over \( K \) if and only if \( W_{\mathbf{\mu_1}, \ldots, \mathbf{\mu_n}}(x_1, \ldots, x_m) \) is not identically zero for some multi-indices $\mathbf{\mu_1}, \ldots , \mathbf{\mu_n}$ with $|\mathbf{\mu_i}|\leq i-1$ for all $i$. 
\end{lemma}

\begin{proof}  
Suppose \( \varphi_1, \ldots, \varphi_n \) are linearly dependent over \( K \). Then there exist constants \( c_1, \ldots, c_n \in K \), not all zero, such that  
\[
c_1 \varphi_1 + \ldots + c_n \varphi_n = 0.
\]  
Applying  the operator \( \partial_{\mathbf{\mu}} \) to this relation yields  
\[
c_1 \partial_{\mathbf{\mu}} \varphi_1 + \ldots + c_n \partial_{\mathbf{\mu}} \varphi_n = 0,
\] 
for every choice of \( \mathbf{\mu} \).
 In particular, this holds for all operators  $\partial_{\mathbf{\mu_1}}, \ldots \partial_{\mathbf{\mu_n}}$ used to construct the  generalized Wronskian $ W_{\mathbf{\mu_1}, \ldots, \mathbf{\mu_n}}(x_1, \ldots, x_m)$. Consequently, the columns of the Wronskian matrix are linearly dependent, implying that the determinant of the matrix is identically zero.

    Conversely, suppose \( \varphi_1, \ldots, \varphi_n \) are linearly independent over \( K \). Choose an integer $d>0$ such that $\deg_{x_h}\varphi_i \le d$, for all $1 \le i \le n$ and $1 \le h \le m$. Consider the Kronecker substitution 
    $$
    (x_1, \ldots, x_m) \mapsto (t, t^d, \ldots, t^{d^{m-1}}),
    $$
    where \(t\) is a new variable. This maps monomials in \(x_1, \ldots, x_m\) to powers of \(t\) and is injective for monomials with partial degrees less than \(d\). Thus, the polynomials \(\varphi_1, \ldots, \varphi_n\) are linearly independent over \(K\) if and only if  
\[
\Phi_j(t) = \varphi_j(t, t^d, \ldots, t^{d^{m-1}})
\]  
are linearly independent over \(K\). By a well-known result of Wronski ( see \cite{SMTB} Lemma 6.4.3), $\Phi_1, \ldots ,\Phi_n $ are linearly independent over \(K\) if and only if the Wronskian  
\begin{equation}\label{Wronski}
W(t) = \det\left(\frac{1}{(i-1)!}\frac{d^{i-1}}{dt^{i-1}} \Phi_j\right)_{i, j=1, \ldots, n}
\end{equation}
is not identically zero. Moreover, there exist universal polynomials \(a_{\boldsymbol{\mu}, i}(t; d, m) \in \mathbb{Q}[t]\) such that  
\[
\frac{d^{i-1}}{dt^{i-1}}  \Phi_j = \sum_{|\boldsymbol{\mu}| \leq i-1} a_{\boldsymbol{\mu}, i}(t; d, m) \partial_{\boldsymbol{\mu}} \varphi_j(t, \ldots, t^{d^{m-1}}).
\]  
Substituting this in \eqref{Wronski} we see that \(W(t)\) is a linear combination of generalized Wronskians \(W_{\boldsymbol{\mu}_1, \ldots, \boldsymbol{\mu}_n}(t, t^d, \ldots, t^{d^{m-1}})\) with \(|\boldsymbol{\mu}_i| \leq i-1\). Since \(W(t)\) is not identically zero, at least one of these generalized Wronskians must be non-zero.
\end{proof}

    %\begin{theorem}[Bezaut Theorem]
        %Let $P\in \mathbb{C}[x]$ and $\deg (P)=D$. Let $\alpha_1,\ldots,\alpha_{D+1}$ be roots of $P$. This implies $P\equiv 0.$ If $C$ and $C$ are plane curves with no common point, then 
        %\[\deg(C)\deg(C)=\deg(CC).\]
    %\end{theorem}

\section{Roth's theorem}\label{s 3}

In this section, we present a proof of Roth's theorem. By way of motivation, we give a quick proof of Liouville's theorem first.
\subsection{Liouville's theorem}
\begin{theorem}[Liouville's Theorem]\label{Liouville theorem}
  Let $\alpha$ be a real algebraic number of degree $n>1.$ Then there exists a constant $c(\alpha)>0$ such that for every $\dfrac{p}{q}\in \mathbb{Q}$
  \[\left|\alpha-\frac{p}{q}\right|>\frac{c(\alpha)}{q^n}.\]
  \end{theorem}
\begin{proof}
    Let $f(x)\in \mathbb{Z}$ be the minimal polynomial of $\alpha$. Say $f(x)=a_nx^n+\ldots+a_1x+a_0$. Therefore, for any $p/q\in \mathbb{Q}$, we have 
    \[\left|f\left(\frac{p}{q}\right)\right|=\frac{|a_np^n+a_{n-1}p^{n-1}q+\ldots+a_0q^n|}{q^n}\ge \frac{1}{q^n}.\]
    Let $\alpha_1=\alpha$, $\alpha_2$, $\ldots$, $\alpha_n$ be the distinct roots of $f(x)$. Then, we can write 
    \[f(x)=a_n(x-\alpha)\cdots(x-\alpha_n).\]
    This implies 
    \[\left|\alpha-\frac{p}{q}\right| \ge \frac{1/q^n}{|a_n|\prod_{i=2}^n\left|\frac{p}{q}-\alpha_i\right|}.\]
    Let $\displaystyle M=\max_{1\le i\le n}\{|\alpha_i|\}$. If $|p/q|\ge 2M$, then by triangle inequality,
    \[\left|\alpha-\frac{p}{q}\right|\ge M>\frac{M}{q^n}.\]
    On the other hand if $|p/q|<2M$, then $|\alpha_i-p/q|< 3M$ for $2\le i\le n.$ Therefore
    \[\left|\alpha-\frac{p}{q}\right|>\left(\frac{1}{|a_n|(3M)^{n-1}}\right)\frac{1}{q^n}.\]
    Choosing $c(\alpha)=\min\left\{M,\dfrac{1}{|a_n|(3M)^{n-1}}\right\},$ we get the required result.
\end{proof}

\subsection{Equivalent formulations of Roth's theorem}
We state two formulations of Lang's generalization of Roth's theorem over a number field $K$, and show that they are equivalent. Note that when $K=\Q$, one recovers the classical theorem of Roth.

\begin{theorem}[Roth's Theorem]\label{Roth's theorem}
Let $ K $ be a number field and $ S $ be a finite set of places  of $K$. Let $\alpha \in \bar{K}$ and $ \varepsilon >0 $ be fixed. Then there are finitely many $ \beta \in K $ such that 
\begin{align}\label{roth's inequality}
    \prod_{v \in S}\min(1, |\beta - \alpha|_{v}) < \frac{1}{H(\beta)^{2 + \varepsilon}}.
\end{align}
\end{theorem}
We will show that this formulation of Roth's theorem is equivalent to the following statement, which we shall prove at the end of this section.
\begin{theorem}\label{alternate version}
    Let $ K $ be a number field and $S$ be a finite set of places of $ K $. Let $\alpha \in \bar{K}$ and $ \delta >0 $ be fixed. Define a map $\zeta:S\to [0,1]$ such that  $\sum_{v\in S}\zeta(v)=1$. Then there are finitely many $ \beta \in K $ such that 
\begin{align}\label{height inequality}
   |\beta - \alpha|_{v}< \frac{1}{H(\beta)^{(2 + \delta)\zeta({v})}} \phantom{mm} {\rm ~ for ~ all ~} v \in S.
 \end{align}
\end{theorem}
\begin{proposition}\label{equivalent propn}
    Theorem \ref{Roth's theorem} and Theorem \ref{alternate version} are equivalent.
\end{proposition}
\begin{proof}
    Assume that Theorem \ref{Roth's theorem} is true. Let $\zeta:S\to [0,1]$ be a map such that  $\sum_{v\in S}\zeta(v)=1$. Assume $\beta$ satisfies \eqref{height inequality}, then 
    \[\prod_{v\in S}|\beta-\alpha|_{v}< \frac{1}{H(\beta)^{2 + \varepsilon}}.\]
    Since 
    \[
    \prod_{v \in S}\min(1, |\beta - \alpha|_{v})\leq \prod_{v\in S}|\beta-\alpha|_{v}<\frac{1}{H(\beta)^{2 + \varepsilon}}
    \]
    by Theorem \ref{Roth's theorem} there are finitely  many such $\beta$. 
    Conversely, assume that Theorem \ref{alternate version} is true. Let $s$ be the cardinality of the set $S$. Define $\zeta:S\to [0,1]$ such that for some integer $a_{v}$ with  $0 \le a_{v} \le s$, $\zeta(v)=a_{v}/s$ and $\sum_{v \in S}\zeta({v})=1.$ If $\mathcal{E}$ denotes the set of all such maps $\zeta$, then cardinality of $\mathcal{E}$ is finite. Next, let  $\beta$ be a solution to \eqref{roth's inequality}. Define   $\lambda_{v}(\beta)\in \mathbb{R}^{+}$ such that  
\[\min(1, |\beta - \alpha|_{v}) =\frac{1}{H(\beta)^{(2 + \varepsilon)\lambda_{v}(\beta)}}.\] Therefore, we have 
\[\prod_{v\in S}\min(1, |\beta - \alpha|_{v}) = \frac{1}{H(\beta)^{(2 + \varepsilon)\sum_{v \in S}\lambda_{v}(\beta)}}.\]
Comparing with \eqref{roth's inequality} we see that  $\sum_{v \in S}\lambda_{v}(\beta)\geq 1$. Therefore, we have 
\[\sum_{v \in S}[2s\lambda_{v}(\beta)]\geq \sum_{v \in S}2s\lambda_{v}(\beta)-s\geq s. \]
This gives that for each $v\in S$, there exists $b_{v}(\beta)\in \mathbb{N}$ such that $0\leq  b_{v}(\beta) \leq 2s\lambda_{v}(\beta)$ and $\sum_{v\in S}b_{v}(\beta)=s.$ Then the function $\zeta:S\to [0,1]$ defined as $\zeta({v})=b_{v}(\beta)/s 
$ lies in $\mathcal{E}$. This yields that a $\beta$ satisfying \eqref{roth's inequality} also satisfies \eqref{height inequality} for some $\zeta\in \mathcal{E}.$ From Theorem \ref{alternate version}, for every $\zeta\in \mathcal{E}$, there exist only finitely many $\beta$ satisfying \eqref{height inequality} and there are finitely many $\zeta\in \mathcal{E}.$ So, there are finitely many $\beta$ satisfying \eqref{roth's inequality}. This completes the proof. 
\end{proof}
\subsection{Auxiliary polynomial}

Next we use Siegel's lemma to construct an auxiliary polynomial with controlled height which vanishes to high order at the point $\bar{\alpha}:=(\alpha, \dots, \alpha)\in K^m$.

\begin{lemma}\label{I(m, epsilon) bound}
    Let $\varepsilon\in (0,1)$ and $r_1, \dots, r_m$ be positive integers and let $\mathcal{I}(m, \varepsilon)$ denote the number of tuples $I=\left(i_{1}, \ldots, i_{m}\right) \in \mathbb{Z}^{m}$ such that $0 \leq i_{h} \leq r_{h}$ for all $1 \leq h \leq m$ satisfying:
    \[\sum_{h=1}^{m} \frac{i_{h}}{r_{h}} \leq \frac{m}{2}(1-\varepsilon).\]
Then,  
    \[\# \mathcal{I}(m, \varepsilon)\leq (r_1+1)(r_2+1)\dots (r_m+1)\op{exp}(-\varepsilon^2 m/16).\] 
\end{lemma}
\begin{proof}
Note that $e^{t} \geq 1$ for all $t\geq 0$. As a result, we find that
\begin{align}\label{bound}
 \# \mathcal{I}(m, \varepsilon)
 &
 =\sum_{\left(i_{1}, \ldots, i_{m}\right) \in \mathcal{I}(m, \varepsilon)} 1 \nonumber \\ 
& \leq \sum_{\left(i_{1}, \ldots, i_{m}\right) \in \mathcal{I}(m, \varepsilon)} \exp \left(\frac{\varepsilon}{4}\left(\frac{m}{2}(1-\varepsilon) -\frac{i_{1}}{r_{1}}-\cdots-\frac{i_{m}}{r_{m}}\right)\right) \nonumber \\
& \leq \sum_{i_{1}=0}^{r_{1}} \cdots \sum_{i_{m}=0}^{r_{m}} \exp \left(\frac{\varepsilon}{4}\left(\frac{m}{2}(1-\varepsilon) -\frac{i_{1}}{r_{1}}-\cdots-\frac{i_{m}}{r_{m}}\right)\right) \nonumber \\
&=\exp \left(-\frac{\varepsilon^{2} m}{8}\right) \sum_{i_{1}=0}^{r_{1}} \cdots \sum_{i_{m}=0}^{r_{m}} \exp \left(\frac{\varepsilon}{4}\left(\frac{m}{2}-\frac{i_{1}}{r_{1}}-\cdots-\frac{i_{m}}{r_{m}}\right)\right) \nonumber \\
&=\exp \left(-\frac{\varepsilon^{2} m}{8}\right) \prod_{h=1}^{m}\left(\sum_{i=0}^{r_{h}} \exp \left(\frac{\varepsilon}{4}\left(\frac{1}{2}-\frac{i}{r_{h}}\right)\right)\right) .
\end{align}

Using the inequality $e^{t} \leq 1+t+t^{2}$, valid for all $|t| \leq 1$, we find that
$$
\begin{aligned}
\sum_{i=0}^{r} \exp \left(\frac{\varepsilon}{4}\left(\frac{1}{2}-\frac{i}{r}\right)\right) & \leq \sum_{i=0}^{r}\left\{1+\frac{\varepsilon}{4}\left(\frac{1}{2}-\frac{i}{r}\right)+\frac{\varepsilon^{2}}{16}\left(\frac{1}{2}-\frac{i}{r}\right)^{2}\right\} \\
& =\sum_{i=0}^{r}\left\{\left(1+\frac{\varepsilon}{8}+\frac{\varepsilon^{2}}{64}\right)-\left(\frac{\varepsilon}{4}+\frac{\varepsilon^{2}}{16}\right) \frac{i}{r}+\frac{\varepsilon^{2}}{16} \frac{i^{2}}{r^{2}}\right\} \\
& =(r+1)\left(1+\frac{\varepsilon^{2}}{192}+\frac{\varepsilon^{2}}{96 r}\right) \\
& \leq(r+1)\left(1+\frac{\varepsilon^{2}}{16}\right), \quad \text { using } r \geq 1.
\end{aligned}
$$
Plugging the above inequality in \ref{bound} and using $1+t \le e^t$, we obtain the following estimates:
$$
\begin{aligned}
\# \mathcal{I}(m, \varepsilon) & \leq \exp \left(-\frac{\varepsilon^{2} m}{8}\right) \prod_{h=1}^{m}\left(\left(r_{h}+1\right)\left(1+\frac{\varepsilon^{2}}{16}\right)\right) \\
& \leq \exp \left(-\frac{\varepsilon^{2} m}{8}\right) \prod_{h=1}^{m}\left(\left(r_{h}+1\right) \exp \left(\frac{\varepsilon^{2}}{16}\right)\right) \\
& =\left(r_{1}+1\right) \cdots\left(r_{m}+1\right) \exp \left(-\frac{\varepsilon^{2} m}{16}\right),
\end{aligned}
$$
which proves the result.
\end{proof}

\begin{lemma}[Auxiliary polynomial]\label{Step 1}
   Let $\alpha\in \bar{\Q}$ be an algebraic integer with $d=[\Q(\alpha):\Q]$ and $\varepsilon>0$. Let $m$ be large enough so that 
   $$\exp(\varepsilon^{2} m / 16)>2d $$ and ${\bar{r}}=(r_{1}, \ldots, r_{m})$ an $m$-tuple of positive integers. Then there exists a polynomial $P \in \mathbb{Z}\left[X_{1}, \ldots, X_{m}\right]$ with $\op{deg}_{X_h} P\leq r_h$ for $h=1, \dots, m$, such that the following conditions are satisfied:

   \begin{enumerate}
       \item $\op{Ind}_{\bar{\alpha}, \bar{r}}(P)\geq \frac{m}{2}(1-\varepsilon)$,
       \item $h(P)\leq (r_1+r_2+\dots+ r_m)C_\alpha$, where $C_\alpha>0$ is a constant depending only on $\alpha$.
   \end{enumerate}
   
\end{lemma}
\begin{proof}
    We write
    \[P(x_1,\ldots,x_m)=\sum_{j_1=0}^{r_1}\cdots\sum_{j_m=0}^{r_m}c(j_1,\ldots,j_m)x_1^{j_1}\cdots x_m^{j_m}\]
    with $c(j_1,\ldots,j_m)\in \mathbb{Z}$. Note that the 
    %multiplicative 
    height $H(P)$ of $P$ coincides with the maximum of the absolute values of the coefficients $|c(j_1, \dots, j_m)|$ of $P$. Clearly, the number of coefficients is $N:=(r_1+1)\cdots (r_m+1)$. Writing $I=(i_1, \dots, i_m)$ and $\bar{\alpha} = (\alpha, \ldots, \alpha)$, the index condition requires that
    \begin{equation}\label{number of equation}
       \partial_{I}P(\bar{\alpha})=0 \ \text{whenever}\ \sum_{h=1}^m\frac{i_h}{r_h}<\frac{m}{2}(1-\varepsilon).
    \end{equation}
    By Lemma \ref{I(m, epsilon) bound}, the number of such $m$-tuples $I$ is at most 
    \[M:=(r_1+1)\cdots (r_m+1)\exp(-\varepsilon^2m/16).\] 
   It follows that  \[\frac{N}{dM}=\frac{ (r_1+1)\cdots (r_m+1)}{d (r_1+1)\cdots (r_m+1)\cdot e^{-\varepsilon^2m/16}}>2\] if   $e^{\varepsilon^2m/16}>2d$. In particular, we find that
\[\frac{dM}{N-dM}=\frac{1}{\frac{N}{dM}-1}< 1. \]
%\frac{1}{2}.\]
 From the expression
    \[\partial_I P=\sum_{j_1=0}^{r_1}\cdots\sum_{j_m=0}^{r_m}\binom{j_1}{i_1}\dots \binom{j_m}{i_m} c(j_1,\ldots,j_m)x_1^{j_1-i_1}\cdots x_m^{j_m-i_m},\]
    one finds that 
\[\partial_I P(\bar{\alpha})=\sum_{j_{1}=0}^{r_{1}} \cdots \sum_{j_{m}=0}^{r_{m}}\binom{j_{1}}{i_{1}} \cdots\binom{j_{m}}{i_{m}} \alpha^{\left(j_{1}+\cdots+j_{m}-i_{1}-\cdots-i_{m}\right)} c(j_1,\ldots,j_m).\] Let $B$ be the vector formed by the coefficients of the equations, we find that 
\[h(B)\leq \op{max}_J\left\{ h\left(\binom{j_{1}}{i_{1}} \cdots\binom{j_{m}}{i_{m}} \alpha^{\left(j_{1}+\cdots+j_{m}-i_{1}-\cdots-i_{m}\right)} \right)\right\}\leq (r_1+\dots+r_h)(h(\alpha)+\log 2).\]
Applying Siegel's lemma, there exists a solution for this system of linear equations, namely $\mathbf{x}=(c(J))_J \in \mathbb{Z}^n$  such that 
\[\begin{split} h(\mathbf{x}) \le &\frac{dM}{N-dM}\left(h(B)+\log N+ c(K)\right)\\
< & (r_1+\dots+r_m)(h(\alpha)+\log 2)+\sum_{i=1}^m\log(r_i+1)+c(K),\\
< & (r_1+\ldots+r_m)(h(\alpha)+\log 2 +1 + c(K)) \\
\end{split}\]
where $K:=\Q(\alpha)$ and $c(K)$ is a positive constant depending only on $K$. This in turn implies that 
\[h(\mathbf{x})\leq (r_1+r_2+\dots+ r_m)C_\alpha\] for some constant $C_\alpha>0$ which depends only on $\alpha$.
\end{proof}
\subsection{A lower bound for the index} In this subsection, we prove lower bounds for the index of an auxiliary polynomial at certain points $\beta=(\beta_1, \dots, \beta_m)$ which shall be prescribed later on in our proof.

\begin{lemma}\label{easy boring lemma}
     Let \( P \in \mathbb{Z}[X_1, \dots, X_m] \) with \(\deg_{X_h}(P) \leq r_h\), and let \(\bar{\beta} = (\beta_1, \dots, \beta_m)\) be an \( m \)-tuple of algebraic numbers in a number field \( K \). Then, for any nonnegative integer \( m \)-tuple \( J = (j_1, \dots, j_m) \),  
\[
H(\partial_J P(\beta)) \leq 4^{(r_1+\cdots+r_m)} H(P) \prod_{h=1}^{m} H(\beta_h)^{r_h}.
\]
\end{lemma}
\begin{proof}
    Fixing $J$, we denote for ease of notation $T:=\partial_J P$. 
    %Let $I=(i_1, \dots, i_m)$, by Lemma \ref{ht of derivative} it is easy to see that the coefficients of $\partial_I T=\partial_{I+J} P$ have absolute value at most $2^{r_1+\cdots+r_m} |P|$. 
    We estimate $|T(\beta_1, \dots, \beta_m)|_v$ for $v\in M_K$. First, suppose that $v$ is archimedean, then using triangle inequality we get  
\[
|T(\beta_1, \dots, \beta_m)|_v \leq |(r_1+1) \cdots (r_m+1)|_v |T|_{v} \prod_{h=1}^{m} \max\{|\beta_h|_v,1\}^{r_h}.
\]  
since the number of terms is at most $(r_1+1) \cdots (r_m+1)$. 
On the other hand, if $v$ is non-archimedean, by the strong triangle inequality, we have
\[
|T(\beta_1, \dots, \beta_m)|_v \leq \prod_{h=1}^{m} \max\{|\beta_h|_v,1\}^{r_h}.
\]  
Multiplying over all \( v \in M_K \), we obtain
\[
H(T(\beta_1, \dots, \beta_m)) \leq 
(r_1+1) \cdots (r_m+1)H(T) \prod_{h=1}^{m} H(\beta_h)^{r_h}
\]
Using Lemma \ref{ht of derivative} and the fact that \( (r_h+1) \leq 2^{r_h} \) for all $h$, we get
$$ 
 H(T(\beta_1, \dots, \beta_m)) 
  \leq 4^{r_1+\cdots+r_m} H(P) \prod_{h=1}^{m} H(\beta_h)^{r_h} 
$$
and this completes the proof.
\end{proof}

\begin{lemma}[Index Lemma]\label{Step 2}
    Let \( 0 < \delta < 1 \) and \( 0 < \varepsilon < \delta/22 \). Choose \( m \) sufficiently large so that  
\[
\exp\left(\frac{\varepsilon^2 m}{16}\right) > 2d.
\]
Let $\bar{r}= (r_1, \ldots , r_m)$ be an $m$-tuple of positive integers and \( P \in \mathbb{Z}[x_1,\dots, x_m] \) a polynomial, with \( \deg_{x_h}(P) \leq r_h \) for each \( h \), satisfying the conditions of Lemma \ref{Step 1}. Given a finite set of places \( S \), define a partition of unity \( \zeta: S \to [0,1] \), and let \(\bar{\beta}=(\beta_1, \dots, \beta_m)\in K^m \) be such that
\[
|\beta_h - \alpha|_v \leq \frac{1}{H(\beta_h)^{(2+\delta) \zeta(v)}}
\]
for all \( v \in S \) and $1 \le h \le m $. Suppose that $\bar{\beta}$ satisfies the height bound  
\[
\max_{1\leq h \leq m} H(\beta_h)^{r_h} \leq \left( \min_{1\leq h \leq m} H(\beta_h)^{r_h} \right)^{1+\varepsilon}
\]
and
\(H(\beta_h) \ge (8e^{C_\alpha}H(\alpha))^{8/\alpha+4/11} \) for all $h$, where $C_\alpha>0$ is the constant depending only on $\alpha$ from Lemma \ref{Step 1}.
%the heights satisfy  
%\[
%r_1 h(\beta_1) \leq r_h h(\beta_h) \leq (1+\varepsilon) r_1 h(\beta_1)
%\]
%for all \( h \), with \( h(\beta_h) \) sufficiently large.
Then, we have 
\[
\operatorname{Ind}_{\bar{\beta},\bar{r}}(P) \geq \varepsilon m.
\]
\end{lemma}
\begin{proof}
    For ease of notation, we set $\theta:=\op{Ind}_{\bar{\alpha}, \bar{r}}(P)$, let $\delta\in (0,1)$ and $\theta_0\in (0, \theta)$. Set $D:=\min \left\{H\left(\beta_{h}\right)^{r_{h}}\mid 1 \le h \le m \right\}$, and suppose that $J=(j_{1}, \ldots, j_m)$ is a tuple such that $\sum_{h=1}^{m} \frac{j_{h}}{r_{h}} \leq \theta_{0}$. Then we claim that
\begin{equation}\label{claim in the index lemma}
\prod_{v \in S}|\partial_J P\left(\beta_{1}, \ldots, \beta_{m}\right)|_{v} 
\leq (8H(\alpha))^{(r_1+\dots+r_m)} H(P)D^{-(2+\delta)\left(\theta-\theta_{0}\right)}.
%4^{[K: \mathbb{Q}]\left(r_{1}+\cdots+r_{m}\right).} H(P)  .
\end{equation}
Before proving the claim, we show how the result follows from the above. \par Given $J=\left(j_{1}, \ldots, j_{m}\right)$ such that $\sum_{h=1}^{m} j_{h} / r_{h} \leq \varepsilon m$, we show that $\partial_{J} P\left(\beta_{1}, \ldots, \beta_{m}\right)=0$. 
We set $\theta_0=\varepsilon m$ and note that $\theta\geq \frac{m}{2}(1-\varepsilon)$ by Lemma \ref{Step 1}. Since $\varepsilon<1/22$, one has that $\theta_0< \theta$. 
We recall from Lemma \ref{Step 1} that 
$H(P)\leq e^{C_\alpha(r_1+\dots+r_m)}$ where $C_\alpha$ depends on $\alpha$.
From \eqref{claim in the index lemma},
\begin{equation}\label{exp 1}
\begin{aligned}
\prod_{v \in S}|\partial_{J} P\left(\beta_{1}, \ldots, \beta_{m}\right)|_{v} 
 \leq \frac{(8H(\alpha))^{\left(r_{1}+\cdots+r_{m}\right)} H(P)}{D^{\left(\theta-\theta_{0}\right)(2+\delta)}}
 \leq \frac{B(\alpha)^{\left(r_{1}+\cdots+r_{m}\right)}}{D^{\left(\frac{m}{2}(1-\varepsilon)-\varepsilon m\right)(2+\delta)}}
\end{aligned}
\end{equation}
where $B(\alpha):= 8e^{C_\alpha}H(\alpha)$ is a constant which only depends on $\alpha$.
 On the other hand, Lemma \ref{easy boring lemma} implies that
\begin{equation}\label{exp 2}
\begin{aligned}
H\left(\partial_{J} P(\beta_{1}, \ldots, \beta_{m})\right) & \leq 4^{\left(r_{1}+\cdots+r_{m}\right)} H(P) \prod_{h=1}^{m} H\left(\beta_{h}\right)^{r_{h}} \\
& \leq B(\alpha)^{\left(r_{1}+\cdots+r_{m}\right)} D^{m(1+\varepsilon)}.
\end{aligned}
\end{equation}
Using Lemma \ref{Fundamental inequality} we know that either the derivative $\partial_{J} P\left(\beta_{1}, \ldots, \beta_{m}\right)$ is zero, or else
$$
\prod_{v \in S}|\partial_{J} P\left(\beta_{1}, \ldots, \beta_{m}\right)|_{v} \geq H\left(\partial_{J} P\left(\beta_{1}, \ldots, \beta_{m}\right)\right)^{-1}.
$$
So it suffices to show that our hypotheses contradict the latter.
Assuming $\partial_{J} P\left(\beta_{1}, \ldots, \beta_{m}\right) \neq 0$, taking the product of \eqref{exp 1} and \eqref{exp 2} and collecting terms, we obtain
$$
D^{m((1+\delta / 2)(1-3 \varepsilon)-(1+\varepsilon))} \leq B(\alpha)^{2\left(r_{1}+\cdots+r_{m}\right)} .
$$
Now, since we assumed $\delta<1$ and $\varepsilon<\delta / 22$, we get
$$
(1+\delta / 2)(1-3 \varepsilon)-(1+\varepsilon)> \delta/2 - 11\varepsilon/2 > \delta / 4
$$
and hence
\iffalse
$$
D^\delta \leq B(\alpha)^{8\left(\frac{r_{1}+\cdots+r_{m}}{m}\right)}.
$$
\fi
$$
\max _{1 \leq j \leq m}\left\{H\left(\beta_{h}\right)^{r_{h}}\right\} \leq D^{1+\varepsilon} \leq B(\alpha)^{8\left(r_{1}+\cdots+r_{m}\right)(1+\varepsilon)/(m\delta) }.
$$
Selecting $j$ such that $r_{j}=\max _{h=1}^{m} r_{h}$, we deduce that
$$
H\left(\beta_{j}\right) \leq B(\alpha)^{8(1+\varepsilon) / \delta}.
$$
Therefore,  for $H\left(\beta_{j}\right) \leq (8e^{C_\alpha}H(\alpha))^{8 / \delta+4 / 11}$, we obtain the desired contradiction, which proves the result.
\par Thus what remains is to prove the claim \eqref{claim in the index lemma}. For ease of notation, set $T:=\partial_J P$. For $v\in M_K$ and $I=(i_1, \dots, i_m)$, we shall estimate $|\partial_{I} T(\bar{\alpha})|_v$. Noting that it is a sum of at most \( (r_1+1) \cdots (r_m+1) \) terms, each of which are bounded by
\[
|T|_v \max \{|\alpha|_v, 1\}^{r_1+\dots+r_m},
\]  
%\leq |P|_v (2 \max \{|\alpha|_v, 1\})^{r_1+\dots+r_m}
we deduce that
\begin{equation}\label{Bound del}
|\partial_{i_1, \dots, i_m} T(\alpha, \dots, \alpha)|_v 
\leq |(r_1+1) \cdots (r_m+1)|_v|T|_v(\max \{|\alpha|_v, 1\})^{r_1+\dots+r_m} .
\end{equation}  
On the other hand, since 
\[
\operatorname{Ind}_{\bar{\alpha}, \bar{r}} T = \operatorname{Ind}_{\bar{\alpha}, \bar{r}} \partial_{J} P \geq \operatorname{Ind}_{\bar{\alpha}, \bar{r}} P - \sum_{h=1}^{m} \frac{j_h}{r_h} \geq \theta - \theta_0,
\]  
we find that the Taylor expansion of $T$ at $\bar{\alpha} = (\alpha, \ldots , \alpha)$ is
\[
T(X_1, \dots, X_m) = \sum_{\frac{i_1}{r_1} + \dots + \frac{i_m}{r_m} \geq \theta - \theta_0} \partial_{i_1, \dots, i_m} T(\alpha, \dots, \alpha) (X_1-\alpha)^{i_1} \cdots (X_m-\alpha)^{i_m}.
\]  
Substituting \( X_h = \beta_h \) and using the fact that there are at most \( (r_1+1) \cdots (r_m+1) \) terms, we obtain  
\begin{align*}
|T(\beta_1, \dots, \beta_m)|_v 
&\leq \sum_{\frac{i_1}{r_1} + \dots + \frac{i_m}{r_m} \geq \theta - \theta_0} |\partial_{i_1, \dots, i_m} T(\alpha, \dots, \alpha)|_v |\beta_1 - \alpha|_v^{i_1} \cdots |\beta_m - \alpha|_v^{i_m}\\
&\leq \max_{\frac{i_1}{r_1} + \dots + \frac{i_m}{r_m} \geq \theta - \theta_0} \left(|\partial_{i_1, \dots, i_m} T(\alpha, \dots, \alpha)|_v 
|\beta_1 - \alpha|_v^{i_1} \cdots |\beta_m - \alpha|_v^{i_m}\right)\sum_{\frac{i_1}{r_1} + \dots + \frac{i_m}{r_m} \geq \theta - \theta_0} 1\\
&\leq C_v \max_{\frac{i_1}{r_1} + \dots + \frac{i_m}{r_m} \geq \theta - \theta_0} \left(|\partial_{i_1, \dots, i_m} T(\alpha, \dots, \alpha)|_v 
|\beta_1 - \alpha|_v^{i_1} \cdots |\beta_m - \alpha|_v^{i_m}\right)
\end{align*}
for all $v \in S$, where $C_v= \prod_{1 \le h \le m}(r_h+1)$. By \eqref{Bound del} we get
\begin{align*}
|T(\beta_1, \dots, \beta_m)|_v  
& \le C_v^2 ~|T|_v(\max \{|\alpha|_v, 1\})^{r_1+\dots+r_m}
\max_{\frac{i_1}{r_1} + \dots + \frac{i_m}{r_m} \geq \theta - \theta_0}\left(\prod_{1\le h \le m}|\beta_h - \alpha|_v^{i_h} \right)
\end{align*} 
Using the fact that \( \beta_h \) is close to \( \alpha \), we obtain 
\[
|T(\beta_1, \dots, \beta_m)|_v 
%\leq 4^{r_1+\dots+r_m} |P|_v \max_{\frac{i_1}{r_1} + \dots + \frac{i_m}{r_m} \geq \theta - \theta_0} \frac{1}{(H(\beta_1)^{i_1} \cdots H(\beta_m)^{i_m})^{(2+\delta) \xi_v}}.
\le C_v^2 ~|T|_v(\max \{|\alpha|_v, 1\})^{r_1 + \ldots + r_m}
\max_{\frac{i_1}{r_1} + \dots + \frac{i_m}{r_m} \geq \theta - \theta_0}\left(\prod_{h=1}^m
H(\beta_h)^{-i_h(2+\delta) \zeta(v)}\right)
\]  
Using  
\[
H(\beta_1)^{i_1} \cdots H(\beta_m)^{i_m} = (H(\beta_1)^{r_1})^{i_1/r_1} \cdots (H(\beta_m)^{r_m})^{i_m/r_m}\geq D^{\theta - \theta_0},
\]  
we conclude  
\[
|T(\beta_1, \dots, \beta_m)|_v 
\leq
C_v^2~ |T|_v(\max \{|\alpha|_v, 1\})^{r_1+\dots+r_m}D^{-(\theta - \theta_0)(2+\delta) \zeta(v)}
%\frac{4^{r_1+\dots+r_m} |P|_v}.
\]  
Multiplying over \( v \in S \), and using \( \sum_{v \in S} \zeta(v) = 1 \), we obtain  
$$
\prod_{v \in S} |T(\beta_1, \dots, \beta_m)|_v \leq 
\prod_{1 \le h \le m}(r_h+1)^2~ H(T) H(\alpha)^{(r_1 + \ldots +r_m)}D^{-(\theta - \theta_0)(2+\delta) \zeta(v)}.
$$
Now using Lemma \ref{ht of derivative} and the fact that $r_h+1 \le 2^{r_h}$ for all $h$, we get
$$
\prod_{v \in S} |T(\beta_1, \dots, \beta_m)|_v \leq 
\frac{(8H(\alpha))^{(r_1+\dots+r_m)} H(P)}{D^{(\theta - \theta_0)(2+\delta)}}.
$$
\end{proof}
\subsection{Roth's lemma}
Next we show that index of a polynomial $P$ at a point $\bar{\beta}$ with respect to certain $m$-tuples of positive integers cannot be very large.
\begin{lemma}[Roth's Lemma]\label{Roth's lemma}
Let $P\in\Z[x_1, \ldots, x_m]$ be a non-zero polynomial with $ \deg_{x_h}(P) \le r_h$, where $r_h>0$ for all $h=1,\dots, m$. Let ${\bar{\beta}}=(\beta_1, \ldots, \beta_m) \in \mathbb{ \bar{Q}}^m$ and fix a  real number $0 < \eta \le \frac{1}{2}$. Suppose that
\begin{itemize}
\item the $m$-tuple $\bar{r} = (r_1, \ldots, r_m)$ satisfies 
$$
\frac{r_{j+1}}{r_j} \le \eta^{2^{m-1}};
$$
\item the point $(\beta_1, \ldots, \beta_m)$ satisfies 
$$
\eta^{2^{m-1}} \min_{1 \le h \le m}(r_h h(\beta_h)) \ge h(P) +2mr_1.
$$
Then $\ind_{\bar{\beta}, \bar{r} }(P) \le 2m \eta$.
\end{itemize}  
\end{lemma}

\begin{proof}
We prove the result by using induction on $m$.
Let $\ord_{\bar{\beta}}(P)$ denote the order of vanishing of $P$ at $\bar{\beta}$.
For $m=1$, let
$
P(x)= (x-\beta_1)^lQ(x)
$
for some polynomial $Q(x)$ with $Q(\beta_1)\ne 0$.
Then we have 
$$
l= \ord_{\beta_1}(P)= r_1 \ind_{\beta_1, r_1}(P) .
$$
This implies
$$
r_1\ind_{\beta_1, r_1}(P)h(\beta_1)
 =  l\log M(x-\beta_1)
 \le  \log M(P)
 \le  h(P) + r_1. 
$$
Therefore
$$
\ind_{\beta_1, r_1}(P) \le \frac{h(P)+r_1}{r_1h(\beta_1)} \le \eta^2 \le 2 \eta.
$$
We assume that the result is true for polynomials in at most $m-1$ variables. Let $P$ be a polynomial in $m$ variables so that $r_h \ge 1$ for all $1 \le h \le m$. Consider  a decomposition of $P$ of the form
\[P(x_1,\ldots,x_m)=\sum_{j=1}^{k}\phi_j(x_1,\ldots,x_{m-1})\psi_j(x_m),\]
where $\phi_j$'s and $\psi_j$'s are polynomials with coefficients in $\bar{\mathbb{Q}}$. We pick a decomposition of minimal length, say $k\le r_m+1$.  Since in this decomposition $k$ is minimal, $\phi_1,\ldots,\phi_k$ are linearly independent polynomials over $\bar{\mathbb{Q}}$. Indeed if $\phi_1,\ldots,\phi_k$ are not linearly independent, then there exist $c_1,\ldots, c_k \in \bar{\mathbb{Q}}$, not all zero, such that
\[c_1\phi_1+\ldots +c_k\phi_k=0. \]
% Since $\phi_1,\ldots,\phi_k$ have rational coefficients, there exists rational numbers $C,\ldots, c_k$ with these properties. 
Without loss of generality if $c_k\ne 0$, then  
\[P = \sum_{j=1}^{k-1}\phi_j\left(\psi_j-\frac{c_j}{c_k}\psi_k\right).\]
This contradicts the minimality of $k$. In the same way one can show that, $\psi_1,\ldots,\psi_k$ are linearly independent over $\bar{\mathbb{Q}}$.
Applying the Wronskian criterion for linear independence (Theorem \ref{Wronskian for LI}), we see that the Wronskian
$$
U_1(x_m) \coloneq W(\psi_1, \ldots, \psi_k)
=
\det\left(\frac{1}{(i-1)!}\frac{\partial^{i-1}}{\partial x_m^{i-1}}\psi_j(x_m)\right)_{1\le i, j\le k} 
$$
is not identically zero.
Similarly there exist differential operators
$\partial_{\mu_1}, \ldots ,\partial_{\mu_k}$, with 
$$
|\mu_i|= \mu_{i1}+ \ldots + \mu_{i(m-1)} \le i-1 \le k-1 \le r_m
$$ for all $ 1 \le i \le k$, such that the Wronskian
$$
U_2(x_1, \ldots, x_{m-1}) \coloneq
W_{\mu_1, \ldots , \mu_k}(\phi_1, \ldots, \phi_k)= \det(\partial_{\mu_i}\phi_j(x_1, \ldots, x_{m-1}))_{1 \le i, j \leq k} 
$$
is not identically zero.

\iffalse
\[\Delta_i=\frac{\partial^{i_1+\ldots+i_{m-1}}}{\partial x_1^{i_1}\cdots \partial x_{m-1}^{i_{m-1}}},\]
where $1\le i\le k$ and $i_1+\ldots+i_{m-1}\le i-1\le k-1\le d_m$ such that
\[W(\Delta_i, \phi_j)=V(x_1,\ldots,x_{m-1})=\det(\Delta_i\phi_j)_{{1\le i,j\le k}}\ne 0.\]  Therefore,
\fi

Multiplying the two determinants we see that
\begin{align*}
    V(x_1,\ldots,x_m)
    & \coloneqq 
    \det\left(\frac{1}{(j-1)!}\frac{\partial^{j-1}}{\partial x_m^{j-1}}\partial_{\mu_i}P\right)_{{1\le i, j\le k}}\\
   % =\det \left(\frac{\partial^{j-1}}{\partial x_m^{j-1}}\partial_{\mu_i}\sum_{r=1}^k\phi_r\psi_r\right)_{1\le i,j\le k}\\
    &= \det \left(\left(\frac{1}{(j-1)!}\frac{\partial^{j-1}}{\partial x_m^{j-1}}\psi_r\right)_{1\le j,r\le k}
    (\partial_{\mu_i}\phi_r)_{1\le r,i\le k}\right)\\
    &=U_1(x_m)U_2(x_1,\ldots,x_{m-1}).
\end{align*}

Clearly, $\deg_{x_m}(U_1) \le kr_m$ and $\deg_{x_h}(U_2) \le k r_h$ for all $1 \le h \le m-1$. 
Moreover,  each term of $  V(x_1,\ldots,x_m)$ in the first equation above is a product of $k$ polynomials and the partial degree of each polynomial with respect to $x_h$ is at most $r_h$. From Lemma \ref{ht of derivative}, we have
\begin{equation}\label{h(delta P)=}
h\left(\frac{1}{(j-1)!}\frac{\partial^{j-1}}{\partial x_m^{j-1}}\partial_{\mu_i}P\right)  \le  h(P)+(r_1 + \ldots + r_m)\log 2.
\end{equation}
Since the determinant is a sum of $k!$ terms, Proposition \ref{propn sum of f1+...+fr} gives us that:
\begin{equation}\label{h(V) bound}h(V)\leq \log k! + \op{max}_{\pi} \left\{h\left(\prod_{i=1}^k \frac{1}{(\pi(i)-1)!}\frac{\partial^{\pi(i)-1}}{\partial x_m^{\pi(i)-1}}\partial_{\mu_i}P\right)\right\},\end{equation}
where $\pi$ ranges over all permutations of $\{1, \dots, k\}$. Lemma \ref{ht of product poly} then implies that
\[\begin{split}& h\left(\prod_{i=1}^k \frac{1}{(\pi(i)-1)!}\frac{\partial^{\pi(i)-1}}{\partial x_m^{\pi(i)-1}}\partial_{\mu_i}P\right)\\ \leq &\sum_{i=1}^k h\left(\frac{1}{(\pi(i)-1)!}\frac{\partial^{\pi(i)-1}}{\partial x_m^{\pi(i)-1}}\partial_{\mu_i}P\right)+(r_1+\dots+r_m)\log 2+m(k-1)\log 2\\
\leq & k h(P) +k(r_1+\dots+r_m)\log2+(r_1+\dots+r_m)\log 2+m(k-1)\log 2\\
\leq & k h(P)+2k (r_1+\dots+r_m)\log2,\end{split}\]
where the second inequality invokes \eqref{h(delta P)=}. Substitute the above inequality into \eqref{h(V) bound}, we deduce that:
\[h(V)\le \log(k!)+kh(P)+2k(r_1 + \ldots + r_m)\log 2.\]
Since $r_{j+1} \le \eta^{2^{m-1}}r_j$ with $\eta \le 1/2$ and $m \ge 2$ we get
$$
r_1 + \ldots + r_m \le (1+ \eta^{2^{m-1}} +\ldots + \eta^{(m-1)2^{m-1}})r_1 \le \frac43 r_1.
$$
This along with the bound 
$$
\log(k!) \le k \log k \le k(k-1) \le kr_m \le k\frac{r_1}{2} 
$$ implies that
$$ 
h(V)  \le   k\left(h(P)+ \left(\frac83 \log 2 + \frac12\right) r_1\right) 
\le k\left(h(P)+ \frac52 r_1\right).
$$

Our next aim is to show that the hypotheses of Roth's lemma are satisfied for $U_1$ and $U_2$. 
First we show that $U_1$ satisfies the hypotheses of Roth's lemma.
We note that the first condition is trivially satisfied as $U_1$ is a polynomial in one variable with $\deg_{x_m}(U_1) \le kr_m$.
By Remark \ref{height of poly} we have $h(V)=h(U_1)+h(U_2)$. Therefore
$$
h(U_1) +2kr_m \le k\left(h(P)+ \frac52 r_1\right) + k\frac{r_1}{2}
\le k(h(P)+ 3 r_1)\leq k(h(P)+ 2m r_1).
$$
By the assumption on $P$, we have that 
\[h(P)+2m r_1\leq \eta^{2^{m-1}} r_m h(\beta_m),\]
and thus
$$
h(U_1) +2kr_m \le k(\eta^{2^{m-1}}r_mh(\beta_m))= \eta^{2^{m-1}}(kr_m)h(\beta_m).
$$
Applying Roth's lemma for polynomials in one variable, we get 
\begin{equation}\label{U1}
\ind_{\beta_m,r_m}(U_1) = k\ind_{\beta_m,kr_m}(U_1)\le k\eta^{2^{m-1}}.
\end{equation}
Next we shall show that $U_2$ satisfies the hypotheses of Roth's lemma.
It is a polynomial in $m-1$ variables with $\deg_{x_h}(U_2) \le kr_h$ for all $ 1 \le h \le m-1$. Now
$$
\frac{kr_{j+1}}{kr_j}= \frac{r_{j+1}}{r_j} \le \eta^{2^{m-1}}=(\eta^2)^{2^{m-2}}.
$$
Now using $h(U_2) \le h(V)$ and the assumption on $P$, we have
$$
h(U_2) + 2(m-1)kr_1 \le k(h(P)+ 2m r_1) 
\le k\eta^{2^{m-1}} (r_h h(\beta_h)) = (\eta^2)^{2^{m-2}}(kr_h) h(\beta_h)
$$
for all $1 \le h \le m-1$.
Applying Roth's lemma for polynomials in $m-1$ variables we get
$$
\ind_{(\beta_1, \ldots, \beta_{m-1}),(r_1, \ldots , r_{m-1})}(U_2)
= k\ind_{(\beta_1, \ldots, \beta_{m-1}),(kr_1, \ldots , kr_{m-1})}(U_2)\le 2k(m-1)\eta^{2}.
$$
Now using Lemma \ref{prop index} we get
\begin{equation}\label{upper bound W}
\ind_{\bar{\beta},\bar{r}}(W) = \ind_{\bar{\beta},\bar{r}}(U_1) + \ind_{\bar{\beta},\bar{r}}(U_2) \le k\eta^{2^{m-1}}+ 2k(m-1)\eta^{2}.
\end{equation}
The next step would be to relate the index of $W$ with the index of $P$ to finally get a bound on index of $P$. 
Using Lemma \ref{prop index}, for each entry in the matrix corresponding to $W$ we have 

\iffalse
\begin{lemma}
    \[\ind(W)\ge \frac{k}{2}\min\{\ind(P),\ind(P)^2\}-\frac{r_m}{r_{m-1}}.\]
\end{lemma}
\fi

\begin{align*}
\ind_{\bar{\beta}, \bar{r}}\left(\frac{1}{(j-1)!}\frac{\partial^{j-1}}{\partial x_m^{j-1}}\partial_{\mu_h}P\right)
&\ge \ind_{\bar{\beta}, \bar{r}}(P) - \sum_{t=1}^{m-1}\frac{\mu_{ht}}{r_t} - \frac{j-1}{r_m}
\\
&\ge \ind(P)-\frac{\mu_{h1}+\ldots+\mu_{h(m-1)}}{r_{m-1}}-\frac{j-1}{r_m}\\
&\ge \ind(P)-\frac{r_m}{r_{m-1}}-\frac{j-1}{r_m}.
\end{align*}

Since the determinant is a sum of $k!$ terms each of which is a product of $k$ polynomials, using Lemma \ref{prop index} we get,
\begin{align*}
\ind_{\bar{\beta}, \bar{r}}(W)
    &\ge \sum_{j=1}^{k}\min_{\mu_{h1},\ldots , \mu_{h(m-1)}}
\ind_{\bar{\beta}, \bar{r}}\left(\frac{1}{(j-1)!}\frac{\partial^{j-1}}{\partial x_m^{j-1}}\partial_{\mu_h}P\right)\\
    & \ge \sum_{j=1}^k \max\left\{\ind_{\bar{\beta}, \bar{r}}(P)-\frac{r_m}{r_{m-1}}-\frac{j-1}{r_m}, 0\right\}\\
    &\ge \sum_{j=1}^{k}\max\left\{\ind_{\bar{\beta}, \bar{r}}(P)-\frac{j-1}{r_m},0\right\}-\frac{kr_m}{r_{m-1}}.
    \end{align*}
When $\ind_{\bar{\beta}, \bar{r}}(P)>\frac{k-1}{r_m}$, we get
\[\sum_{j=1}^{k}\max\left\{\ind_{\bar{\beta}, \bar{r}}(P)-\frac{j-1}{r_m},0\right\}
= k\ind_{\bar{\beta}, \bar{r}}(P)-\frac{(k-1)k}{2r_m}
\ge \frac{k}{2}\ind_{\bar{\beta}, \bar{r}}(P).\]
On the other hand when 
$\ind_{\bar{\beta}, \bar{r}}(P)\le \frac{k-1}{r_m}$, we take
$N= [r_m\ind_{\bar{\beta}, \bar{r}}(P)]$. This implies 
$$
\sum_{j=1}^{k}\max\left\{\ind_{\bar{\beta}, \bar{r}}(P)-\frac{j-1}{r_m},0\right\}
\ge \sum_{j=1}^{N+1}\left(\ind_{\bar{\beta}, \bar{r}}(P)-\frac{j-1}{r_m}\right)
= (N+1)\left(\ind_{\bar{\beta}, \bar{r}}(P)-\frac{N}{2r_m}\right).
$$
Since $N \le r_m\ind_{\bar{\beta}, \bar{r}}(P) \le N+1$, we get
$$
\sum_{j=1}^{k}\max\left\{\ind_{\bar{\beta}, \bar{r}}(P)-\frac{j-1}{r_m},0\right\}
\ge \frac{r_m}{2}\ind_{\bar{\beta}, \bar{r}}(P)^2
\ge 
\max\left\{\frac{k-1}{2},\frac12\right\}\ind_{\bar{\beta}, \bar{r}}(P)^2  \ge
\frac{k}{4}\ind_{\bar{\beta}, \bar{r}}(P)^2
$$
Therefore using the fact that  $\ind_{\bar{\beta}, \bar{r}}(P) \le m$ and $m \ge 2$, we get
\begin{equation}\label{upper bound P}
    \ind_{\bar{\beta}, \bar{r}}(W)+\frac{kr_m}{r_{m-1}}
    \ge \min\left\{\frac{k}{2}\ind_{\bar{\beta}, \bar{r}}(P),
    \frac{k}{4}\ind_{\bar{\beta}, \bar{r}}(P)^2\right\}
    \ge \frac{k}{2m}\ind_{\bar{\beta}, \bar{r}}(P)^2.
\end{equation}
Using \eqref{upper bound W}, \eqref{upper bound P} and the fact that $r_m \le \eta^2r_{m-1}$  we obtain
$$
\frac{k}{2m}\ind_{\bar{\beta}, \bar{r}}(P)^2 \le k\eta^{2^{m-1}}+ 2k(m-1)\eta^{2} + \frac{kr_m}{r_{m-1}} \le 2km\eta^2.
$$
Hence 
$$
\ind_{\bar{\beta}, \bar{r}}(P) \le 2m\eta.
$$
\end{proof}

\subsection{Proof of Roth's Theorem} We now provide a proof of Theorem \ref{alternate version}, which is equivalent to Roth's theorem. Suppose that there exists  $\delta>0$ such that 
\begin{equation}\label{Roth21}
    |\beta-\alpha|_{v}\leq \frac{1}{H(\beta)^{(2+\delta)\zeta(v)}}
\end{equation}
has infinitely many solutions $\beta \in K$. Let $0 < \delta <1$ and choose $\varepsilon \in (0, \frac{\delta}{22})$. Next we choose $m$ large enough so that 
$$
\exp(\varepsilon^2 m/16) > 2d,
$$
where $d=[\mathbb{Q}(\alpha):\mathbb{Q}]$. Let $\eta = \varepsilon /4$.
Since \eqref{Roth21} has infinitely many solutions in $K$ and using the fact that $K$ has only finitely many elements of bounded height we can find a solution $\beta_1$ of \eqref{Roth21} such that
$$
h(\beta_1) \ge \left({\frac{8}{\alpha}+\frac{4}{11}}\right) \left(\log 8+C_\alpha+h(\alpha)\right) \quad {\rm and }
 \quad h(\beta_1) \ge \frac{(m\log C_\alpha +2)}{\eta^{2^{m-1}}}
 $$
where $C_\alpha$  is the constant appearing in Lemma \ref{Step 1}. Note that $C_\alpha$ depends only on the choice of $\alpha$. 
Next, we successively choose solutions $\beta_2, \ldots , \beta_m$  \label{Roth2} satisfying
$$
\eta^{2^{m-1}}h(\beta_{j+1}) \ge 2 h(\beta_j)
$$
for all $2 \le j \le m-1$.
Since $\eta <1$, this implies that $ h(\beta_j)$ is strictly increasing.
We choose a positive integer $r_1$ such that 
$$
\eta^{2^{m-1}}r_1 h(\beta_1) \ge 2 h(\beta_m).
$$
Now we choose $r_2, \ldots , r_m$ such that 
$$
r_j = \left[\frac{r_1h(\beta_1)}{h(\beta_j)}\right]+1.
$$

Since we chose $m$ that satisfies $\exp(\varepsilon^2 m/16) > 2d$, by Lemma \ref{Step 1} there exists a polynomial $P\in \mathbb{Z}[x_1,\ldots,x_m]$ such that $\deg_{x_h}(P)\le r_h$ for $h=1,\dots, m$ satisfying
$$
\op{Ind}_{\bar{\alpha}, \bar{r}}(P)\geq \frac{m}{2}(1-\varepsilon)
$$
and
$$
h(P) \le (r_1 + \ldots + r_m)\log C_\alpha
$$
where 
$\bar{\alpha} = (\alpha, \ldots, \alpha)$, $\bar{r} = (r_1, \ldots , r_m)$.
The choice of $r_j$ and $\beta_j$ allows us to get 
$$
r_1h(\beta_1) \le r_jh(\beta_j) \le r_1h(\beta_1) + h(\beta_j).
$$
Using the fact that $ h(\beta_j)$ is increasing and by choice of $r_1$, we obtain  
\begin{equation}\label{B1}
r_1h(\beta_1) \le  r_1h(\beta_1) + h(\beta_m) 
\le (1+ \varepsilon)r_1h(\beta_1).
\end{equation}
Therefore the assumptions of Lemma \ref{Step 2} are satisfied and we get 
\begin{equation}\label{B2}
\ind_{\bar{\beta}, \bar{r}}(P) \ge \varepsilon m.
\end{equation}
Again, by our choice of $r_j$ and $\beta_j$ we obtain
\begin{align*}
\frac{r_{j+1}}{r_j} 
&\le 
\left(\frac{r_1h(\beta_1)}{h(\beta_{j+1})}+1\right)\left(\frac{r_1h(\beta_1)}{h(\beta_j)}\right)^{-1}\\
&= 
\frac{h(\beta_j)}{h(\beta_{j+1})}+
\frac{h(\beta_{j})}{r_1 h(\beta_1)}\\
&\le 
\frac{\eta^{2^{m-1}}}{2} + \frac{\eta^{2^{m-1}}}{2} = \eta^{2^{m-1}}.
\end{align*}
Therefore $r_{j+1} \le \eta^{2^{m-1}}r_j$ for all $ 1 \le j \le m-1$.
Now using the bounds on $h(P)$ and $h(\beta_1)$ that we have above, we get
$$
h(P) + 2 m r_1 \le (r_1 + \ldots + r_m)\log C_\alpha + 2 m r_1
\le mr_1(\log C_\alpha + 2) \le \eta^{2^{m-1}}r_1 h(\beta).
$$
Since
$$
\min_{1 \le h \le m}r_jh(\beta_j)=r_1h(\beta_1)
$$
we see that the conditions of Lemma \ref{Roth's lemma} are satisfied. Therefore we get
\begin{equation}\label{B3}
\ind_{\bar{\beta}, \bar{r}}(P)
\le 2 m \eta = \frac{m \varepsilon}{2}.
\end{equation}
This is a contradiction to (\ref{B2}) and hence we get the result.

\section{The subspace theorem}\label{s 4}
\subsection{Geometry of numbers and Minkowski's second theorem}

\par We begin by recalling several fundamental notions from the geometry of numbers, particularly in the context of Euclidean spaces over the ring of adeles of a number field \( K \). These notions are essential for understanding subsequent results, including the application of Minkowski’s second theorem in the adelic setting. 

Let \( G \) be a locally compact topological group equipped with a Haar measure \( \mu_G \). A Haar measure is a Borel measure on \( G \) that is invariant under left translation. Specifically, for every \( y \in G \) and any continuous function \( f: G \to \mathbb{R} \) with compact support, the measure satisfies the invariance property:  
\[
\int_G f(yx) \, d\mu_G(x) = \int_G f(x) \, d\mu_G(x).
\]
This invariance condition characterizes the Haar measure uniquely up to multiplication by a positive scalar. The existence and uniqueness of Haar measure on any locally compact group are classical results that play a central role in harmonic analysis and number theory.

Let \( K \) be a number field of degree \( r = [K : \mathbb{Q}] \). Denote by \( M_K \) the set of places of \( K \). For each place \( v \in M_K \), the completion of \( K \) at \( v \) is denoted by \( K_v \), and we let \( |\cdot|_v \) be a corresponding normalized absolute value. It is well known that \( K_v \) is locally compact for each \( v \). Define  
\[
R_v := \{ \alpha \in K_v \mid |\alpha|_v \leq 1 \}.
\]
The set \( R_v \) serves as a local analogue of the ring of integers for non-archimedean places \( v \). For any measurable subset \( \Omega \subseteq K_v \) and any \( \alpha \in K_v \), the Haar measure \( \mu_v \) on \( K_v \) satisfies the scaling property:
\[
\mu_v(\alpha \Omega) = |\alpha|_v \mu_v(\Omega).
\]
This property highlights how \( \mu_v \) interacts with the multiplicative structure of \( K_v \), making it a crucial tool for studying adelic measures.

The adele ring \( K_{\mathbb{A}} \) of \( K \) is defined as  
\[
K_{\mathbb{A}} := \left\{ x = (x_v)_{v \in M_K} \in \prod_{v \in M_K} K_v \mid x_v \in R_v \text{ for all but finitely many } v \right\}.
\]
The topology on \( K_{\mathbb{A}} \) is \emph{not} the product topology, as the latter would not yield a locally compact space. Instead, \( K_{\mathbb{A}} \) is endowed with a topology that arises from the restricted product of the local fields \( K_v \) with respect to the subsets \( R_v \) at non-archimedean places.

Given any finite set \( S \subset M_K \) that contains all the archimedean places, define  
\[
H_S := \prod_{v \in S} K_v \times \prod_{v \notin S} R_v \subset K_{\mathbb{A}}.
\]
The set \( H_S \) is open in the restricted product topology, and each such \( H_S \) is locally compact. By specifying that \( H_S \) is open for every such \( S \), we endow \( K_{\mathbb{A}} \) with the structure of a locally compact topological ring. This topological structure ensures that \( K \) embeds discretely into \( K_{\mathbb{A}} \) via the diagonal embedding  
\[
\iota: K \hookrightarrow K_{\mathbb{A}}, \quad \iota(x) := (x)_{v \in M_K}.
\]
With respect to this embedding, the quotient \( K_{\mathbb{A}} / K \) is compact, a fundamental result in adelic analysis that underpins much of the theory of heights and Diophantine approximation.

We now describe the normalization of the Haar measure \( \mu_v \) on each local field \( K_v \):
\begin{itemize}
    \item If \( v \) is non-archimedean and lies above a rational prime \( p \), the Haar measure \( \mu_v \) is normalized so that \( \mu_v(R_v) = |D_{K_v / \mathbb{Q}_p}|_p^{1/2} \), where \( D_{K_v / \mathbb{Q}_p} \) denotes the discriminant of \( K_v \) over \( \mathbb{Q}_p \) and \( |\cdot|_p \) denotes the \( p \)-adic norm.
    \item If \( v \) is real, then \( \mu_v \) is taken to be the standard Lebesgue measure on \( \mathbb{R} \).
    \item If \( v \) is complex, then \( \mu_v \) is twice the standard Lebesgue measure on \( \mathbb{C} \).
\end{itemize}
This normalization ensures compatibility of the measures across different local fields, which is crucial for the product formula and for defining a well-behaved measure on the adele ring \( K_{\mathbb{A}} \). Such a choice is motivated by the need to maintain consistency when comparing measures across archimedean and non-archimedean components, facilitating the study of heights, lattices, and fundamental domains in the adelic setting.
For a finite subset $S$ of $M_K$ containing the archimedean primes, the product measure 
\[\mu_S:=\prod_{v\in S} \mu_v\times \prod_{v\notin S}\mu_{v|R_v}\]
is a Haar measure on the open topological subgroup $H_S$ of $K_{\mathbb{A}}$. These measures fit together to give a Haar measure $\mu:=\varprojlim_S \mu_S$ on $K_{\mathbb{A}}$.
\par Now fix a natural number $N$ and denote by $E, E_v, E_{\mathbb{A}}$ the spaces $K^N$, $K_v^N$ and $K_{\mathbb{A}}^N$ respectively. We set $\beta$ to be the Haar measure on $E_{\mathbb{A}}$ which is the $N$-fold product of $\mu$.
\begin{definition}
    For a finite place $v$, a $K_v$ lattice in $E_v$ is an open and compact $R_v$  sub-module of $E_v$. A lattice in a finite dimensional real or complex vector space $V$ is a discrete subgroup $\Lambda\subset V$ such that $V/\Lambda$ is compact.
\end{definition}

There is a natural diagonal inclusion $\iota: E\hookrightarrow E_{\mathbb{A}}$. We let $\Lambda$ be a lattice in $E$ and for each place $v\in M_K$ let $\Lambda_v$ denote the $K_v$-lattice in $E_v$ which is spanned by $E$. It is easy to see that for all but finitely many non-archimedean  places $v$, $\Lambda_v=R_v^N$. 

\par Before stating the adelic version of Minkowski's second theorem, we introduce some further notation. For $\lambda\in \mathbb{R}$ and $x=(x_v)\in E_{\mathbb{A}}$ we let $\lambda x\in E$ be given by 
\[(\lambda x)_v:=\begin{cases}
    \lambda x_v & \text{ if }v\text{ is archimedean,}\\
      x_v & \text{ if }v\text{ is non-archimedean.}\\
\end{cases}\]
For $v|\infty$, consider a non-empty, convex, symmetric, bounded subset $S_v$ contained in $E_v$. By symmetric, we mean that $-s\in S_v$ for all $s\in S_v$. Let $\Lambda$ be an $\cO_K$-lattice in $E$ and for each prime $v$, denote by $\Lambda_v$ the completion of $\Lambda$ at $v$. In other words, $\Lambda_v$ is the $R_v$-submodule of $E_v$ which is spanned by $\Lambda$. Consider the subset 
\[S:=\prod_{v|\infty} S_v\times \prod_{v\in M_K^f} \Lambda_v\] of $E_{\mathbb{A}}$. For $n=1, \dots, N$ the $n$-th successive minimum is defined as follows:
\[\lambda_n :=\op{inf}\{t>0\mid t S\text{ contains }n\text{ linearly independent vectors of }\Lambda\text{ over }K\}.\]
Note that $\lambda S\cap E$ is a discrete compact subset of $E_{\mathbb{A}}$ and hence it is finite. We find that 
\[0<\lambda_1\leq \lambda_2\leq \dots \leq \lambda_N <\infty.\]
\begin{theorem}[Minkowski's second theorem]
   With respect to the notation above,
   \[(\lambda_1\dots \lambda_N)^r \op{vol}(S)\leq 2^{rN}\] where  the volume is computed with respect to the Haar measure on $E_{\mathbb{A}}$. Moreover suppose that for every complex place $v$, we have that $\alpha S_v=S_v$ for all $\alpha\in \mathbb{C}$ with $|\alpha|=1$. Let $r_1$ and $r_2$ denote the number of real and complex places of $K$ respectively. Then we have that:
   \[\frac{2^{rN}\pi^{r_2N}}{(N!)^{r_1}((2N)!)^{r_2}}|D_{K/\Q}|^{-N/2}\leq (\lambda_1\dots \lambda_N)^r \op{vol}(S),\] where $D_{K/\Q}$ is the discriminant of $K$ over $\Q$. 
\end{theorem}

\subsection{Reformulations of the subspace theorem}
\par We begin by presenting different versions of the subspace theorem and establishing their equivalence. Given a point $ \mathbf{x} := (x_0, \dots, x_n) \in K^{n+1} $, define  
\[
|\mathbf{x}|_v := \max \{ |x_i|_v \mid i = 0, \dots, n \}.
\]  
Recall that the multiplicative height of the associated projective point is then given by  
\[
H(\mathbf{x}) = \prod_{v \in M_K} |\mathbf{x}|_v.
\]  

\begin{theorem}[Subspace Theorem -- Projective Version]\label{proj subspace thm}
With the above notation, let $ S \subset M_K $ be a finite set of places, and for each $ v \in S $, suppose we are given linear forms $ L_{v,0}, \dots, L_{v,n} \in K_v[X_0, \dots, X_n] $ with coefficients in $ K $. Assume that $ L_{v,0}, \dots, L_{v,n}$ are $K_v$-linearly independent. Let $ \varepsilon > 0 $ be a fixed constant. Then, there exist finitely many proper linear subspaces $ T_1, \dots, T_h $ in $ \mathbb{P}_K^n $ such that all solutions $ \mathbf{x} \in \mathbb{P}^n(K) $ of  
\begin{equation} \label{proj subspace equation}  
    \prod_{v \in S} \prod_{i=0}^n \frac{|L_{v,i}(\mathbf{x})|_v}{|\mathbf{x}|_v} < H(\mathbf{x})^{-n-1-\varepsilon}  
\end{equation}  
are contained in $ T_1 \cup \cdots \cup T_h $.  
\end{theorem} 

It is easy to see that if the above result holds for $S$ and $S'$ is a subset of $S$, then it does also hold for $S'$. Denote by $\cO_{K,S}$ the ring of $S$-integers of $K$. The set $S$ can be enlarged so that it contains all archimedean places and $\cO_{K,S}$ is a principal ideal domain.

\begin{theorem}[Subspace Theorem -- Affine Version]\label{affine subspace theorem}
    Let $S$ be a finite subset of $M_K$ containing the archimedean places and let $ \varepsilon > 0 $ be a fixed constant. Let $ L_{v,0}, \dots, L_{v,n} \in K_v[X_0, \dots, X_n] $ be linear forms with coefficients in $ K $ and assume that $ L_{v,0}, \dots, L_{v,n}$ are $K_v$-linearly independent. Then there exist finitely many hyperplanes $T_1, \dots, T_h$ of $K^{n+1}$ such that all solutions $\bx\in \cO_{K, S}^{n+1}\backslash \{0\}$ of 
    \begin{equation} \label{affine subspace equation}  
    \prod_{v \in S} \prod_{i=0}^n |L_{v,i}(\mathbf{x})|_v < H(\mathbf{x})^{-\varepsilon}  
\end{equation} 
are contained in $T_1\cup \dots \cup T_h$. 
\end{theorem}

We show that the two formulations above are equivalent.

\begin{proposition}\label{affine equiv proj}
    The projective subspace theorem \ref{proj subspace thm} is equivalent to the affine subspace theorem \ref{affine subspace theorem}.
\end{proposition}
\begin{proof}
First, we assume the assertion of the projective subspace theorem \ref{proj subspace thm} and derive the affine version \ref{affine subspace theorem} as a consequence. Let \(\mathbf{x}\) be an \(S\)-integral solution to the inequality 
\[
\prod_{v \in S} \prod_{i=0}^n |L_{v,i}(\mathbf{x})|_v < H(\mathbf{x})^{-\varepsilon}.
\]
Then we have
\[
\prod_{v \in S} \prod_{i=0}^n \frac{|L_{v,i}(\mathbf{x})|_v}{|\mathbf{x}|_v} < \frac{H(\mathbf{x})^{-\varepsilon}}{\left(\prod_{v \in S} |\mathbf{x}|_v\right)^{n+1}}.
\]

Since \(\mathbf{x}\) is \(S\)-integral, it satisfies \(|\mathbf{x}|_v \leq 1\) for all \(v \notin S\). Consequently, we find that 
\[
\prod_{v \in S} |\mathbf{x}|_v \geq \prod_{v \in M_K} |\mathbf{x}|_v = H(\mathbf{x}),
\]
where \(M_K\) is the set of all places of \(K\). Therefore, \(\mathbf{x}\) is a solution to the inequality 
\[
\prod_{v \in S} \prod_{i=0}^n \frac{|L_{v,i}(\mathbf{x})|_v}{|\mathbf{x}|_v} < H(\mathbf{x})^{-n-1-\varepsilon}.
\]
By the projective subspace theorem, there exist finitely many hyperplanes \(T_1, \dots, T_h\) such that all solutions \(\mathbf{x}\) lie on \(T_1 \cup \dots \cup T_h\). This establishes the affine subspace theorem \ref{affine subspace theorem}.

Conversely, suppose the affine subspace theorem \ref{affine subspace theorem} holds. If the projective subspace theorem \ref{proj subspace thm} holds for a set \(S\), it also holds for any subset \(S' \subset S\). Thus, we may enlarge \(S\) to ensure the following conditions:
\begin{itemize}
    \item \(S\) contains all archimedean places,
    \item \(\mathcal{O}_{K, S}\) is a principal ideal domain.
\end{itemize}

Let \(\mathbf{x} \in K^{n+1} \setminus \{0\}\) satisfy the projective Diophantine inequality
\begin{equation}\label{proj ineq}
\prod_{v \in S} \prod_{i=0}^n \frac{|L_{v,i}(\mathbf{x})|_v}{|\mathbf{x}|_v} < H(\mathbf{x})^{-n-1-\varepsilon}.
\end{equation}
Define the fractional ideal \(\mathcal{I} = x_0 \mathcal{O}_{K, S} + \dots + x_n \mathcal{O}_{K, S}\), and let \(\delta \in K^\times\) satisfy \(\mathcal{I} = \delta \mathcal{O}_{K, S}\). For any \(y \in \mathcal{O}_{K, S}\) and \(v \notin S\), we have \(|x_i y|_v = |x_i|_v |y|_v \leq |x_i|_v\). Hence, 
\[
|\mathbf{x}|_v = \max_i |x_i|_v = \max_{z \in \mathcal{I}} |z|_v = |\delta|_v.
\]

Define \(\mathbf{x}' = \delta^{-1}\mathbf{x} = (\delta^{-1}x_0, \dots, \delta^{-1}x_n)\). Then \(|\mathbf{x}'|_v = |\delta|_v^{-1} |\mathbf{x}|_v = 1\) for \(v \notin S\), so \(\mathbf{x}' \in \mathcal{O}_{K, S}^{n+1}\) and \(\prod_{v \in S} |\mathbf{x}'|_v = H(\mathbf{x}')\). Since the projective height is invariant under scalar multiplication, \(H(\mathbf{x}') = H(\mathbf{x})\). From (\ref{proj ineq}), we deduce that 
\[
\prod_{v \in S} \prod_{i=0}^n |L_{v,i}(\mathbf{x}')|_v < H(\mathbf{x}')^{-n-1-\varepsilon} \cdot \left(\prod_{v \in S} |\mathbf{x}'|_v\right)^{n+1} = H(\mathbf{x}')^{-\varepsilon}.
\]

Thus, \(\mathbf{x}'\) satisfies the affine subspace inequality and must lie on \(T_1 \cup \dots \cup T_h\), where \(T_1, \dots, T_h\) are finitely many hyperplanes. Since \(\mathbf{x}'\) is a scalar multiple of \(\mathbf{x}\), it follows that \(\mathbf{x} \in T_1 \cup \dots \cup T_h\). This establishes the projective subspace theorem \ref{proj subspace thm}. 
\end{proof}

\subsection{Approximation domains}

We assume that the affine subspace theorem \ref{affine subspace theorem} is false and derive a contradiction. Assume without loss of generality that $S$ is large enough so that $\cO_{K, S}$ is a principal ideal domain. A solution $\bx\in \cO_{K, S}^{n+1}\backslash \{0\}$ to \eqref{affine subspace equation} is said to be \emph{primitive} if $\prod_{v\in S} |\bx|_v=H(\bx)$. The proof of Proposition \ref{affine equiv proj} shows that every solution to the affine subspace equation \eqref{affine subspace equation} is a scalar multiple of a primitive one. Note that if $\bx$ is a primitive solution and $u$ is an $S$-unit, then $u\bx$ is also a primitive solution. Two primitive solutions $\bx_1$ and $\bx_2$ are said to be equivalent if there exists $u\in \cO_{K,S}^\times$ such that $\bx_2=u \bx_1$. We assume therefore by way of contradiction that there are infinitely many equivalence classes of primitive solutions of \eqref{affine subspace equation} no $(n+1)$ of which lie on the same hyperplane. Let $\bx\in \cO_{K,S}^{n+1}$ be a primitive solution, the affine height of $\bx$ is defined as \[h_{\op{aff}}(\bx):=h((1, x))=\sum_{v\in M_K} \op{max}_{i} \log^+|x_i|_v.\] On the other hand, the height of $\bx$ is $h(\bx)=\sum_{v\in M_K} \op{max}_{i} \log|x_i|_v$.

\begin{lemma}\label{affine to projective inequality lemma}
There exists an $S$-unit $u$ such that 
\[h_{\op{aff}}(u\bx)\leq h(\bx)+C_1,\] where $C_1>0$ is a constant depending on $K$ and $S$.
\end{lemma}

\begin{proof}
    Suppose for example that $x_0\neq 0$. Note that $|x_0|_v\leq 1$ for $v\notin S$ since $x_0\in \cO_{K,S}$. We have by the product formula 
    \[\sum_{v\in S} \log|x_0|_v=-\sum_{v\notin S} \log |x_0|_v\geq 0.\] On the other hand, consider the embedding $\lambda: \cO_{K, S}^\times\rightarrow \mathbb{R}^{|S|}$ defined by $\lambda(u):=(\log |u|_v)_{v}$. By Dirichlet's unit theorem, $\Lambda:=\lambda(\cO_{K, S}^\times)$ is a lattice of the space $\{(t_v)_v\mid \sum_{v\in S} t_v=0\}\subset \mathbb{R}^{|S|}$. Write $S=\{v_1, \dots, v_s\}$ where $s=|S|$ and let $t=(t_1, \dots, t_s)\in \mathbb{R}^{s}$ be defined as follows:
    \[t_i=\begin{cases}
        \log|x_0|_{v_i} & \text{ for }i<s;\\
        -\sum_{i=1}^{s-1} \log |x_0|_{v_i} &\text{ for }i=s,\\
    \end{cases}\]
    and note that $\sum_{i=1}^s t_i=0$. Clearly, there exists a lattice point in $\Lambda$ close to $t$. This means that there exists an $S$-unit $u$ such that for all $i$,
    \[\log |u|_{v_i}+t_i\geq -C_1\] for some positive constant $C_1>0$. In particular, this implies that for all $v\in S$,
    \[\log |u x_0|_v=\log|u|_v+\log |x_0|_v\geq -C_1.\]
    Note that for $v\in S$, 
    \[\log|u\bx|_v=\op{max}_i \log |ux_i|_v\] is either $\geq 0$ or $<0$. If this quantity is $\geq 0$ then 
    \[\log|u\bx|_v=\op{max}_i \log^+|ux_i|_v.\] On the other hand, if it is $<0$, then $\op{max}_i \log^+|ux_i|_v=0$ and we have that 
    \[\log|u\bx|_v\geq \log|ux_0|_v\geq -C_1=-C_1+\op{max}_i \log^+|ux_i|_v.\]
    Thus, in both cases, we have that 
    \[\begin{split}h(\bx)=h(u\bx)=& \sum_{v\in S}\op{max}_i \log^+|ux_i|_v \\ 
    \geq &\sum_{v\in S}\left(-C_1+\op{max}_i \log^+|ux_i|_v\right) \\
    = & -|S|C_1+ \sum_{v\in S} \op{max}_i \log^+|ux_i|_v \\
    = & -|S|C_1 + \sum_{v\in M_K} \op{max}_i \log^+|ux_i|_v \\
    = & -C_1+h_{\op{aff}}(u \bx),
 \end{split}\]
 where we relabel $|S|C_1$ with $C_1$. This completes the proof.
\end{proof}

\begin{proposition}\label{generalized northcott}
    Given $B\geq 0$, there are only finitely many equivalence classes of primitive solutions $\bx$ to the affine subspace inequality \eqref{affine subspace equation} for which $h(\bx)\leq B$. 
\end{proposition}
\begin{proof}
    It follows from Lemma \ref{affine to projective inequality lemma} that there exists $u\in \cO_{K, S}^\times $ such that 
    \[h_{\op{aff}}(u\bx)\leq C_1+B.\]
    Note that \[h_{\op{aff}}(u\bx)=\sum_{v\in M_K} \op{max}_i \log^+|ux_i|_v\geq \op{max}_i\left( \sum_{v\in M_K} \log^+|ux_i|_v\right)=\op{max}_i h(u x_i).\]
    Thus each coordinate $u x_i$ has height at most $(C_1+B)$ and the conclusion follows from Northcott's theorem.
\end{proof}

Before proceeding, let's set up some further notation. Writing 
\[L_v(\bx)=(L_{v, 0}(\bx), \dots, L_{v, n}(\bx))=(L_{v, 0}(x_0, \dots, x_n), \dots, L_{v, n}(x_0, \dots, x_n))\]
we have a linear isomorphism $L_v: K^{n+1}\xrightarrow{\sim} K^{n+1}$. We write $L_{v,i}(\bx)=\sum_{j=0}^n a_{i,j}^{(v)} x_j$, where $a_{i,j}^{(v)}$ are in $K$. Thus we may identify $L_v=(a_{i,j}^{(v)})_{i,j}$. Setting $|L_v|_v:=\op{max}_{i, j} |a_{i,j}^{(v)}|_v$, and $|L_{v,i}|_v:=\op{max}_j |a_{i,j}^{(v)}|_v$, we find that 
\[|L_{v,i}(\bx)|_v=|\sum_{j=0}^n a_{i,j}^{(v)} x_j|_v\leq \begin{cases}(n+1) |L_{v,i}|_v|\bx|_v &\text{ if }v\text{ is archimedean};\\
|L_{v,i}|_v|\bx|_v &\text{ if }v\text{ is non-archimedean};
\end{cases}\]
and therefore we deduce that 
\[|L_v(\bx)|_v\leq \begin{cases}(n+1) |L_{v}|_v|\bx|_v &\text{ if }v\text{ is archimedean};\\
|L_{v}|_v|\bx|_v &\text{ if }v\text{ is non-archimedean}.
\end{cases}\]
In particular, we find that for non-archimedean $v$,
\[\log|L_v(\bx)|_v-\log|\bx|_v\leq \log |L_v|_v\]
and 
\[\log|L_v^{-1}(\bx)|_v-\log|\bx|_v\leq \log |L_v^{-1}|_v.\]
In the second equation above, replace $\bx$ with $L_v(\bx)$ to deduce that 
\[\log|\bx|_v-\log|L_v(\bx)|_v\leq \log |L_v^{-1}|_v.\]
Thus, we have proved the following result. 

\begin{lemma}\label{gamma v lemma}
    For each prime $v\in M_K$, there is a constant $\gamma_v\ge 0$ such that
    \[|\log|L_v(\bx)|_v-\log|\bx|_v|\leq \gamma_v\]
    where $\gamma_v=0$ for all but finitely many places $v\in M_K$.
\end{lemma}

In view of Proposition \ref{generalized northcott} we may assume without loss of generality that $h(\bx)>0$ since there are only finitely many equivalence classes of primitive solutions for which $h(\bx)\leq 0$.
\begin{lemma}\label{affine proj height lemma}
    With respect to the notation above, there is a positive constant $C_2>0$ for which the following assertions hold:
    \begin{enumerate}
        \item every equivalence class of primitive solutions contains an element $\bx$ such that 
        \[h_{\op{aff}}(L_v(\bx))\leq h(\bx)+C_2.\]
        \item If $L_{v,i}(\bx)\neq 0$ then $|\log|L_{v,i}(\bx)|_v|\leq h(\bx)+C_2$.
    \end{enumerate}
\end{lemma}
\begin{proof}
    \par From Lemma \ref{gamma v lemma} it follows that $h_{\op{aff}}(L_v(\bx))$ and $h_{\op{aff}}(\bx)$ differ by a bounded quantity. Therefore, there is an element $\bx$ in the equivalence class for which 
    \[h_{\op{aff}}(L_v(\bx))\leq h_{\op{aff}}(\bx)+C_2'\leq h(\bx)+C_2'+C_1, \] where in the second inequality, we have appealed to Lemma \ref{affine to projective inequality lemma}. Setting $C_2:=C_2'+C_1$, we get the statement of part (1). 
    \par By the fundamental inequality (Proposition \ref{Fundamental inequality}) applied to $S:=\{v\}$ and $\alpha:=L_{v, i} (\bx)$, we deduce that
    \[|\log |L_{v, i} (\bx)|_v|\leq \sum_{w\in M_K} \log^+ |L_{v, i} (\bx)|_w\leq h_{\op{aff}}(L_v(\bx))\]
    and by part (1), we have that $h_{\op{aff}}(L_v(\bx))\leq h(\bx)+C_2$. This proves part (2). 
\end{proof}
Part (2) of the above lemma implies that every equivalence class of primitive solutions has a representative for which the point 
\[P^{\bx}:=\left(\frac{\log |L_{v,i}(\bx)|_v}{h(\bx)}\right)_{v,i}\] lies in $[-(1+C_2/h(\bx)), (1+C_2/h(\bx))]^{|S|(n+1)}$. Hence by Proposition \ref{generalized northcott}, for all but finitely many equivalence classes of solutions $\bx$, we have that $P^{\bx}\in [-2,2]^{|S|(n+1)}$. We subdivide $[-2,2]^{|S|(n+1)}$ into boxes of side length $1/N$ for a large value of $N$. By the pigeon hole principal there are infinitely many equivalence classes of solutions which lie in one of these boxes. Let $(c_{v,i})_{v, i}$ be the "north-east" corner point of this box, which is to say that the following inequalities are satisfied:
\begin{equation}\label{bounds for point constucted}(c_{v,i}-1/N)\leq \frac{\log |L_{v,i}(\bx)|_v}{h(\bx)}\leq c_{v,i}\end{equation} for all $v, i$. On the other hand, the inequality \eqref{affine subspace equation} gives 
\[\sum_{v\in S} \sum_{i=0}^n \frac{\log |L_{v,i}(\bx)|_v}{h(\bx)}<-\varepsilon .\]
This implies that
\[\sum_{v\in S}\sum_{i=0}^n (c_{v,i}-1/N)< -\varepsilon\]
i.e., 
\[\sum_{v\in S}\sum_{i=0}^n c_{v,i}\leq \frac{|S|(n+1)}{N}-\varepsilon.\]
Assume that $N\geq \frac{2|S|(n+1)}{\varepsilon}$, and we arrive at the following:
\begin{equation}\label{c v i neg bound}\sum_{v\in S}\sum_{i=0}^n c_{v,i}\leq -\frac{\varepsilon}{2}.\end{equation}
Recall that $K_{\mathbb{A}}$ is the ring of adeles of $K$. For $v\in S$ and $Q\geq 1$, we define the $v$-adic approximation domain:
\[\Pi_v(Q):=\{\mathbf{y}=(y_0, \dots, y_n)\in K_v^{n+1}\mid |L_{v, i}(\mathbf{y})|_v\leq Q^{c_{v,i}}\text{ for }i=0, \dots, n\}.\]
Observe that $\Pi_v(Q)$ is the inverse image 
\[L_v^{-1}\left(\{\mathbf{y}=(y_0, \dots, y_n)\in K_v^{n+1}\mid |y_i|_v\leq Q^{c_{v,i}}\text{ for }i=0, \dots, n\}\right)\]
and thus we see that:
\begin{itemize}
    \item if $v$ is archimedean then $\Pi_v(Q)\subset \mathbb{R}^{n+1}$ or $\mathbb{C}^{n+1}$ is a compact, convex and symmetric subset. 
    \item For non-achimedian $v$, $\Pi_v(Q)$ is a lattice in $K_v^{n+1}$.
\end{itemize}
It follows from \eqref{bounds for point constucted} that 
\[|L_{v,i}(\bx)|_v\leq H(\bx)^{c_{v,i}}\]
and this motivates the following definition.
\begin{definition}
    The approximation domain is defined in $E_{\mathbb{A}}:=K_{\mathbb{A}}^{n+1}$ by 
    \[\Pi(Q):=\prod_{v\in S} \Pi_v(Q)\times \prod_{v\notin S} R_v.\]
\end{definition}

We set $E_K:=K^{n+1
}$ and $\iota: E_K\rightarrow E_{\mathbb{A}}$ be the embedding defined by $\iota(z):=(z)_{v\in M_K}$. We have that $\iota(\bx)\in \Pi(H(\bx))$.
\par We now compute the volume of $\Pi(Q)$. Recall that  $r=[K:\Q]$. For each prime $v\in M_K$ let $\mu_v$ be the Haar measure on $K_v$ normalized as follows. If $v$ is real (resp. complex) archimedean place $\mu_v$ is the Lebesgue measure (resp. twice the Lebesgue measure). For a non-archimedean place $v$, $\mu_v$ is normalized by setting $\mu_v(R_v):=|D_{K_v/\Q_p}|^{1/2}$. Take $\beta_v$ to denote the measure on $K_v^{n+1}$ defined by $\beta_v:=\mu_v^{n+1}$. Let $\beta$ be the product measure $\prod_{v\in M_K} \beta_v$ on $E_{\mathbb{A}}$. The volume of a region in $E_{\mathbb{A}}$ (resp. $K_v^{n+1}$) is defined to be the measure with respect to $\beta$ (resp. $\beta_v$). We set $\Delta_v$ to denote the determinant of the linear transformation $L_v$.
\begin{lemma}\label{local vol comps}
    With respect to notation above, the following assertions hold:
    \begin{enumerate}
        \item For $v\notin S$, we have that $\op{Vol}\left(\Pi_v(Q)\right)=|D_{K_v/\Q_p}|^{\frac{n+1}{2}}$. 
        \item For $v\in S$ a non-archimedean prime, 
        \[\op{Vol}\left(\Pi_v(Q)\right)=|D_{K_v/\Q_p}|^{\frac{n+1}{2}}|\Delta_v|_v^{-r}Q^{r\left(\sum_i c_{v,i}\right)}.\]
        \item If $v$ is real, then 
        \[\op{Vol}\left(\Pi_v(Q)\right)=2^{n+1}|\Delta_v|_v^{-r}Q^{r\left(\sum_i c_{v,i}\right)}.\]
        \item If $v$ is complex, then 
        \[\op{Vol}\left(\Pi_v(Q)\right)=(2\pi)^{n+1}|\Delta_v|_v^{-r}Q^{r\left(\sum_i c_{v,i}\right)}.\]
    \end{enumerate}
\end{lemma}
\begin{proof}
    For part (1), we note that \[\op{Vol}(\Pi_v(Q))=\mu_v(R_v)^{n+1}=|D_{K_v/\Q_p}|^{\frac{n+1}{2}}.\]
    For part (2), let 
    \[B_v:=\{\mathbf{y}=(y_0, \dots, y_n)\in K_v^{n+1}\mid |y_i|_v\leq Q^{c_{v,i}}\}\]
    and note that 
    \[\op{Vol}(B_v)=Q^{r\left(\sum_i c_{v, i}\right)}\mu_v(R_v)^{n+1}=Q^{r\left(\sum_i c_{v, i}\right)}|D_{K_v/\Q_p}|^{\frac{n+1}{2}}.\] Note that $\Pi_v(Q)=L_v^{-1}(B_v)$ and hence \[\op{Vol}(\Pi_v(Q))=|\Delta_v|_v^{-r}\op{Vol}(B_v)=|\Delta_v|_v^{-r}Q^{r\left(\sum_i c_{v, i}\right)}|D_{K_v/\Q_p}|^{\frac{n+1}{2}}.\]
    Parts (3) and (4) are similar to part (2) and their proofs are omitted.
\end{proof}
\begin{proposition}\label{Vol(Pi(Q))<< Q to the -epsilon/2}
    Let $Q\geq 1$ and $\Pi(Q)$ be the approximation domain associated with $(c_{v,i})$. Then, the volume of $\Pi(Q)$ satisfies the asymptotic 
    \[\op{Vol}(\Pi(Q))\ll Q^{-r\varepsilon/2}.\]
\end{proposition}
\begin{proof}
    It follows from Lemma \ref{local vol comps}
    \[\op{Vol}(\Pi(Q))\ll Q^{r\left(\sum_{v,i} c_{v,i}\right)}.\] Recall that \eqref{c v i neg bound} asserts that $\sum_{v,i} c_{v,i}\leq -\varepsilon/2$ and the result follows.
\end{proof}
Since we assume that for a choice of $\varepsilon>0$, the affine subspace equation \eqref{affine subspace equation} is false, there exists an infinite set of solutions $\mathcal{X}=\{\bx_i\}$ such that no $(n+1)$ solutions lie on the same hyperplane in $K^{n+1}$. 
Setting $Q_i = H(\bx_i)$ we note that $\iota(\bx_i)\in \Pi(Q_i)$. Let $V(Q)$ be the $K$-vector space spanned by the vectors in $\iota^{-1}(\Pi(Q))$. The rank of $\Pi(Q)$ is defined to be $\op{rank}_K V(Q)$. Note that by construction $\bx_i\in V(Q_i)$, where $Q_i:=H(\bx_i)$.
\begin{lemma}\label{rank Pi is <=n}
    For all but finitely many $\bx_i\in \mathcal{X}$, we have that $\op{rank}\Pi(Q_i)\leq n$. 
\end{lemma}
\begin{proof}
    Let $\bx\in \mathcal{X}$, $Q:=H(\bx)$ and set $\Lambda:=\iota(\cO_{K, S}^{n+1})$. Let $\lambda_1, \dots, \lambda_{n+1}$ be the successive minima of $\Pi(Q)$ and let $x^{(1)}, \dots, x^{(n+1)}\in \cO_{K, S}^{n+1}$ be linearly independent points defining the successive minima. This is to say that $\lambda_i \Pi(Q)\cap \Lambda$ is generated by $x^{(1)}, \dots, x^{(i)}$. We find that $V(Q)$ is spanned by $x^{(1)}, \dots, x^{(j)}$ where $j=\op{max}\{i\mid \lambda_i\leq 1\}$. Since $\lambda_{n+1}\geq \lambda_i$, one deduces that 
    \[\lambda_{n+1}^{r(n+1)}\geq (\lambda_1\dots \lambda_{n+1})^r.\]
    On the other hand, by Minkowski's second theorem, we find that 
    \[\frac{2^{r(n+1)}\pi^{r_2(n+1)}}{((n+1)!)^{r_1}(2(n+1)!)^{r_2}} |D_{K/\Q}|^{-\frac{(n+1)}{2}}\leq (\lambda_1\dots \lambda_{n+1})^r \op{Vol}(\Pi(Q)).\]
    Proposition \ref{Vol(Pi(Q))<< Q to the -epsilon/2} asserts that  \[\op{Vol}(\Pi(Q))\ll Q^{-r\varepsilon/2}.\]
    Thus we find that 
    $\lambda_{n+1}\gg Q^{\frac{\varepsilon}{2(n+1)}}$ and therefore when $Q$ is large, $\lambda_{n+1}>1$ and thus, $\op{dim}V(Q)\leq n$ for large enough values of $Q$. It then follows from Proposition \ref{generalized northcott} that the set of $\bx_i \in \mathcal{X}$ for which $\op{rank}\Pi(Q_i)=n+1$ is finite. 
\end{proof}

In view of Lemma \ref{rank Pi is <=n} we assume that for all $\bx_i\in \mathcal{X}$, the vector space $V(Q_i)$ has $K$-dimension $n$ and moreover, that $V(Q_i)\neq V(Q_j)$ for $i\neq j$. Using generalizations of Roth's machinery, we prove the following result.
\begin{theorem}\label{theorem 4.14}
    For all but finitely many $\bx_i\in \mathcal{X}$, we have that $\op{rank}\Pi(Q_i)\leq n-1$.
\end{theorem}

\subsection{Auxiliary polynomials} Let \( m \geq 1 \) be an integer, and let \( \mathbf{x}_1, \dots, \mathbf{x}_m \in \mathcal{X} \). Assume that for each \( i = 1, \dots, n \), the vector space \( V(Q_i) \) has dimension \( \dim_K V(Q_i) = n \). For \( i = 1, \dots, n \), let \( E_i := K^{n+1} \) represent the ambient space in which \( V(Q_i) \) is contained. We aim to construct an auxiliary polynomial on the Cartesian product \( (K^{n+1})^m = E_1 \times \cdots \times E_m \), ensuring it has controlled height and prescribed vanishing properties.

\par We introduce a multi-vector variable \( \mathbf{X} = (\mathbf{X}_1, \dots, \mathbf{X}_m) \), where each \( \mathbf{X}_h = (X_{h,0}, \dots, X_{h,n}) \) represents the \( h \)-th coordinate vector in \( K^{n+1} \). This notation allows us to describe points in the product space \( (K^{n+1})^m \) succinctly. Given any vector \( \mathbf{u} = (u_0, \dots, u_n) \) with non-negative components, define its norm by  
\[
|\mathbf{u}| := \sum_{i=0}^n u_i.
\]
Now, fix a tuple \( d = (d_1, \dots, d_m) \), where each \( d_h \) is a non-negative integer representing the degree of the desired polynomial in the variables corresponding to the \( h \)-th factor of the product \( (K^{n+1})^m \). Define \( \mathcal{A}(d) \) as the set of integer multi-indices \( I = (\mathbf{i}_1, \dots, \mathbf{i}_m) \), where each  
\[
\mathbf{i}_h = (i_{h,0}, \dots, i_{h,n}) \quad \text{with} \quad i_{h,j} \in \mathbb{Z}_{\geq 0}
\]
is a vector of non-negative integers satisfying the degree condition  
\[
|\mathbf{i}_h| = d_h \quad \text{for all } h = 1, \dots, m.
\]
In other words, \( \mathcal{A}(d) \) consists of all tuples of multi-indices whose total degrees in each set of variables match the specified degrees in \( d \).

For each multi-index \( I \in \mathcal{A}(d) \), define the corresponding monomial  
\[
\mathbf{X}^I := \mathbf{X}_1^{\mathbf{i}_1} \cdots \mathbf{X}_m^{\mathbf{i}_m},
\]
where the notation \( \mathbf{X}_h^{\mathbf{i}_h} \) represents the monomial  
\[
\mathbf{X}_h^{\mathbf{i}_h} := X_{h,0}^{i_{h,0}} \cdots X_{h,n}^{i_{h,n}}.
\]
Thus, \( \mathbf{X}^I \) denotes a product of monomials, one for each coordinate vector \( \mathbf{X}_h \), with the degrees specified by the corresponding multi-index \( \mathbf{i}_h \).

Finally, let \( P(d) \) denote the \( K \)-vector space of multi-homogeneous polynomials of degree \( d \) in the variables \( \mathbf{X} = (\mathbf{X}_1, \dots, \mathbf{X}_m) \). A polynomial \( P \in P(d) \) is a \( K \)-linear combination of the monomials \( \mathbf{X}^I \) with multi-indices \( I \in \mathcal{A}(d) \), meaning  
\[
P(\mathbf{X}) = \sum_{I \in \mathcal{A}(d)} c_I \mathbf{X}^I, \quad c_I \in K.
\]
To facilitate manipulation of these polynomials, we define partial derivatives with respect to each multi-index \( I \in \mathcal{A}(d) \). Specifically, set 
\[
\partial_I := \frac{1}{\mathbf{i}_1! \cdots \mathbf{i}_m!} \left( \frac{\partial}{\partial \mathbf{X}_1} \right)^{\mathbf{i}_1} \cdots \left( \frac{\partial}{\partial \mathbf{X}_m} \right)^{\mathbf{i}_m},
\]
where \( \mathbf{i}_h! := i_{h,0}! \cdots i_{h,n}! \) for each \( h \). For a polynomial \( f \in P(d) \), we observe the expansion
\[
f(\mathbf{X}) = \sum_{I \in \mathcal{A}(d)} (\partial_I f)(0) \mathbf{X}^I,
\]
expressing \( f \) in terms of its partial derivatives evaluated at \( \mathbf{X} = 0 \).
We shall write 
\[\begin{split}L_v(\bbX_h):=&\left(L_{v,0}(\bbX_h), \dots, L_{v,i} (\bbX_h), \dots, L_{v,n} (\bbX_h)\right)\\
=& \left(\sum_{j=0}^n a_{0,j}^{(v)} X_{h,j}, \dots, \sum_{j=0}^n a_{i,j}^{(v)} X_{h,j}, \dots, \sum_{j=0}^n a_{n,j}^{(v)} X_{h,j}\right)\end{split}\]
for linearly independent linear forms $L_{v,i } $ over $K$.
We write
\[\partial_I P(\bbX_1, \dots, \bbX_h, \dots, \bbX_m)=\sum_{J\in \mathcal{A}(d-I)} a(L_v; J;I) L_v(\bbX_1)^{\mathbf{j}_1}\dots L_v(\bbX_m)^{\mathbf{j}_m}. \]
For $i=0, \dots, n$ we write $J_i=(j_{1, i}, \dots, j_{m, i})$. Given a vector $v=(v_1, \dots, v_m)$ set
\[\left(\frac{v}{d}\right):=\frac{|v_1|}{d_1}+\dots+ \frac{|v_m|}{d_m}. \] Here, the $v_i$ could also consist of vectors.
\begin{lemma}\label{lemma 4.15}Given the notation above, let \[ 0 < \eta \leq \frac{2}{n+1} \] and define \( r := [K:\Q] \). Assume that \[m \geq \frac{4}{(n+1)(n+2)\eta^2} \log\big(2(n+1)r|S|\big).\] Then, for all sufficiently large \( d_1, \dots, d_m \), there exists a nonzero polynomial \( P \in P(d) \) satisfying the following height bounds:  
\[
h(P) \leq C_3|d| \quad \text{and} \quad h\big(a\big(L_v; J, I\big)\big) \leq C_3|d|,
\]  
where \( C_3 \) is a positive constant that depends only on \( K \) and the linear forms \( L_{v,i} \).  

Furthermore, \( a\big(L_v; J; I\big) = 0 \) for \( v \in S \) whenever \( J \) and \( I \) satisfy the conditions:  
\[
\left(\frac{J_i}{d}\right) \leq \frac{m}{n+1} - 2m\eta \quad \text{or} \quad \left(\frac{J_i}{d} \right)\geq \frac{m}{n+1} + 2mn\eta \quad \text{for some } i = 0, \ldots, n,
\]  
and $\left(\frac{I}{d}\right) \leq m\eta$.
\end{lemma}
\begin{proof}
    The proof can be viewed as significant generalization to Lemma \ref{Step 1}. The dimension of the vector space $P(d)$ is the cardinality of $\mathcal{A}(d)$, i.e., the number of $m$-tuples $J=\left(\mathbf{j}_{1}, \ldots, \mathbf{j}_{m}\right)$ for which $|\mathbf{j}_h|=d_h$ for $h=1, \dots, m$. Let $V_0'$ be the region defined by 
    \[V_0'=\{(x_{h, i})\in \mathbb{R}^{(n+1)m} \mid 1\leq h \leq m, 0\leq i\leq n, x_{h,i}\geq 0, \sum_{i=0}^n x_{h, i} =d_h\}.\]
    Notice that $\mathcal{A}(d)$ consists of the lattice points in the region $V_0'$. As $d_1, \dots, d_m\rightarrow \infty$, we find that $\# \mathcal{A}(d)\sim \op{Vol}(V_0')$. By change of variables, $\op{Vol}(V_0')=(d_1\dots d_m)^{n} \op{Vol}(V_0)$ where 
    \[V_0:=\{(x_{h, i})\in \mathbb{R}^{(n+1)m} \mid 1\leq h \leq m, 0\leq i\leq n, x_{h,i}\geq 0, \sum_{i=0}^n x_{h, i} =1\}.\]
    We write $V_0=\prod_{h=1}^m V_0^{(h)}$ where \[V_0^{(h)}=\{(x_{h,i})\in \mathbb{R}^{n+1}\mid x_{h, i}\geq 0\text{ and }\sum_{i=0}^n x_{h, i}=1 \}.\]
    An elementary exercise in multivariable calculus gives us that $\op{Vol}(V_0^{(h)})=\frac{1}{n!}
    $. Putting these steps together, we find that 
    \[\op{dim} P(d)\sim (d_1\dots d_m)^{n} \left(\frac{1}{n!}\right)^{m}.\] Let us now fix $v\in S$ and compute the asymptotic upper bound for the number of linear conditions with coefficients in $K$ for 
    \[\left(\frac{J_i}{d}\right)\leq \frac{m}{n+1}-m\eta\] for some fixed $i$. We work out the computation for $i=0$ and also set $I=0$. The number of such $m$-tuples is asymptotic to $(d_1\dots d_m)^{n} \op{Vol}(V)$ where $V$ is defined as follows:
    \[V:=\left\{(x_{h,i})\in \mathbb{R}^{(n+1)m}\mid \sum_{i=0}^n x_{h, i}=1, \sum_{h=1}^m x_{h, 0}\leq \left(\frac{m}{n+1}\right)-m\eta\right\}\]
    \[\chi_{[a,b]}(x):=\begin{cases}
        1 &\text{ of }x\in [a, b];\\
        0 &\text{ of }x\notin [a, b].
    \end{cases}\]
    The volume of $V$ is expressed as the following integral:
    \[\op{Vol}(V)=\int_{0}^{1} \cdots \int_{0}^{1} \chi_{\left[0, \frac{m}{n+1}-m \eta\right]}\left(\sum_{h=1}^{m} x_{h 0}\right)\left\{\prod_{h=1}^{m} \chi_{\left[0,1-x_{h 0}\right]}\left(\sum_{i=1}^{n-1} x_{h i}\right)\right\} \prod_{h=1}^{m} \prod_{i=0}^{n-1} d x_{h i} .\]
    In the integral above, we think of $x_{h, i}$ as corresponding to $j_{h,i}/d_i$. The variable $x_{h, n}$ is given by $1-\sum_{i=0}^{n-1} x_{h, i}$ and is thus determined by $x_{h,i}$ for $i\leq n-1$. Noting that
    \[\int_{0}^1 \dots \int_0^1 \chi_{[0, 1-y_0]}\left(\sum_{i=1}^{n-1} y_i\right) dy_1\dots dy_n  =\frac{(1-y_0)^{n-1}}{(n-1)!}.\]
    Therefore, we have
$$
\op{Vol}(V)=\int_{0}^{1} \cdots \int_{0}^{1} \chi_{\left[0, \frac{m}{n+1}-m \eta\right]}\left(\sum_{h=1}^{m} x_{h 0}\right) \prod_{h=1}^{m} \frac{\left(1-x_{h 0}\right)^{n-1}}{(n-1)!} d x_{10} \cdots d x_{m 0}.
$$
In order to obtain a good upper bound for $\op{Vol}(V)$ we shall make use of the fact that $$
\chi_{[a, b]}(x) \leq e^{\lambda \cdot(b-x)},
$$
which is valid for $\lambda \geq 0$. This decouples the variables and we obtain:
\[\op{Vol}(V)\leq \left(e^{\lambda /(n+1)-\lambda \eta} \int_{0}^{1} e^{-\lambda x} \frac{(1-x)^{n-1}}{(n-1)!} d x\right)^{m}\] for any $\lambda\geq 0$. Expand $e^{-\lambda x}$ into a Maclaurin series at $x=1$ to obtain:
\[\int_{0}^{1} e^{-\lambda x} \frac{(1-x)^{n-1}}{(n-1)!} d x  =\sum_{k=0}^{\infty} \frac{(-\lambda)^{k}}{(n+k)!}.\]
Now assume that $\lambda\in (0, n+4]$ to arrive at the following inequalities:
\[\begin{aligned}
& \sum_{k=0}^{\infty} \frac{(-\lambda)^{k}}{(n+k)!} \\
& \leq \frac{1}{n!}\left\{1-\frac{\lambda}{n+1}+\frac{1}{(n+1)(n+2)} \lambda^{2}\right\} \\
& <\frac{1}{n!} \exp \left(-\frac{\lambda}{n+1}+\frac{1}{(n+1)(n+2)} \lambda^{2}\right).
\end{aligned}\]
We choose $\lambda:=\frac{\eta (n+1)(n+2)}{2}$ with $\eta\leq \frac{2}{(n+1)}$ and get:
\[\op{Vol}(V)<\frac{1}{(n!)^{m}} \op{exp}\left(\frac{-(n+1)(n+2) \eta^{2} m }{4}\right).\]
The number of $(i,v)$ is $|S|(n+1)$. Let $M$ denote the number of equations and $N$ be the number of variables. We shall insist on the inequality $\frac{rM}{N}\leq \frac{1}{2}$. Therefore we have that as $d_1, \dots, d_m\rightarrow \infty$, the following asymptotic:
\[\frac{rM}{N}=\frac{r(d_1\dots d_m)^{n} \op{Vol}(V) (n+1)|S|}{(d_1\dots d_m)^{n}\op{Vol}(V_0)}\leq \frac{1}{2},\]
which is to say that \[\op{exp}\left((n+1)(n+2)\frac{\eta^2m}{4}\right)\geq 2(n+1)|S|r.\]
Thus we require that 
\[m\geq \frac{4}{(n+1)(n+2)\eta^2}\log\left(2(n+1)|S|r\right).\]
By an application of Siegel's lemma, the auxiliary polynomial satisfies \[h(P)\leq C_3|d|.\] This part of the proof is similar to the construction in Roth's theorem. 
Note that the vanishing of $a(L_v; J;I)$ whenever $\left(\frac{J_i}{d}\right)\leq \frac{m}{(n+1)}-m\eta$ for some $i=0, \dots, n$ implies that $a(L_v; J;I)=0$ whenever $\left(\frac{J_i}{d}\right)\leq \frac{m}{(n+1)}-2m\eta$ for some $i=0, \dots,  n$ and every $I$ such that $\left(\frac{I}{d}\right)\leq m\eta$. For $\left(\frac{I}{d}\right)\leq m\eta$ the vanishing of $a(L_v; J;I)$ for $\left(\frac{J_i}{d}\right)\geq \frac{m}{(n+1)}+2nm\eta$ for all $i$ is automatic. This is because $a(L_v; J;I)\neq 0$ only if $\left(\frac{J_i}{d}\right)>\frac{m}{(n+1)}-2m \eta$ for every $i$. We find that in this case,
\[\begin{split}\left(\frac{J_i}{d}\right)=& \sum_{j=0}^n \left(\frac{J_j}{d}\right)-\sum_{j\neq i} \left(\frac{J_j}{d}\right)\\
<& \sum_{j=0}^n \left(\frac{J_j}{d}\right)-\left(\frac{m}{(n+1)}-2m \eta\right)n \\
\leq & m-\frac{mn}{(n+1)}+2mn \eta=\frac{m}{(n+1)}+2mn\eta
\end{split}\]
for every $i$. This completes the proof of the result.
\end{proof}
\subsection{A generalized Roth lemma}
Let $V$ be vector space over $\bar{\Q}$ of dimension $(n+1)$ and let $e_0, \dots, e_n$ be a prescribed basis for $V$. Given an integer $m$ in the range $1\leq m \leq (n+1)$, the $m$-fold exterior product $\wedge^m V$ is a vector space of $\bar{\Q}$ with basis vectors $e_{i_1}\wedge e_{i_2}\wedge \dots \wedge e_{i_m}$, where $0\leq i_1< i_2<\dots <i_m \leq n$. Thus $\wedge^m V$ is a vector space over $\bar{\Q}$ of dimension $\binom{n+1}{m}$. Let $W$ be an $m$-dimensional $\bar{\Q}$-subspace of $V$. Then, $\wedge^m W$ is a $1$-dimensional subspace of $\wedge^m V$. Therefore we may view $W$ as a point $P_W$ in the projective space $\mathbb{P}(\wedge^m V)=\mathbb{P}^{\binom{n+1}{m}-1}$. We define the height of $W$ as follows $h(W):=h(P_W)$.

\begin{definition}
    Let $m$ and $n$ be positive integers with $m\leq n$. Let $A$ be an $n\times m$ matrix with entries in $\bar{\Q}$. Then the height $h(A)$ is defined to be the height of the column span of $A$. We also set $h(A^t):=h(A)$. 
\end{definition}

Let $A$ be an $n\times m$ matrix with entries in $\bar{\Q}$ and let $I=\{i_1, \dots, i_m\}$ be a subset of $\{0, \dots, n\}$ with $i_1<i_2<\dots <i_m$. Let $e_I$ denote the basis vector $e_{i_1}\wedge \dots \wedge e_{i_m}$ and let $A_I$ be the $m\times m$ submatrix of $A$ formed by the $i$-th rows for $i\in I$. Consider the line passing through the origin in $\wedge^m V$ and the point \[S_A:=\sum_{\substack{I=\{i_1,\dots, i_m\}\\
0\leq i_1<i_2<\dots<i_m\leq n}} \op{det}(A_I) e_I.\] This defines a point in $\mathbb{P}(\wedge^m V)$ corresponding to $A$. 
\begin{proposition}\label{W Wperp}
   Given an $m$-dimensional subspace $W$ of $V$ set \[W^\perp:=\{w\in V^*\mid w(v)=0\text{ for all }v\in W\}.\] Then one has $h(W)=h(W^\perp)$. 
\end{proposition}

\begin{proof}
    Given $x\in \wedge^m V$, we have a linear map $y\mapsto x\wedge y$ in $\op{Hom}(\wedge^{n+1-m} V, \wedge^{n+1} V)\simeq \wedge^{n+1-m} V^*$. In fact, this gives an isomorphism 
    \[\wedge^{m} V\xrightarrow{\sim} \wedge^{n+1-m} V^*\] and its inverse can be defined in a similar way. It is easy to see that the line $\wedge^m W$ is mapped to the line $\wedge^{n+1-m} W^\perp$ under this transformation, and it follows that $h(W)=h(W^\perp)$.
\end{proof}
\begin{corollary}
    Given an $m\times n$ matrix $A$ with coordinates in $\bar{\Q}$ and $m\leq n$, the height of the space of solutions to $Ax=0$ equals $h(A)$.
\end{corollary}
\begin{proof}
    The rows of $A$ form a basis of $W:=\{x : A x=0\}^\perp$ and thus by definition, $h(A)=h(W)$. On the other hand, $h(W)=h(W^\perp)$ by Proposition \ref{W Wperp}. Thus we deduce that $h(A)=h(\{x : A x=0\})$, which completes the proof.
\end{proof}
Let $M_1, \dots, M_m\in \bar{\Q}[x_0, \dots, x_m]$ be non-zero linear forms and let $M=(M_1, \dots, M_m)$ and $d=(d_1, \dots, d_m)$. We shall also recall that $\bbX_h$ denotes the variable multivector $(X_{h,0}, \dots, X_{h, n})$ for $h=1, \dots, m$.

\begin{definition}
    Denote by $\mathcal{I}(t; d; M)$ the ideal in $\bar{\Q}[\bbX_1, \dots, \bbX_m]$ generated by the monomials $M_1(\bbX_1)^{j_1}\dots M_m(\bbX_m)^{j_m}$ with $\left(\frac{J}{d}\right)=\frac{j_1}{d_1}+\dots+ \frac{j_m}{d_m}\geq t$ . Then for a multihomogeneous polynomial $P$ of degree $d$, set 
    \[\op{Ind}(P;d;M):=\op{sup}\{t\geq 0\mid P\in \mathcal{I}(t; d; M)\}.\]
\end{definition}
   
    \begin{remark}\label{rem on Roth lemma}We remark that when $n=1$, $M_j:=x_1-\alpha_j x_0$ and $Q$ is the dehomogenized polynomial defined by:
\[Q(x_1, \dots, x_m):=P(1, x_1, \dots, 1, x_m),\] one has that $\op{Ind}(P; d;M)=\op{Ind}(Q;d;\alpha)$. We are now in a position to state the generalized Roth Lemma. 
\end{remark}
\begin{lemma}
 Let $P$ be a multihomogeneous polynomial of degree $d$ and $\sigma\in (0, 1/2]$. Moreover suppose that:
    \begin{enumerate}
        \item $d_{j+1}/d_j\leq \sigma$ for $j=1, \dots, m-1$;
        \item $\op{min}_j d_j h(M_j)\geq n \sigma^{-1}(h(P)+4m d_1)$. 
    \end{enumerate}
    Then, $\op{Ind}(P;d;M)\leq 2m \sigma^{\frac{1}{2^{m-1}}}$. 
\end{lemma}
\begin{proof}
    The result is obtained from the classical version of Roth's lemma upon specializing all variables except $(X_{h, 0}, X_{h,1})$ to $0$. By Remark \ref{rem on Roth lemma}, when $n=1$, the statement is equivalent to Roth's lemma. Thus assume without loss of generality that $n\geq 2$. Let $b=(b_0,\dots, b_n)\in K^{n+1}\backslash\{0\}$ and assume for simplicity that $b_0\neq 0$. Then we have that 
    \[\begin{split}h(b)=& \sum_{v\in M_K} \log \op{max}_i \{|b_i|_v\}\\ 
    \leq & \sum_{v\in M_K} \sum_{i=1}^n \log \op{max}\{|b_0|_v, |b_i|_v\}\\
    \leq & n \op{max}_i h((b_0, b_i)).
    \end{split}\]
    We relabel the variables and write \[M_j(\bbX_j)=b_{j,0} X_{j, 0}+\dots+ b_{j, n} X_{j, n}\] with $b_{j, i}\in K$ and $b_{j, 0}\neq 0$ for all $j$. Now specialize $X_{j, i}=0$ for $i=2, \dots, n$ to get $\widetilde{M}_j(\bbX_j)=b_{j,0} X_{j, 0}+b_{j,1} X_{j, 1}$. One has that 
    \[\begin{split}h(\widetilde{M}_j)=& \sum_{v\in M_K} \op{max}\{\log |b_{j, 0}|_v, \log |b_{j, 1}|_v\}\\
    \geq & \frac{1}{n} \sum_{v\in M_K} \op{max}_i \{\log |b_{j, i}|_v\}\geq \frac{1}{n} h(M_j).\end{split}\]
    This implies that $h(\widetilde{M}_j)\geq \frac{1}{n} h(M_j)$. Since $b_{j, 0}\neq 0$ for all $j$, we have that $M_j, X_{j, 1}, X_{j, 2}, \dots, X_{j, n}$ are $K$-linearly independent linear forms. This implies that 
    \[K[M_1, \dots, M_m, \{X_{j, i}\mid i\geq 1, j=1, \dots, m\}]=K[\{X_{j, i}\mid i\geq 0, j=1, \dots, m\}].\]
    Thus $P$ can be written uniquely as 
    \[P=\sum_J c(J) M_1(\bbX_1)^{j_1}M_2(\bbX_2)^{j_2}\dots M_m(\bbX_m)^{j_m} q_J(\bbX_1', \dots, \bbX_m')\] where the coefficients $q_J$ are polynomials in $\bbX_j'=(X_{j,1}, \dots, X_{j, n})$. Let us pick $i$ and $j$ with $j\geq 2$, remove the highest power of $X_{i, j}$ dividing $P$ and specialize $X_{i, j}=0$ to obtain a new polynomial $P^*$. Since the coefficients of $P^*$ forms a subset of $P$,
    \[\begin{split}
        & h(P^*)\leq h(P),\\
        & \op{Ind}(P^*; M^*; d)=\op{Ind}(P; M;d), 
    \end{split}\] 
    where $M^*:=M_{|X_{i, j}=0}$. In this way we can inductively specialize all variables $X_{i, j}=0$ for $j\geq 2$ and obtain a new polynomial $\widetilde{P}(\widetilde{X}_1, \dots, \widetilde{X}_m)$ in the variables $\widetilde{X}_j=(X_{j, 0}, X_{j, 1})$ such that 
    \[\begin{split}
        & h(\widetilde{P})\leq h(P);\\
        &\op{Ind}(\widetilde{P}; \widetilde{M};d)= \op{Ind}(P; M; d);\\
        & h(\widetilde{M}_j)\geq h(M_j)/n.
    \end{split}\]
    Apply Roth's lemma to $\widetilde{P}$ and find that 
    \[d_j h(\widetilde{M}_j)\geq n^{-1} d_j h(M_j)\geq \sigma^{-1}(h(P)+4md_1).\] This implies that 
    \[\op{Ind}(P; M;d)= \op{Ind}(\widetilde{P};\widetilde{M}; d)\leq 2m \sigma^{1/2^{m-1}}.\]
    This completes the proof.
\end{proof}
\subsection{A lower bound on the height of $V(Q)$}

\par Given $ Q \ge 1$, we would like a better control on the height 
of $V(Q)$. 
Before we commence, we recall some basic notation about exterior algebras. Let $V$ be a vector space over $K$ with basis $e_0, \dots, e_n$ then $\wedge^k V$ has basis $e_{i_1}\wedge e_{i_2}\wedge \dots \wedge e_{i_k}$ where $0\leq i_1< i_2< \dots < i_k\leq n$. 
Define the inner product on $\wedge^k V$ by:
\begin{equation}\label{Laplace identity}
(u_1\wedge \dots \wedge u_k)\cdot (v_1\wedge \dots \wedge v_k):=\op{det}\left((u_i\cdot v_j)\right).
\end{equation}
This gives the standard inner product with respect to the basis \[\{e_{i_1}\wedge e_{i_2}\wedge \dots \wedge e_{i_k}\mid 0\leq i_1< i_2< \dots < i_k\leq n\}.\]
We also define a $K$-linear $\ast$-operator on the exterior algebra $\wedge V:=\bigoplus_{k=0}^{n+1} \wedge^k V$ by setting 
\[(e_{i_1}\wedge \dots \wedge e_k)^\ast :=(-1)^{\op{sign}(\pi)}e_{j_0}\wedge \dots \wedge e_{j_{n-k}} \]
where $\pi=(j_0, \dots, j_{n-k}, i_1, \dots, i_k)$ is a permutation of $\{0, \dots, n\}$. It is easy to check that:
\begin{itemize}
    \item $\ast$ is an isometry, 
    \item one has the following expansion 
    \begin{equation}\label{laplace expansion}
        x_0\cdot (x_1\wedge \dots \wedge x_n)^{\ast} =\op{det}(x_0, \dots, x_n)
    \end{equation} for vectors $x_0, \dots, x_n\in V$. 
\end{itemize}

\begin{lemma}\label{technical lemma on height of V(Q)}
    Assume that $\log Q\geq C_4\varepsilon^{-1}$ for a suitably large constant $C_4>0$ and that $V(Q)$ has dimension $n$. There exists a hyperplane $W\subset V$, independent of $\Pi(Q)$ and $\varepsilon>0$ such that either 
    \begin{itemize}
        \item $V(Q)=W$, or, 
        \item $\left((4|S|)^{-1}\varepsilon \log (Q) -C_4\right)\leq h(V(Q))$.
    \end{itemize}
\end{lemma}
\begin{proof}
    Recall that $\iota: K^{n+1}\hookrightarrow K_{\mathbb{A}}^{n+1}$ is the natural inclusion map. We let $y^{(1)}, \dots, y^{(n)}\in K^{n+1}$ be a basis of $V(Q)$ such that $\iota(y^{(i)})\in \Pi(Q)$ for all $i$. Observe that $V(Q)$ is the kernel of the map 
    \[x\mapsto (y^{(1)}\wedge \dots\wedge y^{(n)})\wedge x.\] Thus, it follows that the height of $V(Q)$ is given by $h(y^{(1)}\wedge \dots\wedge y^{(n)})$.
    Note that in particular, $y^{(i)}$ has coordinates in $\cO_{K, S}$. For ease of notation, we set $V:=K^{n+1}$ and $V^\vee$ shall denote its dual. We consider the linear forms 
    \[\widehat{L}_{v, k}:=\left(L_{v,0}\wedge \dots \wedge L_{v, k-1}\wedge L_{v, k+1}\wedge \dots \wedge L_{v,n}\right)^\ast\in V^\vee\]and set 
    \[D_{v,k}:=\widehat{L}_{v, k}\left((y^{(1)}\wedge \dots\wedge y^{(n)})^\ast\right).\] Since $\ast$ is an isometry, we find that 
    \[\begin{split}D_{v,k}=& \left(L_{v,0}\wedge \dots \wedge L_{v, k-1}\wedge L_{v, k+1}\wedge \dots \wedge L_{v,n}\right)(y^{(1)}\wedge \dots\wedge y^{(n)})\\ =& \op{det}\left(L_{v,i}(y^{(j)})\right)_{(i,j); i\neq k}.\end{split}\]
    Note that since $\iota(y^{(j)})\in \Pi(Q)$, we have that $|L_{v,i}(y^{(j)})|_v\leq Q^{c_{v,i}}$. Expanding the determinant, we find that:
    \[\begin{split}
|D_{v,k}|_v  \leq &\op{max}\left(1, |n!|_v\right)\times \op{max}_\pi \prod_{i\neq k}|L_{v,i}(y^{\pi(i)})|_v\\
\leq & \op{max}\left(1, |n!|_v\right) Q^{-c_{v,k}+\sum_i c_{v,i}}
\end{split}\]
where $\pi$ ranges over all bijections $\pi: \{0, \dots, n\}\backslash \{k\}\rightarrow \{1, \dots, n\}$. 
\par The argument involves two steps. First, we assume that there exists a tuple $\{i_v\}_{v\in S}$ such that 
\[\sum_{v\in S} c_{v, i_v}\geq -\frac{\varepsilon}{4}\]
and $D_{v, i_v}\neq 0$ for all $v\in S$. Thus we have that 
\[\prod_{v\in S} |D_{v, i_v}|_v\leq n! Q^{-\sum_v c_{v, i_v}+\sum_{v,i} c_{v,i}}\leq  n! Q^{\varepsilon/4-\varepsilon/2}=n! Q^{-\varepsilon/4}.\]
We have that 
\[\begin{split}\sum_{v\in S} \log |D_{v, i_v}|_v \geq & -\sum_{v\in S} h(D_{v, i_v}) \\ 
= & -\sum_{v\in S} h\left(\widehat{L}_{v, k}\left((y^{(1)}\wedge \dots\wedge y^{(n)})^\ast\right)\right)\\
\geq & -|S|h(V(Q))-C_4',
\end{split}\]
for some constant $C_4'>0$ which arises from the coefficients of the forms $\widehat{L}_{v,k}$. Putting these inequalities together, one finds that 
\[-|S|h(V(Q))-C_4'\leq -\frac{\varepsilon}{4} \log(Q)+\log(n!).\] Setting $C_4'':=|S|^{-1}\left(C_4'+\log(n!)\right)$, we find that $\left((4|S|)^{-1}\varepsilon \log (Q) -C_4''\right)\leq h(V(Q))$.
\par Next, we show that there is a hyperplane $W$ such that $V(Q)=W$ if there does not exist an $S$-tuple $\{i_v\}_{v\in S}$ such that 
\begin{equation}\label{Scv}
\sum_{v\in S} c_{v, i_v}\geq -\frac{\varepsilon}{4}
\end{equation}
and $D_{v, i_v}\neq 0$ for all $v\in S$. Setting $\mathcal{I}_v:=\{i\in \{0,\dots, n\}\mid D_{v, i}\neq 0\}$ the system of equations 
\[\widehat{L}_{v, i}(z)=0\text{ for }i\in \{0, \dots, n\}\backslash \mathcal{I}_v\]
has a nontrivial solution $w=(y^{(1)}\wedge \dots \wedge y^{(n)})^\ast$. Setting $W:=\{x\in V\mid x\cdot w=0\}$, we show that if $\log(Q)$ is sufficiently large then $V(Q)$ is contained in $W$. Since $V(Q)$ has dimension $n$ and $W$ is a hyperplane, it follows then that $V(Q)=W$. As a consequence of \eqref{laplace expansion}, we have the following algebraic identity:
\[x\cdot w=\op{det}(L_v)^{-1}\sum_{i=0}^n (-1)^i L_{v,i}(x)\widehat{L}_{v,i}(w)\] for any $x$ and $w$. Note that $\widehat{L}_{v,i}(w)=0$ for $i\notin \mathcal{I}_v$ and hence, 
\[x\cdot w=\op{det}(L_v)^{-1}\sum_{i\in \mathcal{I}_v} (-1)^i L_{v,i}(x)\widehat{L}_{v,i}(w).\] Now for each $v\in S$, choose $j_v$ such that 
\[c_{v,j_v}=\op{max}_{i\in \mathcal{I}_v} c_{v,i}.\] Note that if $x\in V(Q)\backslash W$ with $\iota(x)\in \Pi(Q)$, then $|L_{v,i}(x)|_v\leq Q^{c_{v,i}}\leq Q^{c_{v, j_v}}$. Thus we get the bound:
\[|x\cdot w|\ll Q^{c_{v, j_v}}.\] On the other hand, since we assume the non-existence of an $|S|$-tuple satisfying \eqref{Scv}, we have that 
\[\sum_{v\in S} c_{v, j_v}<-\varepsilon/4.\] By the product formula, 
\[\begin{split}1=& \prod_{v\in M_K} |x\cdot w|_v
=\prod_{v\in S} |x\cdot w|_v\times \prod_{v\notin S} |x\cdot w|_v,\\
\leq & \prod_{v\in S} |x\cdot w|_v\times \prod_{v\notin S} |w|_v\ll Q^{\sum_{v\in S} c_{v, j_v}}\ll Q^{-\varepsilon/4}. 
\end{split}\]
Therefore, we find that $\log Q\ll \varepsilon^{-1}$. We choose $C_4\geq C_4''$ large so that $\log Q< C_4\varepsilon^{-1}$ and this gives a contradiction. Thus, we find that $V(Q)=W$. 
\end{proof}

\section{Nonvanishing results}

Let
\[
0<\eta\leq \frac{1}{n+1},
\]
choose an integer \(m\) such that
\[
m\geq 
\frac{4}{(n+1)(n+2)\eta^2}
\log\left(2(n+1)r|S|\right)
\]
and set
\[
\sigma:=\left(\frac{\eta}{4}\right)^{2^{m-1}}.
\]
Having
fixed \(\eta\), \(m\), and \(\sigma\), we impose lower bounds on the parameters
\(Q_i\), in particular on \(Q_1\). The points $\mathbf{x}_1,\ldots,\mathbf{x}_m$ are then chosen from the infinite set of solutions so that their associated
parameters \(Q_1,\ldots,Q_m\) satisfy these bounds. 

\par Assume that for $i=1, \dots, m$, $V_i:=V(Q_i)$ has dimension $n$ and $V_i\neq W$ where $W$ is the exceptional space given to us in Lemma \ref{technical lemma on height of V(Q)}. Recall that it is assumed that $V_i\neq V_j$ for $i\neq j$. Furthermore, we assume that $\log (Q_1)\geq C_5$ and $\log (Q_{j+1})\geq 2\sigma^{-1} \log(Q_j)$ for every $j$ and a constant $C_5>\op{max}\{C_3,C_4\}$ which will be chosen to be large. We write $V_j$ as the zero set of the linear form $M_j$, where 
\[M_j(x_0, \dots, x_n):=b_{j,0} x_0+\dots +b_{j, n} x_n\] and note that 
\[h(V_j)=h(M_j)=h\left((b_{j,0}, \dots, b_{j, n})\right).\]
Given a large constant $D$ which is yet to be determined, we set $d_j:=\lfloor D/\log Q_j\rfloor$, we get a polynomial $P$ of multidegree $d=(d_1, \dots, d_m)$, satisfying the following conditions. Writing
\[\partial_I P(\bbX_1, \dots, \bbX_m)=\sum_J a(L_v; J; I) L_v(\bbX_1)^{\mathbf{j}_1}\dots L_v(\bbX_m)^{\mathbf{j}_m},\] Lemma \ref{lemma 4.15} then implies that:
\begin{enumerate}
    \item $h(P)\leq C_5|d|$, $h(a(L_v; J;I))\leq C_5|d|$;
    \item $a(L_v; J; I)=0$ if 
    \[\begin{split}
        & \left(J_i/d\right)\leq \frac{m}{n+1}-2m \eta\text{, or,}\\
         & \left(J_i/d\right)\geq \frac{m}{n+1}+2nm \eta
    \end{split}\]
    and $\left(I/d\right)\leq m\eta$.
\end{enumerate}
We take $M=(M_1, \dots, M_m)$ and verify that the hypotheses of the generalized Roth lemma are satisfied. 
    \par First we verify that $d_{j+1}/d_j<\sigma$ where we recall that $d_j:=\lfloor D/\log Q_j\rfloor$. This follows for large enough values of $D$, since it is assumed that $\log(Q_{j+1})\geq 2\sigma^{-1} \log Q_j$. Second, we are to verify that 
    \[d_j h(M_j)\geq n \sigma^{-1} (h(P)+4md_1).\]
    Since $h(M_j)=h(V_j)$ and $V_j\neq W$, it follows that 
    \[d_j h(M_j)\geq \left\lfloor\frac{D}{\log Q_j}\right\rfloor\left((4r|S|)^{-1}\varepsilon \log(Q_j)-C_5\right) \geq C_5\varepsilon D.\]
   This implies
   \begin{align*}
n \sigma^{-1}\left(h(P)+4 m d_{1}\right) 
& \leq n \sigma^{-1}\left(C_5|d|+4 m d_{1}\right) \\
&=n\sigma^{-1}\left(C_5\left(d_{1}+\ldots+d_{m}\right)+4 m d_{1}\right)\\
& \leq n \sigma^{-1}(C_5+4 m)\left(\frac{D}{\log Q_{1}}\right)\left(1+\sigma+\sigma^{2}+\cdots+\sigma^{m-1}\right) \\
& \ll \frac{n\sigma^{-1} m D}{\log \left(Q_{1}\right)} .
\end{align*}

Since $\log \left(Q_{1}\right)$ is as large as we want, $\dfrac{\sigma^{-1} m n D}{\log \left(Q_{1}\right)}$ is small compared to $\varepsilon D$.
Thus the generalized Roth Lemma implies that
$$
\ind(P ; d ; M)  \leq 2 m \sigma^{1 / 2^{m- 1}} 
 \leq \frac{m \eta}{2}.
$$
In particular, $P$ does not vanish identically on $V_{1} \times \cdots \times V_{m}$.
We want to prove that $\partial_{I} P$ is non-vanishing at a certain point in $V_{1}  \times \cdots \times V_{m}$ of small height for some indices $I$.
\begin{lemma}
    Let $K$ be a field of characteristic $0$ and $f \in K\left[x_{1}, \cdots, x_{N}\right]$ such that $f \neq 0$ and $\operatorname{deg}_{x_{j}} f \leqslant e_{i}$. Also let $B>0$.
Then, there exist integers $z_{j}$ and $i_{j}$ such that
$
\left|z_{j}\right| \leq B \text { and }
$ $0 \leqslant i_{j} \leqslant e_{j} / B$ such that

$$
\partial_{\mathbf{i}}f\left(z_{1}, \ldots, z_{N}\right) \neq 0, \quad \mathbf{i}=\left(i_{1}, \ldots, i_{N}\right).
$$
\end{lemma}
\begin{proof}
   We  argue by induction on $N$. For $N=1$, suppose that $\frac{\partial^{i}}{\partial x^{i}} f(x)$ vanishes at all $|z| \leqslant B$ $(z \in \mathbb{Z})$ and $i \leq e_{1} / B$. This implies $\prod_{|b| \leqslant B}(x-b)^{\left\lfloor e_{1} / B\right\rfloor+1} $ divides  $f(x)$. Therefore, we have 
$$
\begin{aligned}
 \operatorname{deg} f 
 &\geqslant \operatorname{deg}\left(\prod_{|b| \leqslant B}(x-b)^{\left\lfloor e_{1} / B\right\rfloor+1}\right) \\
&=\sum_{|b| \leqslant B}\left(\left\lfloor e_{1} / B\right\rfloor+1\right) \\
& \geqslant\left(e_{1} / B\right)(2 B+1)>e_{1} 
\end{aligned}
$$
This is  a contradiction. 
Assume $N \geqslant 2$ and the conclusion holds for $(N-1)$. We write

$$
\begin{aligned}
\tilde{f}(x_N) = f\left(x_{1}, \ldots, x_{N-1}, x_{N}\right)=A_{0}+A_{1} x_N+\cdots+ & A_{d} x_{N}^d,
\end{aligned}
$$
where $d \leq e_{N}$ and $A_i\in K[x_1,\ldots,x_{N-1}]$.
Pick $\quad g\left(x_{1}, \ldots, x_{N-1}\right):=A_{0}$
and apply induction to show that there exist $z^{\prime}=\left(z_{1}, . ., z_{N-1}\right) \in \mathbb{Z}^{N-1}$ and $i^{\prime}=\left(i_{1}, \ldots, i_{N-1}\right)$ such that $\quad\left(\partial_{i^{\prime}} A_{0}\right)\left(z^{\prime}\right) \neq 0$  with 
$\left|z_{j}\right| \leqslant B$ and $i_{j} \leqslant e_j / B$ for  $j=1, \cdots, N-1.$ Therefore, 

$$
g_2(x):  =\partial_{i^{\prime}} f(x) =  \left(\partial_{i^{\prime}} A_0\right)\left(z^{\prime}\right)+\left(\partial_{i^{\prime}} A_{1}\right)\left(z^{\prime}\right) x+\cdots  +\left(\partial_{i^{\prime}} A_{d}\right)\left(z^{\prime}\right)x^{d}
$$
is a non-zero polynomial in $1$-variable. By induction hypothesis, there exists $z_N \in \mathbb{Z}$ and $i_N$ such that 
$$
\left(\frac{\partial^{i_{N}}}{\partial x^{i_{N}}}  g_2\right)\left(z_{N}\right) \neq 0
$$
with $|z_N| \le B$ and $|i_N| \leqslant e_{N} / B$. Hence the proof.
\end{proof}
For $h=1, \ldots, m$, let $y_{h}^{(1)}, \ldots, y_{h}^{(n)} \in E_{h}=K^{n+1}$ be linearly independent points such that

$$
L\left(y_{h}^{(l)}\right) \in \prod(Q_{h}).
$$

This is possible since $\prod\left(Q_{h}\right)$ has rank $n$. We write $\mathbf{z}_{h}=\left(z_{h1}, \ldots, z_{hn}\right)$, $\mathbf{Z}=\left(\mathbf{z}_{1}, \ldots, \mathbf{z}_{m}\right)$ and $y_{h}=\left(y_{h}^{(1)}, \ldots, y_{h}^{(n)}\right).$ Consider the polynomial

$$
R(\mathbf{Z}):=\left(\partial_{I} P\right)\left(\mathbf{z}_{1} y_{1}, \mathbf{z}_{2} y_{2}, \cdots, \mathbf{z}_{m} y_{m}\right).
$$
 Therefore, $R(\mathbf{Z})$ has degree at most $d_{h}$ in the block of variables $\mathbf{z}_{h}$. By the non-vanishing lemma, there exists 
$\mathbf{Z}^{\prime}=(z_{hl}^{\prime})$ and $j_{hl} \ge 0$
such that $z_{hl} \in \mathbb{Z}$ with $\left|z_{hl}^{\prime}\right| \leqslant B$ and 
$0 \leqslant j_{hl} \leqslant d_{h} / B$, and $\left(\partial_{J} R\right)\left(\mathbf{Z}^{\prime}\right) \neq 0$.
Note that $\partial_{J} R$ can be related to $\partial_{I^{\prime}} P$. Let

$$
x^{\prime}  =\left(\mathbb{x}_{1}^{\prime}, \ldots, \mathbb{x}_{m}^{\prime}\right)  
=\left(\sum_{l=1}^{n} z_{1 l}^{\prime} y_{1}^{(l)}, \cdots, \sum_{l=1}^{n} z_{m l}^{\prime} y_{m}^{(l)}\right) .
$$

Thus $(\partial _JR)(\mathbf{Z}^{\prime})$ is a linear combination of $(\partial_{I^{\prime}} P)(x^{\prime})$
where $I^{\prime}=I+I^{*}$ and $I^{*}=\left(i_{1}^{*}, \ldots, i_{m}^{*}\right)$ 
with $\quad\left|i_{h}^{*}\right| \leqslant \frac{n d_{h}}{B}$. Choose $B=\frac{2n}{\eta}$,
note that 

$$
\left(\frac{|I^{\prime}|}{d}\right)
=\left(\frac{|I|}{d}\right)+\left(\frac{|I^{*}|}{d}\right)
\leqslant \frac{m \eta}{2}+\sum_{h=1}^{m} \frac{\left|i_{h}^{*}\right|}{d_{h}}
 \leqslant \frac{m \eta}{2}+\frac{m \eta}{2}=m\eta. 
$$

%Set $T(x):=\partial_{I^{\prime}} P(x)$, then $T\left(x^{\prime}\right) \neq 0$  since  $(\partial_J R)\left(\mathbf{Z}^{\prime}\right) \neq 0$.
Since $\partial_{J} R\left(\mathbf{Z}^{\prime}\right) \neq 0$ and a linear combination of $\left(\partial_{I^{\prime}} P\right)\left(x^{\prime}\right)$ with $(I^{\prime}/d) \leq m \eta$, hence there exists   $I^{\prime}$ such that

$$
\left(\partial_{I^{\prime}} P\right)\left(x^{\prime}\right) \neq 0.
$$

Set $T(x):=\left(\partial_{I^{\prime}} P\right)(x)$. For $n$ fixed, $m$ will be large. Choose $\eta\in\left(\left.0,\frac{1}{n+1}\right.\right]$ and 

$$
m \geqslant \frac{4}{(n+1)(n+2) \eta^{2}} \log (2(n+1) r|s|),
$$
where $r = [K:\mathbb{Q}]$. Let 
$
\sigma=(\eta / 4)^{2^{m-1}}$ and choose $Q_{h}$ for $ h=1, \cdots, m$, such that 
\begin{itemize}
    \item $\operatorname{rank} \prod\left(Q_{h}\right)=n$
    \item $\log \left(Q_{1}\right) \geqslant C(\varepsilon \sigma)^{-1} m$
    \item $\log \left(Q_{h+1}\right) \geqslant 2 \sigma^{-1} \log \left(Q_{h}\right)$.
\end{itemize}  
Let $y_{h}^{(l)}$, for $l=1, \ldots, n$ be a basis for $V\left(Q_{h}\right)$ such that $i\left(y_{h}^{(l)}\right)$ is a basis of $\prod\left(Q_{h}\right)$.
Then, there exists a non-zero polynomial $T(x) \in K[x]$ of degree at most $d=\left(d_{1}, \ldots ,d_{m}\right)$ and $z_{hl} \in \mathbb{Z}$ with $\left|z_{hl}\right| \leqslant\frac{2n}{\eta}$ such that
\begin{enumerate}
    \item  $\quad T\left(x^{\prime}\right) \neq 0$.
\item $\quad h(T)\ll d_1.$
\item For each \(v\in S\), write
\[
T(\mathbf{X}_1,\ldots,\mathbf{X}_m)
=
\sum_J a(L_v;J)
L_v(\mathbf{X}_1)^{\mathbf{j}_1}
\cdots
L_v(\mathbf{X}_m)^{\mathbf{j}_m}.
\]
Then, whenever \(a(L_v;J)\neq 0\), we have
\[
\frac{m}{n+1}-2m\eta
\leq
\left(\frac{J_i}{d}\right)
\leq
\frac{m}{n+1}+2mn\eta
\qquad
(0\leq i\leq n).
\]
\end{enumerate}
\subsection{Proof of Theorem \ref{theorem 4.14}}
We have $$
T(x)=\sum_{J} a\left(L_{v}, J\right) L_{v}\left(x_{1}\right)^{\mathbf{j}_{1}} \cdots L_{v}\left(x_{m}\right)^{\mathbf{j}_{m}}.
$$
We obtain upper and lower bounds for
$$
A:=\sum_{v \in S} \log \max _{a\left(L_{v}, J\right) \neq 0}\left|L_{v}\left(x_{1}\right)^{\mathbf{j}_{1}} \cdots L_{v}\left(x_{m}\right)^{\mathbf{j}_{m}}\right|_{v} .
$$
\subsubsection{Upper bound for $A$}
Recall that $d_{h}=\left\lfloor\frac{D}{\log Q_{h}}\right\rfloor.$ This implies

\begin{equation}\label{i}
  d_{h} \log Q_{h}=D+o(D),  
  \end{equation}
  \begin{equation}\label{ii}
      \left|z_{hl}\right| \leqslant \frac{2n}{\eta}, 
\end{equation} and
\begin{equation}\label{iii}
\left|L_{v,i}\left(y_{h}^{(l)}\right)\right|_{v} \leq Q_{h}^{c_{v, i}}. 
\end{equation}
Since $i\left(y_{h}^{(l)}\right) \in \prod\left(Q_{h}\right)$.
Also, we have
\begin{equation}\label{iv}
 \sum_{v \in S} \sum_{i=0}^{n} c_{v, i}<-\varepsilon / 2.     
\end{equation}
From the above inequalities, we have
$$
\begin{aligned}
\log \left|L_{v}\left(\mathbb{x}_{1}^{\prime}\right)^{\mathbf{j}_{1}} \cdots L_{v}\left(\mathbb{x}_{m}^{\prime}\right)^{\mathbf{j}_{m}}\right|_{v} 
& =\log \left(\prod_{h=1}^{m} \prod_{i=0}^{n}\left|L_{v,i}\left(\mathbb{x}_{h}^{\prime}\right)\right|_{v}^{{j}_{h_i}}\right) \\
& =\sum_{k=1}^{m} \sum_{i=0}^{n} j_{h_{i}} \log \left|L_{v,i}\left(\mathbb{x}_{h}^{\prime}\right)\right|_{v}.
\end{aligned}$$
Since, we have
$$
\begin{aligned}
L_{v,i}\left(\mathbb{x}_{h}^{\prime}\right)&=L_{v,i}\left(\sum_{l=1}^{n} z_{hl} y_{h}^{(l)}\right) \\
& =\sum_{l=1}^{n} z_{hl} L_{v,i}\left(y_{h}^{(l)}\right) . \end{aligned}$$
Therefore, this gives
$$
\begin{aligned}
\left|L_{v,i}\left(\mathbb{x}_{h}^{\prime}\right)\right|_{v}&=\left|\sum_{l=1}^{n} z_{hl} L_{v,i}\left(y_{h}^{(l)}\right)\right|_{v}. \end{aligned}$$
Taking the log on both sides, we have 
$$
\begin{aligned}
\log\left|L_{v,i}\left(\mathbb{x}_{h}^{\prime}\right)\right|_{v}
& \leqslant \max _{h,l}\log|z_{h l}|_{v}+ \max_{h,l}\log \left|L_{v,i}\left(y_{h}^{(l)}\right)\right|_{v}\\ & \leqslant  
\max _{h, l} \log \left|z_{h l}\right|_{v}+c_{v,i} \log \left(Q_{h}\right),
\end{aligned}
$$

where the last inequality comes from $\eqref{iii}$. Therefore, we have 

$$
\begin{aligned}
 \log \left|L_{v}\left(\mathbb{x}_{1}^{\prime}\right)^{\mathbf{j}_{1}} \cdots L_{v}\left(\mathbb{x}_{m}^{\prime}\right)^{\mathbf{j}_{m}}\right|_{v}  
& \leqslant \sum_{h=1}^{m} \sum_{i=0}^{n} j_{h_i}\left(\max _{h, l} \log \left|z_{h, l}\right|_{v}+c_{v, i} \log \left(Q_{h}\right)\right) \\
& \leqslant 
 \sum_{h=1}^{m} \sum_{i=0}^{n} \frac{j_{h_i}}{d_{h}} c_{v,i} d_{h} \log \left(Q_{h}\right)  + O \left(\sum_{h = 1}^{m} \sum_{i = 0}^{n} \frac { j_{ h_i } } { d_{ h } } d_{ h } \left(\frac{2 n}{\eta}\right)\right).
\end{aligned}
$$
Since
$$
\left|\left(\frac{J_{i}} {d}\right)-\frac{m}{n+1}\right| \leq 2 m n \eta 
$$
we get 
$$
 \log \left|L_{v}\left(\mathbb{x}_{1}^{\prime}\right)^{\mathbf{j}_{1}} \cdots L_{v}\left(\mathbb{x}_{m}^{\prime}\right)^{\mathbf{j}_{m}}\right|_{v} 
 \leqslant \sum_{h=1}^{m} \sum_{i=0}^{n} \frac{j_{h_i}}{d_{h}} c_{v_{i}} d_{h} \log \left(Q_{h}\right) 
+O\left(m\left(\frac{m}{n+1}+2 m n \eta\right) \log \left(\frac{2 n}{\eta}\right) d_{1}\right)
$$
Summing over all places in $S$ we get 
$$
\sum_{v \in S} \log \left|L_{v}\left(\mathbb{x}_{1}^{\prime}\right)^{\mathbf{j}_{1}} \cdots L_{v}\left(\mathbb{x}_{m}^{\prime}\right)^{\mathbf{j}_{m}}\right|_{v}  
 \leqslant\left(\sum_{v \in S} \sum_{i=0}^{n} c_{v i}\right) \frac{m}{(n+1)} D+O(m \eta D)
 +O\left(\log \left(\frac{1 }{ \eta}\right) \frac{D}{\log Q_{1}} \frac{m^{2}}{(n+1)}\right),
$$
where the first error term in last inequality comes from $\eqref{ii}$. We choose \(Q_1\) so that
\[
\log Q_1\gg \frac{m}{\eta^2}.
\]
Hence the last error term is \(O(m\eta D)\). Dividing by \(mD\), and using \eqref{iv}, we obtain
\[
\frac{1}{mD}
\sum_{v\in S}
\log
\left|
L_v(\mathbf{x}'_1)^{\mathbf{j}_1}
\cdots
L_v(\mathbf{x}'_m)^{\mathbf{j}_m}
\right|_v
\leq
-\frac{\epsilon}{n+1}+O(\eta).
\]
Choosing \(\eta\) sufficiently small in terms of \(\epsilon\), we get
\begin{equation}\label{UB}
\frac{1}{m D} \sum_{v \in S} \log \max _{a\left(L_{v };J\right) \neq 0}\left|L_{v}\left(\mathbb{x}_{1}^{\prime}\right)^{\mathbf{j}_{1}} \cdots L_{v}\left(\mathbb{x}_{m}^{\prime}\right)^{\mathbf{j}_{m}}\right|_{v}
\leq -\frac{\varepsilon}{4(n+1)}.
\end{equation}

\iffalse
Therefore, we have
$$
\frac{1}{m D} \sum_{v \in S} \log \left|L_{v}\left(\mathbb{x}_{1}^{\prime}\right)^{\mathbf{j}_{1}} \cdots L_{v}\left(\mathbb{x}_{m}^{\prime}\right)^{\mathbf{j}_{m}}\right|_{v}<0 .
$$
\fi

\subsubsection{Lower bound of A}
For a non-archimedean valuation $v$, we have
$$
  \sum_{v \in S} \log \left|T\left(x^{\prime}\right)\right|_{v}
 \leq \sum_{v \in S} \log \max _{J}\left|a\left(L_{v} ; J\right)\right|_{v} 
 +\sum_{v \in S} \log \max _{a\left(L_{v} ; J\right) \neq 0}\left|L_{v}\left(\mathbb{x}_{1}^{\prime}\right)^{\mathbf{j}_{1}} \cdots L_{v}\left(\mathbb{x}_{m}^{\prime}\right)^{\mathbf{j}_{m}}\right|_{v} 
$$
Therefore, we write
$$
\sum_{v \in S} \log\max _{a\left(L_{v} ; J\right) \neq 0} \left|L_{v}\left(\mathbb{x}_{1}^{\prime}\right)^{\mathbf{j}_{1}} \cdots L_{v}\left(\mathbb{x}_{m}^{\prime}\right)^{\mathbf{j}_{m}}\right|_{v}  \geqslant \sum_{v \in S} \log \left|T\left(x^{\prime}\right)\right|_{v}-\sum_{v \in S} \log \max _{J}\left|a\left(L_{v}; J\right)\right|_{v}.
$$
Since $T\left(x^{\prime}\right) \neq 0$ and 
$\prod_{v\in M_{K}}\left|T\left(x^{\prime}\right)\right|_{v}=1$, we have
\[\sum_{v \in S}\log \left|T\left(x^{\prime}\right)\right|_{v}=-\sum_{v \not\in S} \log \left|T\left(x^{\prime}\right)\right|_v.\]
Write
\[
T(X_1,\ldots,X_m)=\sum_I b_I X^I,
\]
where \(I\) ranges over the multi-indices occurring in \(T\). We set
\[
|T|_v:=\max_I |b_I|_v.
\] Since \(T\) has multidegree at most \(d=(d_1,\ldots,d_m)\) and nonarchimedian $v$, we have
\[
|T(x')|_v
\leq
|T|_v\,\max\{1,|x'|_v\}^{|d|},
\]
where
\[
|d|:=d_1+\cdots+d_m.
\]
In particular,
\[
\log |T(x')|_v
\leq
\log |T|_v
+
|d|\log^+|x'|_v.
\]
This yields
\[
\begin{aligned}
\sum_{v \notin S} \log |T(x')|_v
\leq
\sum_{v \notin S} \log |T|_v
+
|d|\sum_{v \notin S} \log^+ |x'|_v
\end{aligned}
\]
where
\[
|x'|_v
=
\max_h\{|x'_h|_v\}
=
\max_h
\left\{
\left|\sum_l z_{hl}y_h^{(l)}\right|_v
\right\}.
\]
Note that $i(y_{h}^{(l)}) \in \prod\left(Q_{h}\right)$. So, in particular
$|y_{h}^{(l)}|_{v} \leqslant 1 $ for all $v \notin S.$
This implies $\left|x^{\prime}\right|_{v} \leqslant 1 $ for  all $ v \notin S$ and so $\sum_{v \notin S} \log |x'|_{v} \leqslant 0 .$ Therefore, we have 
\begin{align*}
\sum_{v \notin S} \log \left|T\left(x^{\prime}\right)\right|_{v} \leq \sum_{v \notin S} \log |T|_{v}.
\end{align*}
This can be rewritten as 
\begin{equation}\label{inequality(a)}
   \sum_{v \in S} \log \left|T\left(x^{\prime}\right)\right|_{v} \geqslant-\sum_{v \notin S} \log |T|_{v}. 
\end{equation}
On the other hand, we have 

\begin{equation} \label{inequality(b)}
\sum_{v \in S} \log\max _{J}\left|a\left(L_{v}, J\right)\right|_{v} =\sum_{v \in S} \log |T|_{v}+O\left(d_{1}\right).
\end{equation}
Putting \eqref{inequality(a)} and \eqref{inequality(b)} together, we have
\begin{align*}
    &\frac{1}{mD}\sum_{v \in S} \log \max _{a\left(L_{v} ; J\right) \neq 0} \left|L_{v}\left(\mathbb{x}_{1}^{\prime}\right)^{\mathbf{j}_{1}} \cdots L_{v}\left(\mathbb{x}_{m}^{\prime}\right)^{\mathbf{j}_{m}}\right|_{v}  \\
& \geqslant -\frac{1}{mD}\sum_{v \notin S} \log |T|_{v}-\frac{1}{mD}\sum_{v \in S} \log |T|_{v}+O\left(\frac{d_{1}}{mD}\right).
\end{align*}
Choose $
d_{1}=\left\lfloor\frac{D}{\log Q_{1}}\right\rfloor \sim \frac{D}{\log Q_{1}},$
This implies 
$$
\frac{d_{1}}{m D} \sim \frac{1}{m \log Q_{1}}.
$$
It is  small when  $Q_1 \gg 1$.
Combining the two sums, we obtain
\[
-\sum_{v\notin S}\log |T|_v
-
\sum_{v\in S}\log |T|_v
=
-\sum_{v\in M_K}\log |T|_v.
\]
Since
\[
\log |T|_v\leq \log^+ |T|_v
\]
for every \(v\in M_K\), it follows that
\[
-\sum_{v\in M_K}\log |T|_v
\geq
-\sum_{v\in M_K}\log^+ |T|_v
=
-h(T).
\]
Therefore
\[
-\frac{1}{mD}\sum_{v\notin S}\log |T|_v
-
\frac{1}{mD}\sum_{v\in S}\log |T|_v
\geq
-\frac{h(T)}{mD}.
\]
Recall that $ h(T) \ll d_{1}$, so 
$$
-\frac{h(T)}{m D}=O\left(\frac{d_{1}}{m D}\right)=O\left(\frac{1}{m \log Q_{1}}\right). 
$$
Hence we get
$$
\begin{aligned}
& \frac{1}{m D} \sum_{v \in S} \log \max _{a\left(L_{v} ; J\right) \neq 0}\left|L_{v}\left(\mathbb{x}_{1}^{\prime}\right)^{\mathbf{j}_{1}} \cdots L_{v}\left(\mathbb{x}_{m}^{\prime}\right)^{\mathbf{j}_{m}}\right|_{v} \gg \frac{-1}{m \log Q_{1}}
\end{aligned}
$$
which is arbitrarily close to $0$ for large $Q_1$. Also the upper bound of this term is $ \frac{-\varepsilon}{4(n+1)}$ by \eqref{UB}. Hence we get a contradiction. 
\subsection{Evertse's lemma}

A significant simplification of Schmidt's original proof was later achieved by Evertse, who introduced an elegant lemma that circumvented the need for Schmidt’s reliance on Mahler’s results on successive minima of compound convex bodies and Davenport’s lemma on stretched parallelopipeds.
\begin{lemma}[Evertse's lemma]\label{evertse's lemma}
Let \( K \) be a number field, and let \( S \) be a finite set of places of \( K \) that includes all the archimedean places. Suppose \( \mathbf{x}^{(1)}, \ldots, \mathbf{x}^{(n+1)} \) form a basis for a \( K \)-vector space \( V \). For each \( v \in S \), let \( L_{v 0}, \ldots, L_{v n} \) be linearly independent linear forms on \( V \) with coefficients in \( K_v \). Further, let \( \mu_{v, j} \) be real numbers such that  
\[
0 < \mu_{v, 1} \leq \mu_{v, 2} \leq \cdots \leq \mu_{v, n+1}, \quad \text{and} \quad \left|L_{v k}\left(\mathbf{x}^{(j)}\right)\right|_v \leq \mu_{v, j}
\]  
for all \( k \) and \( j \). Then, there exist vectors  
\[
\mathbf{v}^{(1)} = \mathbf{x}^{(1)}, \quad \mathbf{v}^{(i)} = \sum_{j=1}^{i-1} \xi_{i j} \mathbf{x}^{(j)} + \mathbf{x}^{(i)} \quad \text{with} \quad \xi_{i j} \in \cO_{K, S} \quad (1 \leq j < i \leq n+1),
\]  
along with bijective maps \( \pi_v : \{1, \ldots, n+1\} \to \{0, \ldots, n\} \), and a positive constant \( C_6\) such that for any \( v \in S \) and \( i, j \in \{1, \ldots, n+1\} \), the following inequality holds:  
\[
\left|L_{v, \pi_v(i)}\left(\mathbf{v}^{(j)}\right)\right|_v \leq 
\begin{cases} 
C_6 \min \left(\mu_{v, i}, \mu_{v, j}\right) & \text{if } v \text{ is archimedean}, \\ 
\min \left(\mu_{v, i}, \mu_{v, j}\right) & \text{if } v \text{ is non-archimedean}.
\end{cases}
\]  

The constant \( C_6\) depends only on \( K \), \( S \) and the set of linear forms \( L_{v i} \).
\end{lemma}
\begin{proof}
    The proof proceeds by induction on \( n \). The base case \( n = 0 \) is straightforward, as the result holds trivially when there is only one vector. Now assume \( n \geq 1 \) and suppose that the lemma holds for \( n-1 \); we aim to establish it for \( n \). 

Fix a place \( v \in S \), and consider the \( K \)-vector subspace \( V' \subseteq V \) spanned by the first \( n \) basis vectors \( \mathbf{x}^{(1)}, \ldots, \mathbf{x}^{(n)} \). 
%Since \( L_{v 0}, \ldots, L_{v n} \) are linearly independent 
The linear system  
\[
\sum_{k=0}^{n} L_{v k}\left(\mathbf{x}^{(j)}\right) \alpha_{v k} = 0 \quad \text{for } j = 1, \ldots, n  
\]  
has a non-trivial solution \( (\alpha_{v 0}, \ldots, \alpha_{v n}) \in K_v^{n+1} \). Let \( \pi_v(n+1) \) denote an index where the maximum norm is attained, i.e.,  
\[
|\alpha_{v, \pi_v(n+1)}|_v = \max_{i} |\alpha_{v i}|_v.  
\]  
Since the solution is non-trivial, we have \( \alpha_{v, \pi_v(n+1)} \neq 0 \). Define the ratios  
\[
\beta_{v i} = -\frac{\alpha_{v i}}{\alpha_{v, \pi_v(n+1)}} \quad \text{for } i \neq \pi_v(n+1).  
\]  
Substituting these ratios into the system yields  
\begin{equation}\label{E1}
L_{v, \pi_v(n+1)}(\mathbf{x}^{(j)}) = \sum_{k \neq \pi_v(n+1)} \beta_{v k} L_{v k}(\mathbf{x}^{(j)}) \quad \text{with } |\beta_{v k}|_v \leq 1. 
\end{equation}

The restrictions to $V^{\prime}$ of the linear forms $L_{v k}$,
$k \neq \pi_{v}(n+1)$, are linearly independent over $K_{v}$. Applying the inductive hypothesis to these restricted forms, there exist bijections  
\[
\pi_v : \{1, \ldots, n\} \to \{0, \ldots, n\} \setminus \{\pi_v(n+1)\},  
\]  
and vectors  
\[
\mathbf{v}^{(1)} = \mathbf{x}^{(1)}, \quad \mathbf{v}^{(i)} = \sum_{j=1}^{i-1} \xi_{i j} \mathbf{x}^{(j)} + \mathbf{x}^{(i)} \quad \text{for } i = 2, \ldots, n,  
\]  
with coefficients \( \xi_{i j} \in \cO_{K,S} \), such that  
\[
|L_{v, \pi_v(i)}(\mathbf{v}^{(j)})|_v \leq  
\begin{cases}  
C_6' \min(\mu_{v, i}, \mu_{v, j}) & \text{if } v \text{ is archimedean}, \\  
\min(\mu_{v, i}, \mu_{v, j}) & \text{if } v \text{ is non-archimedean}.  
\end{cases}  
\]  

Now, to extend this construction to include \( \mathbf{v}^{(n+1)} \), we need coefficients \( \xi_{n+1, j}' \in \cO_{K,S} \) such that  
\[
\mathbf{v}^{(n+1)} = \mathbf{x}^{(n+1)} + \sum_{j=1}^{n} \xi_{n+1, j}' \mathbf{v}^{(j)}  
\]  
satisfies the desired bounds. Specifically, for each \( i \in \{1, \ldots, n\} \) and \( v \in S \), we require  
\[
|L_{v, \pi_v(i)}(\mathbf{v}^{(n+1)})|_v \leq  
\begin{cases}  
C_6 \mu_{v, i} & \text{if } v \text{ is archimedean}, \\  
\mu_{v, i} & \text{if } v \text{ is non-archimedean}.  
\end{cases}  
\]  
For $i=n+1$, we use \eqref{E1} to get the result.
To achieve this, note that \( L_{v, \pi_v(i)} \) for \( i \neq n+1 \) remains linearly independent on \( V' \), which is generated by \( \mathbf{v}^{(1)}, \ldots, \mathbf{v}^{(n)} \). Thus, we can solve the linear system  
\[
L_{v, \pi_v(i)}(\mathbf{x}^{(n+1)}) = \sum_{j=1}^{n} \gamma_{v j} L_{v, \pi_v(i)}(\mathbf{v}^{(j)}) \quad \text{for } i \neq n+1,  
\]  
with \( \gamma_{v j} \in K_v \). Since \( \cO_{K,S} \) forms a discrete lattice in \( \prod_{v \in S} K_v \) with a bounded fundamental domain, there exists a vector \( (\xi_{n+1, 1}', \ldots, \xi_{n+1, n}') \in \cO_{K, S}^n \) such that  
\[
|\xi_{n+1, j}' + \gamma_{v j}|_v \leq A_v,  
\]  
where \( A_v \) is bounded in terms of \( K \) and \( S \). For non-archimedean places \( v \), we may choose \( A_v = 1 \) using a scaling argument with a suitable nonzero element \( m \in \cO_{K,S} \). This ensures the desired bound on \( |L_{v, \pi_v(i)}(\mathbf{v}^{(n+1)})|_v \), completing the proof.
\end{proof}

\subsection{Proof of the subspace theorem}

\par Let $\Pi(Q)$ denote an approximation domain associated with the forms \( L_{v,i} \) and the parameters \( c_{v,i} \). Assume that the rank \( R := \operatorname{rank}(\Pi(Q)) \) satisfies \( 1 \leq R \leq n \). The value of \( R \) is determined by the inequalities \( \lambda_{R} \leq 1 < \lambda_{R+1} \). Moreover, as noted earlier in the proof of Lemma \ref{rank Pi is <=n}, an application of Minkowski's second theorem gives us that 
\[
\lambda_{n+1} \gg Q^{\varepsilon / (2n+2)}.
\]
We now introduce \( k \) as the smallest integer within the range \( [R, n] \) for which the ratio \( \lambda_{k} / \lambda_{k+1} \) is minimized. To estimate this minimum ratio, observe that the product of all successive ratios from \( j = R \) to \( j = n \) can be expressed as  
\[
\prod_{j=R}^{n} \frac{\lambda_{j}}{\lambda_{j+1}} = \frac{\lambda_{R}}{\lambda_{n+1}}.
\]  
Taking the geometric mean of these ratios yields  
\[
\left( \prod_{j=R}^{n} \frac{\lambda_{j}}{\lambda_{j+1}} \right)^{\frac{1}{n+1-R}} = \left( \frac{\lambda_{R}}{\lambda_{n+1}} \right)^{\frac{1}{n+1-R}}.
\]  
Since \( k \) is chosen such that \( \lambda_{k} / \lambda_{k+1} \) is minimal, it follows that  
\[
\frac{\lambda_{k}}{\lambda_{k+1}} \leq \left( \frac{\lambda_{R}}{\lambda_{n+1}} \right)^{\frac{1}{n+1-R}}.
\]  
By applying the lower bound \( \lambda_{n+1} \gg Q^{\varepsilon / (2n+2)} \), we deduce that  
\[
\frac{\lambda_{R}}{\lambda_{n+1}} \ll Q^{-\varepsilon / (2n+2)},
\]  
which implies  
\begin{equation}\label{lambdak/lambdak+1}
\frac{\lambda_{k}}{\lambda_{k+1}} \ll Q^{-\varepsilon / \{2(n+1)n\}}.
\end{equation}
Setting
$$
\varepsilon_{v}= \begin{cases}{\left[K_{v}: \mathbb{R}\right] /[K: \mathbb{Q}]} & \text { if } v \text { is archimedean } \\ 0 & \text { otherwise,}\end{cases}
$$
we then define $\mu_{v, j}=\lambda_{j}^{\varepsilon_{v}}$ for $j=1, \ldots, n+1$. Since $\sum_{v \mid \infty} \varepsilon_{v}=1$, it is clear that
$\prod_{v \mid \infty} \mu_{v, j}=\lambda_{j}$ and $\mu_{v, j}=1$ if $v \nmid \infty, v \in S$. Let \( \mathbf{x}^{(j)} \in O_{K, S}^{n+1} \) for \( j = 1, \ldots, n+1 \) be linearly independent points, chosen so that their images under the map \( \iota \left( \mathbf{x}^{(j)} \right) \) correspond to the successive minima of the domain \( \Pi(Q) \). Then, for each \( v \in S \), \( i = 0, \ldots, n \), and \( j = 1, \ldots, n+1 \), the following inequality holds:  
\[
\left| L_{v,i} \left( \mathbf{x}^{(j)} \right) \right|_{v} \leq \mu_{v,j} Q^{c_{v,i}},
\]  
.

Applying Evertse's lemma (Lemma \ref{evertse's lemma}) in this context, we deduce the existence of vectors \( \mathbf{v}^{(i)} \in O_{K, S}^{n+1} \) for \( i = 1, \ldots, n+1 \), defined recursively by  
\[
\mathbf{v}^{(1)} = \mathbf{x}^{(1)}, \quad \mathbf{v}^{(i)} = \sum_{j=1}^{i-1} \xi_{ij} \mathbf{x}^{(j)} + \mathbf{x}^{(i)},
\]  
where \( \xi_{ij} \in K \) are coefficients. For each \( v \in S \), there exists a bijection \( \pi_v: \{1, \ldots, n+1\} \to \{0, \ldots, n\} \) that establishes a correspondence between the indices of the vectors \( \mathbf{v}^{(j)} \) and the linear forms \( L_{v, \pi_v(i)} \) at the place \( v \). This bijection enables us to express bounds for the values of these forms evaluated on the vectors \( \mathbf{v}^{(j)} \) as follows:  
\[
\left| L_{v, \pi_v(i)} \left( \mathbf{v}^{(j)} \right) \right|_v \leq  
\begin{cases}  
C_6 \min \left( \mu_{v,i}, \mu_{v,j} \right) Q^{c_{v, \pi_v(i)}} & \text{if } v \mid \infty, \\  
Q^{c_{v, \pi_v(i)}} & \text{if } v \nmid \infty,  
\end{cases}  
\]  
for all \( i, j = 1, \ldots, n+1 \).

We introduce the following notation to handle exterior products of the vectors and forms:
\begin{itemize}
    \item \( \mathbf{i} = (i_1, \ldots, i_{n+1-k}) \) with indices ordered such that \( i_1 < i_2 < \cdots < i_{n+1-k} \), 
    \item \( L_{v, \pi_v(\mathbf{i})} = L_{v, \pi_v(i_1)} \wedge L_{v, \pi_v(i_2)} \wedge \cdots \wedge L_{v, \pi_v(i_{n+1-k})} \),
    \item \( \mathbf{v}^{(\mathbf{i})} = \mathbf{v}^{(i_1)} \wedge \mathbf{v}^{(i_2)} \wedge \cdots \wedge \mathbf{v}^{(i_{n+1-k})} \),
    \item \( c_{v, \pi_v(\mathbf{i})} = c_{v, \pi_v(i_1)} + c_{v, \pi_v(i_2)} + \cdots + c_{v, \pi_v(i_{n+1-k})} \),
    \item \( \lambda_{\mathbf{i}} = \prod_{u=1}^{n+1-k} \lambda_{i_u}, \quad \mu_{v, \mathbf{i}} = \prod_{u=1}^{n+1-k} \mu_{v, i_u}. \)
\end{itemize}

The set of linear forms \( L_{v, \pi_{v}(\mathbf{i})} \) is linearly independent, as are the points \( \mathbf{v}^{(\mathbf{i})} \). Using \eqref{Laplace identity}, it follows directly that for any \( \mathbf{i}, \mathbf{j} \) and any \( v \in S \), we have  
\begin{equation}\label{random eqn 1}
\left|L_{v, \pi_{v}(\mathbf{i})}\left(\mathbf{v}^{(\mathbf{j})}\right)\right|_{v} \leq  
\begin{cases}  
C_7 \mu_{v, \mathbf{i}} Q^{c_{v, \pi_{v}(\mathbf{i})}} & \text{if } v \mid \infty \\  
Q^{c_{v, \pi_{v}(\mathbf{i})}} & \text{if } v \nmid \infty,
\end{cases}
\end{equation}
for some constant $C_7>0$. Additionally, when \( \mathbf{i} = (k+1, k+2, \ldots, n+1) \) but \( \mathbf{j} \neq (k+1, k+2, \ldots, n+1) \), applying Evertse’s lemma allows for a sharper bound:  
\begin{equation}\label{random eqn 2}
\left|L_{v, \pi_{v}(\mathbf{i})}\left(\mathbf{v}^{(\mathbf{j})}\right)\right|_{v} \leq  
C_7 \cdot \left( \frac{\mu_{v, k}}{\mu_{v, k+1}} \right) \cdot \mu_{v, \mathbf{i}} Q^{c_{v, \pi_{v}(\mathbf{i})}}  
\quad \text{if } v \mid \infty  
\end{equation}
since in this case, the index \( j_1 \) satisfies \( j_1 \leq k \) and so
\begin{equation*}
\mu_{v, \mathbf{j}} \leq \mu_{v, k} \mu_{v, k+2} \cdots \mu_{v, n+1} = \left( \frac{\mu_{v, k}}{\mu_{v, k+1}} \right) \mu_{v, \mathbf{i}}.
\end{equation*}

We now proceed to establish the following result:  
\begin{lemma}Let \( \mathcal{S}(Q) \) denote a symmetric convex domain in \( \wedge^{n+1-k}(K_{\mathbb{A}}^{n+1}) \), specified by the conditions  
\[
\begin{aligned}  
\left|L_{v, \pi_{v}(\mathbf{i})}(\mathbf{X})\right|_{v} & \leq C_8 \mu_{v, \mathbf{i}} Q^{c_{v, \pi_{v}(\mathbf{i})}} \quad & & \text{if } v \mid \infty, \mathbf{i} \neq (k+1, \ldots, n+1), \\  
\left|L_{v, \pi_{v}(\mathbf{i})}(\mathbf{X})\right|_{v} & \leq C_8 \left( \frac{\mu_{v, k}}{\mu_{v, k+1}} \right) \mu_{v, \mathbf{i}} Q^{c_{v, \pi_{v}(\mathbf{i})}} & & \text{if } v \mid \infty, \mathbf{i} = (k+1, \ldots, n+1), \\  
\left|L_{v, \pi_{v}(\mathbf{i})}(\mathbf{X})\right|_{v} & \leq Q^{c_{v, \pi_{v}(\mathbf{i})}} & & \text{if } v \in S, v \nmid \infty, \\  
|\mathbf{X}|_{v} & \leq 1 & & \text{if } v \notin S 
\end{aligned}
\]
for some constant $C_8>\op{max}\{C_6, C_7\}$.
Denote by \( \lambda_j(\mathcal{S}(Q)) \) the successive minima of \( \mathcal{S}(Q) \). Then there exists a constant $C_9>0$ for which the following inequalities hold:  
\[
\lambda_j(\mathcal{S}(Q)) \leq 1 \quad \text{for } j < \binom{n+1}{k},  
\]
\[
\lambda_j(\mathcal{S}(Q)) > C_9 Q^{\varepsilon / (2n(n+1))} \quad \text{for } j = \binom{n+1}{k}.
\]
\end{lemma}
\begin{proof}
The inequality \( \lambda_j(\mathcal{S}(Q)) \leq 1 \) for \( j < \binom{n+1}{k} \) is evident from equations \eqref{random eqn 1} and \eqref{random eqn 2}, as they ensure the existence of linearly independent points within \( \mathcal{S}(Q) \). For the second part of the lemma, we invoke Minkowski’s second theorem, which yields  
\begin{equation}\label{ME1}
1 \ll \left( \prod_{j=1}^{\binom{n+1}{k}} \lambda_j(\mathcal{S}(Q)) \right) 
%\beta^{\binom{n+1}{k}}
({\rm vol}(\mathcal{S}(Q)))^{1/r} \ll 1. 
\end{equation}
By Minkowski's second theorem, 
$$
%\beta^{\binom{n+1}{k}}
{\rm vol}(\mathcal{S}(Q))^{1/d} \ll \prod_{v \in S} \left( \frac{\mu_{v, k}}{\mu_{v, k+1}} \prod_{\mathbf{i}} \mu_{v, \mathbf{i}} Q^{c_{v, \pi_{v}(\mathbf{i})}} \right)  
\ll \frac{\lambda_k}{\lambda_{k+1}} \left( \prod_{\mathbf{i}} \lambda_{\mathbf{i}} \right) Q^{\sum_{v \mathbf{i}} c_{v, \pi_{v}(\mathbf{i})}},  
$$
and consequently,
\[
\ll \frac{\lambda_k}{\lambda_{k+1}} \left( \lambda_1 \cdots \lambda_{n+1} (
%\beta^{n+1}
{\rm vol}(\Pi(Q)))^{1/r} \right)^{\binom{n}{k}} \ll \frac{\lambda_k}{\lambda_{k+1}} \ll Q^{-\varepsilon / (2n(n+1))}.
\]
The result then follows by using the left inequality of \eqref{ME1}.
\end{proof}
Let $\bx_i$ be a family of solutions to the subspace inequality and set $Q_i:=H(\bx_i)$ such that $Q_i\rightarrow \infty$ and $\operatorname{rank}\left(\Pi\left(Q_i\right)\right)=R$, where $1 \leq R<n$. We recall that $R<n$ by Theorem \ref{theorem 4.14}. Recall that $k$ is the smallest integer in the range $[R, n]$ for which the ration $\lambda_k/\lambda_{k+1}$ is minimized. Without loss of generality (by going to a subfamily) assume that $k$ is constant for all $\bx_i$. By an argument similar to that of Lemma \ref{affine proj height lemma}, there is a constant $C_{10}>0$ such that
$$
Q_{i}^{-C_{10}} \leq \mu_{v, \mathbf{i}} \leq Q_{i}^{C_{10}}
$$
for every $v \mid \infty$. By passing once again to a suitable subfamily, we may assume that for any small positive number \(\gamma > 0\), there exist bounded real numbers \(d_{v \mathbf{i}}(v \mid \infty)\) such that  
\[  
-\gamma \varepsilon_{v} + d_{v \mathbf{i}} < \frac{\log \left(C_{11} \mu_{v, \mathbf{i}}\right)}{\log \left(Q_i\right)} \leq d_{v \mathbf{i}} \quad \text{if } \mathbf{i} \neq (k+1, \ldots, n+1)  
\]  
and  
\[  
-\gamma \varepsilon_{v} + d_{v \mathbf{i}} < \frac{\log \left(C_{11}\left(\frac{\mu_{v, k}}{\mu_{v, k+1}}\right) \mu_{v, \mathbf{i}}\right)}{\log \left(Q_i\right)} \leq d_{v \mathbf{i}} \quad \text{if } \mathbf{i} = (k+1, \ldots, n+1) 
\]
for some constant $C_{11}>0$. Moreover, if \(v \in S\) and \(v \nmid \infty\), we define \(d_{v \mathbf{i}} = 0\). Let \(\Pi_{k}(Q_i)\) denote the parallelopiped in \(\wedge^{n+1-k}(K_{\mathbb{A}}^{n+1})\) specified by the inequalities  
\[  
\left|L_{v, \pi_{v}(\mathbf{i})}(\mathbf{X})\right|_{v} \leq Q_i^{c_{v, \pi_{v}(\mathbf{i})} + d_{v \mathbf{i}}} \quad \text{for } v \in S, \quad \text{and} \quad |\mathbf{X}|_{v} \leq 1 \quad \text{for } v \notin S.  
\]  
It is then clear that \(\mathcal{S}(Q_i) \subset \Pi_{k}(Q_i)\), and in particular, the inequality  
\[  
\lambda_{j}\left(\mathcal{S}(Q_i)\right) \geq \lambda_{j}\left(\Pi_{k}(Q_i)\right)  
\]  
holds for every \(j\). On the other hand, the Lemma \ref{local vol comps} implies that
$$
%\beta^{\binom{n+1}{k}}
{\rm vol}\left(\Pi_{k}\left(Q_{i}\right)\right)^{1 / d} \ll %\beta^{\binom{n+1}{k}}
{\rm vol}\left(\mathcal{S}\left(Q_{i}\right)\right)^{1 / d} Q_i^{\binom{n+1}{k} \gamma}
$$
Thus, choosing \(\gamma = \varepsilon / \left\{ 4n(n+1)\binom{n+1}{k} \right\}\) as a specific example, and applying Minkowski's second theorem, we deduce that for sufficiently large \(Q_i\), the following inequalities hold:  
\[  
\lambda_{j}\left(\Pi_{k}(Q_i)\right) \leq 1 \quad \text{if } j < \binom{n+1}{k},  
\]  
and  
\[  
\lambda_{j}\left(\Pi_{k}(Q_i)\right) > C_{12} Q_i^{\varepsilon / \{4n(n+1)\}} \quad \text{if } j = \binom{n+1}{k}  
\]  
for some $C_{12} >0$.
It follows that  
\[  
\operatorname{rank}\left(\Pi_{k}(Q_i)\right) = \binom{n+1}{k} - 1  
\]  
once \(Q_i\) is large enough, which we may assume without loss of generality. Furthermore, we have  
\[
\begin{aligned}  
\sum_{v \in S} \sum_{\mathbf{i}} d_{v \mathbf{i}} & \leq \sum_{v \mid \infty} \left( \frac{1}{\log(Q_i)} \left( \log \frac{\mu_{v, k}}{\mu_{v, k+1}} + \sum_{\mathbf{i}} \log \left(C_{11} \mu_{v, \mathbf{i}}\right) \right) + \gamma \varepsilon_{v} \right) \\  
& = \frac{1}{\log(Q_i)} \left( \log \frac{\lambda_{k}}{\lambda_{k+1}} + \log \left( \prod_{\mathbf{i}} \lambda_{\mathbf{i}} \right) + O(1) \right) + \gamma \\  
& = \frac{1}{\log(Q_i)} \left( \log \frac{\lambda_{k}}{\lambda_{k+1}} + \binom{n}{k} \log \left( \lambda_{1} \cdots \lambda_{n+1} \right) + O(1) \right) + \gamma.  
\end{aligned}  
\] Therefore, using \eqref{lambdak/lambdak+1}, we get
$$
\sum_{v \in S} \sum_{\mathbf{i}} d_{v \mathbf{i}} \leq-\frac{\varepsilon}{2(n+1) n}-\binom{n}{k} \sum_{v, i} c_{v, i}+\gamma+O\left(\frac{1}{\log Q_i}\right)
$$

Now we note that
$$
\sum_{\mathbf{i}} \sum_{v \in S}\left(c_{v, \pi_{v}(\mathbf{i})}+d_{v \mathbf{i}}\right)= \sum_{v \in S}\sum_{\mathbf{i}} c_{v, \pi_{v}(\mathbf{i})}+\sum_{v \mathbf{i}} d_{v \mathbf{i}}=\binom{n}{k} \sum_{v,i }c_{v, \pi_{v}(i)}+\sum_{v \mathbf{i}} d_{v \mathbf{i}}=\binom{n}{k} \sum_{v,i } c_{v, i}+\sum_{v \mathbf{i}} d_{v \mathbf{i}}
$$

Thus the above and our choice of $\gamma$ gives
$$
\sum_{\mathbf{i}} \sum_{v \in S}\left(c_{v, \pi_{v}(\mathbf{i})}+d_{v \mathbf{i}}\right) \leq-\frac{\varepsilon}{2(n+1) n}\left(1-\frac{1}{2\binom{n+1}{k}}\right)+O\left(\frac{1}{\log Q_i}\right)
$$
which is negative for $Q_i$ sufficiently large.
Now we are able to apply Theorem \ref{theorem 4.14} to this situation, concluding that the vector spaces $V\left(\Pi_{k}\left(Q_i\right)\right)$ form a finite set, which in turn implies that the associated spaces $W_{k}\left(Q_i\right)$, defined as the $K$-vector space generated by the vectors corresponding to the first $k$ successive minima of $\Pi_{k}\left(Q_i\right)$, also form a finite set. Since $k \geq \operatorname{rank}\left(\Pi\left(Q_i\right)\right)$, we conclude that $\mathbf{x}_{i}$ is a linear combination of the vectors in $K^{n+1}$, denoted by $\mathbf{x}^{(1)}, \ldots, \mathbf{x}^{(k)}$, determining the first $k$ successive minima. By construction, this holds also for $\mathbf{v}^{(1)}, \ldots, \mathbf{v}^{(k)}$ and hence the points $\bx_i$ belong to finitely many proper subspaces of $K^{n+1}$. This completes the proof of the affine subspace theorem \ref{affine subspace theorem}.
    
\bibliographystyle{alpha}
\bibliography{references}
\end{document}